\documentclass{amsart}
\usepackage{a4wide,tikz,tikz-cd,amsbsy,enumerate,hyperref}

\title[Continued fraction algorithms and translations]{Multidimensional continued fractions  and symbolic codings of toral translations}

\author[V.~Berth\'e]{Val\'erie Berth\'e}
\author[W.~Steiner]{Wolfgang Steiner}
\address{Université de Paris, CNRS, IRIF, F-75006 Paris, France}
\email{berthe@irif.fr, steiner@irif.fr}
\author[J. M. Thuswaldner]{J\"org M. Thuswaldner}
\address{Chair of Mathematics and Statistics, University of Leoben, A-8700 Leoben, AUSTRIA}
\email{joerg.thuswaldner@unileoben.ac.at}
\date{\today}

\keywords{Symbolic dynamical systems; symbolic codings; multidimensional continued fraction algorithms; $S$-adic dynamical systems; substitutions; Pisot conjecture;  Lyapunov exponents; Rauzy fractals; bounded remainder sets; toral translations; purely discrete spectrum}

\subjclass[2010]{37B10, 37A30, 11K50, 28A80}

\thanks{This work was supported by the Agence Nationale de la Recherche through the project CODYS (ANR-18-CE40-0007).
The third author was supported by the projects FWF P27050 and FWF P29910 granted by the Austrian Science Fund.}

\newtheorem{lemma}{Lemma}[section]
\newtheorem{theorem}[lemma]{Theorem}
\newtheorem{conjecture}[lemma]{Conjecture}
\newtheorem{proposition}[lemma]{Proposition}
\newtheorem{corollary}[lemma]{Corollary}
\newcommand{\thistheoremname}{}
\newtheorem*{genericthm}{\thistheoremname}
\newenvironment{namedthm}[1]
  {\renewcommand{\thistheoremname}{#1}%
   \begin{genericthm}}
  {\end{genericthm}}
\theoremstyle{remark}

\newtheorem{remark}[lemma]{Remark}

\newtheorem{definition}[lemma]{Definition}

\numberwithin{equation}{section}

\newcommand{\tr}[1]{{}^t #1}
\newcommand{\NN}{\mathbb{N}}
\newcommand{\ZZ}{\mathbb{Z}}
\newcommand{\QQ}{\mathbb{Q}}
\newcommand{\RR}{\mathbb{R}}
\newcommand{\TT}{\mathbb{T}}

\newcommand{\cA}{\mathcal{A}} % the alphabet
\newcommand{\cC}{\mathcal{C}} % collection of tiles
\newcommand{\cD}{\mathcal{D}} % follower set
\newcommand{\cF}{\mathcal{F}} % fundamental domain
\newcommand{\cM}{\mathcal{M}} % set of matrices
\newcommand{\cR}{\mathcal{R}} % Rauzy fractal
\newcommand{\cS}{\mathcal{S}} % set of substitutions
\newcommand{\cL}{\mathcal{L}} % labels
\newcommand{\fD}{\mathfrak{D}} % set of matrix sequences
\newcommand{\sB}{{\scriptscriptstyle\mathrm{B}}} % Brun
\newcommand{\sC}{{\scriptscriptstyle\mathrm{C}}} % Cassaigne
\newcommand{\sAR}{{\scriptscriptstyle\mathrm{AR}}} % Arnoux--Rauzy
\newcommand{\sJP}{{\scriptscriptstyle\mathrm{JP}}} %Jacobi--Perron

\newcommand{\bsigma}{\boldsymbol{\sigma}}
\newcommand{\bphi}{\boldsymbol{\varphi}}
\newcommand{\bpsi}{\boldsymbol{\psi}}
\newcommand{\bseta}{\boldsymbol{\eta}}
\newcommand{\btau}{\boldsymbol{\tau}}

\newcommand{\bl}{\boldsymbol{\ell}}
\newcommand{\bs}{\mathbf{s}}
\newcommand{\bu}{\mathbf{u}}
\newcommand{\bv}{\mathbf{v}}
\newcommand{\bM}{\mathbf{M}}
\newcommand{\bt}{\mathbf{t}}
\newcommand{\bw}{\mathbf{w}}
\newcommand{\bx}{\mathbf{x}}
\newcommand{\be}{\mathbf{e}}
\newcommand{\by}{\mathbf{y}}
\newcommand{\bz}{\mathbf{z}}
\newcommand{\bone}{\mathbf{1}}

\makeatletter
\newcommand{\labitem}[2]{\def\@itemlabel{#1}\item\def\@currentlabel{#1}\label{#2}}
\makeatother

\begin{document} 
\begin{abstract}
It has been a long standing problem to find good symbolic codings for translations on the $d$-dimensional torus that enjoy the beautiful properties of Sturmian sequences like low factor complexity and good local discrepancy properties. Inspired by Rauzy's approach we construct such codings in terms of multidimensional continued fraction algorithms that are realized by sequences of substitutions. In particular, given any exponentially convergent continued fraction algorithm, these sequences lead to renormalization schemes which produce symbolic codings of toral translations and bounded remainder sets at all scales in a natural way. 

The exponential convergence properties of a continued fraction algorithm can be viewed in terms of a Pisot type condition imposed on an attached symbolic dynamical system. Using this fact, our approach provides a systematic way to confirm purely discrete spectrum results for wide classes of symbolic dynamical systems. Indeed, as our examples  illustrate, we are able to confirm the Pisot conjecture for many well-known families of sequences of substitutions. These examples comprise classical algorithms like the Jacobi--Perron, Brun, Cassaigne--Selmer, and Arnoux--Rauzy algorithms. 

As a consequence, we gain symbolic codings of almost all translations of the $2$-dimensional torus having factor complexity $2n+1$ that are balanced for words, which leads to multiscale bounded remainder sets. 
Using the Brun algorithm, we also give symbolic codings of almost all $3$-dimensional toral translations having multiscale bounded remainder sets.
\end{abstract}

\maketitle

%\setcounter{tocdepth}{2}
%\tableofcontents

\section{Introduction}

One of the classical motivations of symbolic dynamics is to provide  representations of  dynamical systems as subshifts made of infinite sequences which code itineraries through suitable choices of partitions. 
In the present paper, we focus on symbolic models for toral translations. 
More precisely, for a given toral translation, we provide symbolic realizations based on multidimensional continued fraction algorithms. 
These realizations have strong dynamical and arithmetic properties. 
In particular, they define bounded remainder sets for toral translations with a natural subdivision structure governed by the underlying continued fraction algorithm.
We recall that bounded remainder sets are defined as sets having bounded local discrepancy. 
In ergodic terms, these are sets for which the Birkhoff sums of their characteristic function have bounded deviations. 
Their study started with the work of W.~M.~Schmidt in his series of papers on  irregularities of distributions (see for instance \cite{Schmidt:74})  and has led to many important contributions; see \cite{Grepstad-Lev} for references.

Our approach is inspired by the seminal example of Sturmian dynamical systems, introduced by M.~Morse and  G.~Hedlund  in \cite{MorseHedlund:40}. 
There is an impressive literature devoted to their study and to possible generalizations in word combinatorics \cite{Fog02}, and also in digital geometry \cite{Rosenfeld}. 
The importance of Sturmian dynamical systems is due to several reasons. 
For instance, they provide symbolic codings for the simplest arithmetic dynamical systems, namely for irrational translations on the circle, they code discrete lines, and  they are one-dimensional models of quasicrystals \cite{BaakeGrimm}.
Besides that, Sturmian dynamical systems are characterized as the minimal shifts having $1$-balanced language over a two-letter alphabet \cite{MorseHedlund:40}. Balance is a classical notion in word combinatorics and symbolic dynamics that has been widely studied from many viewpoints, for instance in ergodic theory and word combinatorics (see e.g.\ \cite{Cassaigne-Ferenczi-Zamboni:00}) and in number theory in connection with Fraenkel's conjecture \cite{Fraenkel:73,Tijdeman:2000}.
The scale invariance of Sturmian dynamical systems allows them to be described by using a renormalization scheme governed by classical continued fractions, which in turn can be interpreted as Poincar\'e sections of the geodesic flow acting on the modular surface. 
This admits important generalizations in the study of interval exchange transformations in relation with the Teichm\"uller flow and renormalization schemes that can often be interpreted as continued fractions \cite{Yoccoz}. 
The basic combinatorial elements for the understanding of Sturmian dynamical systems together with their renormalization scheme are substitutions which are symbolic versions of induction steps (i.e., of first return maps).

In order to get symbolic models, in the present work we rely on substitutive dynamical systems as well as on the more general $S$-adic dynamical systems. 
A~substitution is a rule, either combinatorial or geometric, that replaces a letter by a word, or a tile by a patch of tiles. 
Substitutions are used to define substitutive dynamical systems which play a fundamental role in symbolic dynamics, as emphasized e.g.\ in the monographs \cite{BaakeGrimm,Fog02,Queffelec:10}. 
In particular, Pisot substitutions are of importance in this context since they create  hierarchical structures with a significant amount of long range order \cite{AkiBBLS}. 
Each substitutive dynamical system defined in terms of a Pisot substitution is conjectured to have purely discrete spectrum, that is, to be isomorphic (in the measure-theoretic sense) to a translation on a compact abelian group. 
The fact that this so-called Pisot substitution conjecture is still open (even though it is solved for beta-numeration in~\cite{Barge15}) shows that important parts of the picture are still to be developed.

More generally, $S$-adic dynamical systems are defined in terms of words that are generated by iterating sequences of substitutions, rather than iterating just a single substitution, much the same way like multidimensional continued fraction algorithms in general produce sequences of matrices, and not just powers of a single one. 
A~survey on $S$-adic dynamical systems is provided in \cite{Berthe-Delecroix}. 
The $S$-adic formalism offers representations similar to the Bratteli--Vershik systems related to Markov compacta, and to representations by Rohlin towers as studied for instance in \cite{DHS:99} or \cite[Chapter~6]{CANT}. 
In \cite{BST:19}, we extend the Pisot conjecture to $S$-adic dynamical systems, which enables us to go beyond algebraicity. 
Since $S$-adic dynamical systems are defined in terms of sequences of substitutions, they can be regarded as nonabelian  and combinatorial versions of multidimensional continued fraction algorithms. 
The requirement of Pisot substitutions in the substitutive case is replaced here by a more general condition, called \emph{Pisot condition}, which essentially is an exponential convergence condition imposed on this underlying continued fraction algorithm (see Section~\ref{sec:cf} for precise definitions).
Under this condition, $S$-adic dynamical systems are conjectured to have purely discrete spectrum. 
In \cite{BST:19}, we prove that this extended Pisot conjecture holds for large families of three-letter $S$-adic dynamical systems based on well-known continued fraction algorithms, such as the Brun or the Arnoux--Rauzy algorithm.
As a striking outcome, this yields symbolic codings for almost every translation of the torus~$\TT^2$ \cite{BST:19}, paving the way for the development of equidistribution results for the associated two-dimensional Kronecker sequences.      

In order to apply the results of \cite{BST:19} for a given family of $S$-adic dynamical systems, one has to check quite tedious combinatorial conditions for the involved sequences of substitutions (like the ones checked in \cite{BBJS14} in case of the Brun algorithm; see \cite[Proposition~9.7]{BST:19}). These arguments crucially relied on the topology of the plane and were thus applicable only for three-letter alphabets. This is why the results of \cite{BST:19} are not sufficient for setting up a general theory that is easy to apply for a given family of $S$-adic dynamical systems. 

In the present paper, we circumvent this problem by a new ergodic argument which ensures that the required combinatorial conditions are generically satisfied under mild and natural assumptions. This enables us to formulate results that are easily applicable to any given class of $S$-adic dynamical systems that satisfies the Pisot condition  (see Definition~\ref{d:Pisot}) on any finite alphabet. 
For instance, our new theory works for generalized continued fraction algorithms including the Arnoux--Rauzy algorithm in arbitrary dimension, the Jacobi--Perron algorithm (in dimension~$3$),  the Brun algorithm (in dimension~$4$),  and the Cassaigne--Selmer algorithm in dimension~$3$.
(Only the case of Brun and  Arnoux-Rauzy  algorithms in dimension $3$ were handled in  \cite{BST:19}.)

Another novelty we present in this paper builds on recent results from \cite{BCDLPP:19}.
In particular, we can refine the theory of bounded remainder sets established in \cite{BST:19} in the sense that bounded remainder sets  (for letters) admit natural subdivisions into subsets that form bounded remainder sets (for words).
This results in multiscale natural codings for almost all translations on the torus; see Theorem~\ref{t:nc}, whose informal version is provided in Theorem~B. 
Note that the constructions of bounded remainder sets given in \cite{Grepstad-Lev,HKK:17} do not offer such a scalability.

To each continued fraction algorithm satisfying the Pisot condition, we attach a shift-invariant set of $S$-adic sequences, which generically leads to $S$-adic dynamical systems having purely discrete spectrum. 
This shows that $S$-adic dynamical systems are measurably conjugate to minimal translations on the torus. 
In other words, we provide symbolic representations of toral translations, i.e., symbolic dynamical systems that code toral translations in the measure-theoretic sense, as well as symbolic representations for multidimensional continued fractions.  
In particular, we gain symbolic codings of almost all translations of the $2$-dimensional torus having factor complexity $2n + 1$ that are balanced for words (and not only for letters). Thus they admit bounded remainder sets at all scales; see Corollaries~C and~\ref{cor:CS}. 
Using the Brun algorithm, we also give symbolic codings of almost all 3-dimensional toral translations with bounded remainder sets for all words; see Corollaries~D and~\ref{cor:Br}.

In our results on purely discrete spectrum (see Theorems~\ref{theo:MCF} and~\ref{theo:main}, and Theorem~A for an informal version), we use two main conditions. Firstly, the above-mentioned Pisot condition  (see Definition~\ref{d:Pisot}), 
%i.e., exponential convergence of the underlying continued fraction algorithm 
which is formulated in terms of negativity of the second Lyapunov exponent. Secondly, the existence of a single substitutive dynamical system  that ``behaves well''  and corresponds to a periodic sequence in the set of $S$-adic sequences under consideration. As mentioned above, contrary to the results in \cite{BST:19}, our results on the purely discrete spectrum of $S$-adic dynamical systems do not require combinatorial conditions which are hard to verify.  In fact, some of our results do not need any combinatorial conditions to be verified, see Theorems~\ref{theo:MCF2} and~\ref{theo:main2}.
Indeed, we can prove that each algorithm that satisfies the Pisot condition has an acceleration that leads to toral translations almost surely by using the existence of arbitrarily large blocks of Pisot substitutions in the set of $S$-adic sequences.

\begin{figure}[ht]
\includegraphics[trim= 0cm 4cm 0 5cm, clip,height=5cm]{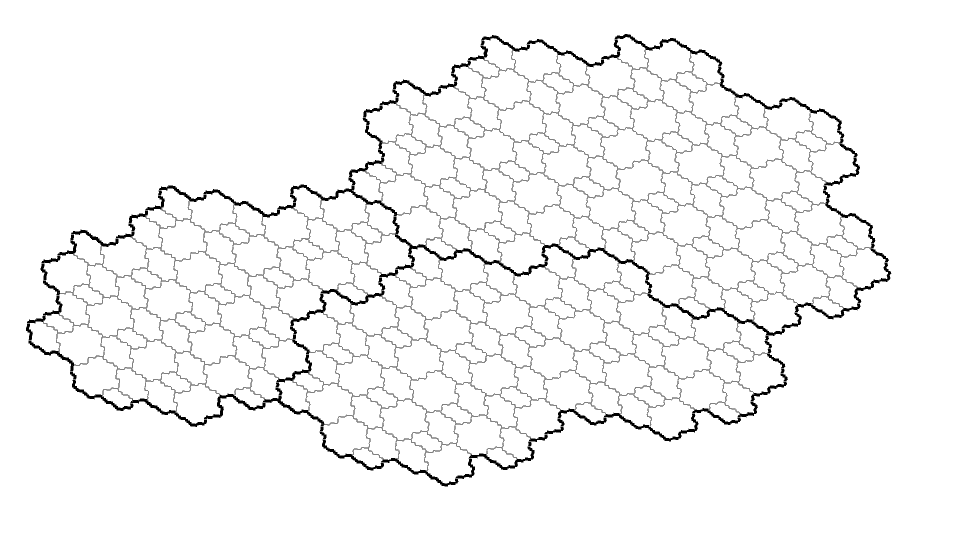}
\caption{An (affine image of an) $S$-adic Rauzy fractal and its subdivision (cf.\ Section~\ref{sec:natur-codings-bound}) whose directive sequence $(\sigma_n)_{n\in\NN}$ starts with $\sigma_0=\cdots=\sigma_7$ and $\sigma_8=\cdots=\sigma_{15}$, where $\sigma_0$, defined by $1\mapsto 13$,  $2\mapsto 12$, $3\mapsto2$, is a Cassaigne--Selmer substitution (see Section~\ref{sec:exSelmer}), and $\sigma_8$ is the classical Tribonacci substitution $1\mapsto 12$,  $2\mapsto 13$, $3\mapsto1$.} \label{fig:selmertribu}
\end{figure}

In our proofs, we also heavily rely on the theory of $S$-adic Rauzy fractals, which has been developed in \cite{BST:19}.  For an illustration of such a Rauzy fractal, see Figure~\ref{fig:selmertribu}. 
Rauzy fractals have been introduced in \cite{Rauzy:82} for the so-called Tribonacci substitution; see also \cite{Thu89}.  
One motivation for Rauzy's construction was to exhibit explicit factors of the substitutive dynamical system as translations on compact abelian groups, under the Pisot hypothesis.
The formalism of $S$-adic Rauzy fractals allows us to verify the Pisot conjecture on sequences of substitutions for wide families of systems satisfying the Pisot condition, thereby extending the results in \cite{BST:19,Fogg:20}; see Theorems~\ref{theo:MCF} and~\ref{theo:main}.  
Already in \cite{BST:19}, for the Brun  algorithm as well as the Arnoux--Rauzy algorithm, purely discrete spectrum results have been shown. 
Parallel to our work, \cite{Fogg:20} proved results on purely discrete spectrum of $S$-adic dynamical systems coming from continued fraction algorithms with special emphasis on the Cassaigne--Selmer algorithm. 
However, the conditions we have to assume in our main results are easy to check effectively and our results (stated in Section~\ref{sec:main-results}) are more general than the ones in \cite{BST:19,Fogg:20}.  
This allows us to treat the Arnoux--Rauzy algorithm in arbitrary dimensions as well as multiplicative continued fraction algorithms like the Jacobi--Perron algorithm (which requires to work  with $S$-adic dynamical systems based on infinitely many substitutions).  

\bigskip

In order to state our results in full mathematical precision, we require several concepts and notation that will be introduced in Section~\ref{sec:land}. Nevertheless, for the convenience of the reader, we provide already here an informal ``prototype'' of our theorems on purely discrete spectrum; for the exact statements, we refer to Theorems~\ref{theo:MCF},  \ref{theo:MCF2}, \ref{theo:main}, and~\ref{theo:main2}.

\begin{namedthm}{Theorem A}
%Let $(\mathcal{X},\nu)$ be a shift of $S$-adic directive sequences that is equipped with a suitable measure. 
If $(D,\nu)$ is an $S$-adic version of a $(d{-}1)$-dimensional continued fraction algorithm satisfying the Pisot condition, then, under mild conditions that are easy to check, the $S$-adic dynamical system $(X_{\bsigma},\Sigma)$ has purely discrete spectrum for $\nu$-almost every $\bsigma\in D$. 
Moreover, $X_{\bsigma}$ is a bounded natural coding of an explicitly given translation on $\TT^{d-1}$.
% whose bounded remainder sets for letters and words are given by $S$-adic Rauzy fractals and their subtiles.  
\end{namedthm}

The next result, which can be considered as a partial  converse of Theorem A, is an informal statement of Theorem~\ref{t:nc}, which shows that $S$-adic Rauzy fractals essentially are the only candidates of bounded remainder sets for $S$-adic dynamical systems.

\begin{namedthm}{Theorem B}
Assume that the $S$-adic dynamical system $(X,\Sigma)$ is the natural coding of a minimal translation~$R_{\bt}$ on~$\TT^{d-1}$ w.r.t.\ a partition $\{\cF_1,\dots,\cF_d\}$ of a bounded fundamental domain of~$\TT^{d-1}$. Then the sets  $\cF_1,\dots,\cF_d$ are affine images of $S$-adic Rauzy fractals. Moreover, they are bounded remainder sets of~$R_{\bt}$ for letters (and, under some properness condition, also for words).  
\end{namedthm}

When we apply these theorems to concrete examples in Section~\ref{sec:ex}, we will see that they have several consequences. We want to mention two of these consequences already here; see Corollaries~\ref{cor:CS} and~\ref{cor:Br}.

\begin{namedthm}{Corollary C}
Almost every rotation (w.r.t.\ Lebesgue measure) of the $2$-torus $\TT^2$ has a natural coding by a subshift of three letters of complexity $2n+1$ that is balanced for words. The associated bounded remainder sets for letters and words are the (bounded) $S$-adic Rauzy fractals corresponding to the Cassaigne--Selmer algorithm. 
\end{namedthm}

\begin{namedthm}{Corollary D}
Almost every rotation (w.r.t.\ Lebesgue measure) of the $3$-torus~$\TT^3$ has a natural coding by a subshift of four letters that is balanced for words. The associated bounded remainder sets for letters and words are the (bounded) $S$-adic Rauzy fractals corresponding to the Brun algorithm. 
\end{namedthm}

As applications for our results, we want to mention the recent paper  \cite{CDFG}, where our present results are used in the framework of Schr\"odinger operators with quasi-periodic multi-frequency potentials based on toral translations. 
In particular, they use our theory to produce Cantor spectra of zero Lebesgue measure for these potentials.
Moreover, we are currently considering higher-dimensional versions of the three-distance theorem in \cite{ABKST} where the  involved shapes are generated by  symbolic and geometric versions of continued fraction algorithms (related again to $S$-adic Rauzy fractals). Note that there   have been recently several  major advances  on   higher-dimensional  distance  theorems, such as \cite{BK:18,HM:20,HM:20bis,HR:20}. We also mention that sequences with good properties of balance are used in operations research, for optimal routing and scheduling (see e.g.\ \cite{AltmanGH:00,BraunerCrama:04,BraunerJost:08}).

More generally, we would like to deduce global discrepancy estimates for multidimensional Kronecker sequences from the local study of bounded remainder sets and thanks to the symbolic codings considered here. This is in the spirit of the one-dimensional results obtained in \cite{Adamczewski:04}. In \cite{sadic3}, we also consider Markov partitions for nonstationary hyperbolic toral automorphisms (as defined in \cite{Arnoux-Fisher:05}) related to continued fraction algorithms. 
We thereby develop symbolic models as nonstationary subshifts of finite type and Markov partitions for sequences of toral automorphisms.
The pieces of the corresponding Markov partitions are fractal sets (and more precisely $S$-adic Rauzy fractals) defined by associating substitutions to (incidence) matrices, or in terms of Bratteli diagrams, obtained by constructing suspensions via two-sided Markov compacta \cite{Bufetov:14b}.     

In the present paper, we are dealing exclusively with results that hold for almost all parameters (with respect to a given measure). However, similarly to the examples on the Arnoux--Rauzy algorithm in \cite[Theorem~3.8 and Corollary~3.9]{BST:19}, it is possible to produce concrete families of $S$-adic dynamical systems having purely discrete spectrum (characterized e.g.\ by properties of their partial quotients or by recurrence properties) for other continued fraction algorithms as well. According to \cite[Theorem~3.1]{BST:19}, their study involves the investigation of combinatorial properties of the underlying sequences of substitutions. Other explicit examples are   provided  by  $S$-adic systems related to a constant sequence given by the repetition of a single Pisot substitution. We end up with a substitutive dynamical system for such examples and for this class of parameters, there exist many algorithms for checking purely discrete spectrum; see e.g.~\cite{CANTBST} or Section~\ref{subsec:ba}
 below. Besides that,  given any Pisot matrix,  we show how to construct Pisot substitutions giving rise to substitutive dynamical systems with purely discrete spectrum for large enough powers of  this   Pisot matrix; we refer to Section~\ref{sec:proof-theor-refth},  and in particular, to Proposition~\ref{p:sigmatilde}.

\subsection*{Outline of the paper} 
After recalling basic notation and definitions in Section~\ref{sec:land}, Section~\ref{sec:main-results} is devoted to the precise statement of our main results on purely discrete spectrum including their consequences on natural codings of translations and bounded remainder sets. 
The concepts needed in the proofs of our results are provided in Section~\ref{sec:defprep}. 
In particular, we recall the required background on Rauzy fractals. 
These proofs are then given in Section~\ref{sec:proofs}. 
Section~\ref{sec:ex} is devoted to the detailed discussion of some examples which provide codings of a.e.\ translation on~$\TT^2$ and~$\TT^3$ that lead to bounded remainder sets of all scales.

\section{Mise en sc\`ene}\label{sec:land}
\subsection{Multidimensional continued fraction algorithms}\label{sec:cf} There are several formalisms for defining  multidimensional continued fractions,
 see e.g.\ \cite{AL18,BRENTJES,BAG:01,KLM:07,Lagarias:93,Lagarias:94,SCHWEIGER}.
In the present paper, a \emph{$(d{-}1)$-dimensional continued fraction algorithm} $(\Delta,T,A)$ is defined on a set
\[
\Delta \subseteq \{\bx \in [0,1]^d \,:\, \|\bx\|_1 = 1\}
\]
by a map 
\[
A:\, \Delta \to \mathrm{GL}(d,\ZZ)
\]
satisfying $\frac{\tr{\!A}(\bx)^{-1} \bx}{\|\tr{\!}A(\bx)^{-1} \bx\|_1}\in \Delta$ for all $\bx \in \Delta$, together with the associated  transformation
\begin{equation}\label{eq:Tdef}
T:\, \Delta \to \Delta, \quad \bx \mapsto \frac{\tr{\!A}(\bx)^{-1} \bx}{\|\tr{\!}A(\bx)^{-1} \bx\|_1}.
\end{equation}
Here $\tr{\!M}$ denotes the transpose of a matrix~$M$. The map $A$ is usually piecewise constant which entails that $T$ is piecewise continuous. These algorithms are called \emph{linear simplex-splitting} in \cite[Section~2]{Lagarias:93}, and their iteration produces \emph{convergent matrices} used for simultaneous Diophantine approximation. The matrices $\tr{\!}A(\bx)$  are called \emph{partial quotient matrices}. 
This class of algorithms contains prominent examples like the classical algorithms of Brun~\cite{Brun19,Brun20,BRUN}, Jacobi--Perron \cite{Bernstein:71,Heine1868,Perron:07,Schweiger:73}, and Selmer~\cite{Selmer:61}, which are discussed in Section~\ref{sec:ex}. When we refer to these  well-known continued fraction algorithms we will often informally talk about the {\em classical continued fraction algorithms}. 

In the present paper the transition from the linear homogeneous version of the algorithm given by the piecewise linear map $\bx\mapsto \tr{\!A}(\bx/\|\bx\|_1)^{-1} \bx$ to its projectivized version \eqref{eq:Tdef} is performed by a normalization by the $1$-norm. 
This choice allows working with a symmetric version of the algorithm, as e.g.\ in \cite{AL18}.

A~multidimensional continued fraction algorithm $(\Delta,T,A)$ is called \emph{positive} if $A(\bx)$ is a nonnegative matrix for all $\bx \in \Delta$, i.e., if $A(\Delta)$ is contained in 
\[
\cM_d = \{M \in \NN^{d\times d} \,:\, |\!\det M| = 1\},
\]
with $\NN = \{0,1,2,\ldots\}$. 
It is \emph{additive} if  the set of produced  matrices  $A(\Delta)$  is finite, \emph{multiplicative} otherwise.
Setting 
\[
A^{(n)}(\bx) = A(T^{n-1}\bx) \cdots \,A(T\bx)\,A(\bx),
\]
$A$~is a \emph{linear cocycle} for~$T$, i.e., it fulfills the cocycle property $A^{(m+n)}(\bx) = A^{(m)}(T^n \bx) A^{(n)}(\bx)$; this is the reason for defining $T$ by the transpose of~$A$. 

The column vectors $\by^{(n)}_i$, $1 \le i \le d$, of the convergent matrices  $\tr{\!}A^{(n)}(\bx)$ produce $d$ sequences of \emph{rational convergents} $(\by^{(n)}_i/\|\by^{(n)}_i\|_1)_{n\in\NN}$ that are supposed to converge to~$\bx$. More precisely, 
\begin{itemize}
\item
$T$ converges \emph{weakly} at $\bx\in\Delta$  if $\lim_{n\to\infty} \by^{(n)}_i/\|\by^{(n)}_i\|_1 = \bx$ holds for all $i\in\{1,\ldots,d\}$;
\item
$T$ converges \emph{strongly} at $\bx\in\Delta$ if $\lim_{n\to\infty}   \|\by^{(n)}_i  -  \|\by^{(n)}_i\|_1 \,\bx\| =0$ holds\footnote{We indicate which norm we use only if the choice of the norm is relevant. Here, $\|\cdot\|$ can be any norm in~$\RR^d$.}
for all $i\in\{1,\ldots,d\}$;
\item
$T$ converges \emph{exponentially} at $\bx\in\Delta$ if there are positive constants $\kappa,\delta\in\RR$ such that 
$\|\by^{(n)}_i  -  \|\by^{(n)}_i\|_1 \,\bx\| < \kappa e^{-\delta n}$ holds for all $i\in\{1,\ldots,d\}$ and all $n\in\NN$.
\end{itemize}
An important role is played by the following condition, which is essentially equivalent to almost everywhere exponential convergence of the algorithm. 
%see \cite[Equation (4.21)]{Lagarias:93}

\begin{definition}[Pisot condition, cf.\ \cite{Berthe-Delecroix,BST:19}] \label{d:Pisot}
Let $(X,T,\nu)$ be a dynamical system with ergodic invariant probability measure~$\nu$, and let $C: X\to \cM_d$ be a log-integrable linear cocycle for~$T$; here log-integrable means that $\int_X \log \max(1,\|C(x)\|)\, \mathrm{d}\nu(x) < \infty$.
Then the Lyapunov exponents $\vartheta_k(C)$ of~$C$ exist and are given for $k \in \{1,\ldots,d\}$ by
($\wedge^k$ denotes the $k$-fold exterior product)
\[
\vartheta_1(C) + \cdots + \vartheta_k(C) = \lim_{n\to\infty} \frac{1}{n} \log \|\wedge^k C(T^{n-1} x) \cdots \,C(T x)\,C(x)\| \quad \mbox{for $\nu$-almost all $x \in X$}.
\]
We say that $(X,T,C,\nu)$ satisfies the \emph{Pisot condition} if
$\vartheta_1(C) > 0 > \vartheta_2(C)$. 
\end{definition}

We always assume that the continued fraction algorithm $(\Delta,T,A)$ is endowed with an ergodic $T$-invariant probability measure~$\nu$ such that the map~$A$ is $\nu$-measurable; here $\mathrm{GL}(d,\ZZ)$ carries the discrete topology. 
Then the Pisot condition together with the Oseledets theorem (see e.g.~\cite[Theorem~3.4.1]{Arnold:98}) implies that there is a constant $\delta < 0$ such that, for $\nu$-almost all $\bx \in \Delta$, there is a hyperplane $V$ of $\RR^d$ with
\[
\lim_{n\to\infty} \frac{1}{n} \log \|A^{(n)}(\bx)\, \bv\| \le \delta \quad \mbox{for all}\ \bv \in V.
\]
According to Lagarias~\cite[Theorem~4.1]{Lagarias:93} the Pisot condition is equivalent to a.e.\  exponential convergence of $(\Delta,T,A)$ under some natural conditions called (H1) -- (H5) that are introduced in \cite[Section~4]{Lagarias:93}. These conditions are true in many cases; see e.g.\ \cite{berth2021second}. In the present paper, we will only rely on the Pisot condition; the relation between the Pisot condition and exponential convergence will not be used. Thus we do not go into details.

\subsection{Substitutive and $S$-adic dynamical systems, shifts of directive sequences}
Substitutions will be very important objects in our constructions. 
Let $\cA=\{1,2,\ldots,d\}$ be a finite ordered alphabet and let $\sigma: \cA^* \to \cA^*$ be an endomorphism of the free monoid $\cA^*$ of words over~$\cA$, which is equipped with the operation of concatenation. 
If $\sigma$ is \emph{nonerasing}, i.e., if $\sigma$ does not map a nonempty word to the empty word, then we call $\sigma$ a \emph{substitution} over the alphabet~$\cA$. A word~$w$ is a \emph{factor} of a word~$v$ if there exist words $p,s$ such that $v = pws$.
Moreover, if $p$ is the empty word, then $w$ is a \emph{prefix} of~$v$,  which will often be denoted by $w\preceq v$; we write $w\prec v$ when $w\preceq v$ and $w \ne v$.
On the space $\cA^\NN$ of one-sided infinite sequences over $\cA$ (equipped with the product topology of the discrete topology on $\cA$), the notions of \emph{factor} and \emph{prefix} are defined in a similar  way. With the substitution~$\sigma$ we associate the language
\[
L_\sigma= \big\{w\in \cA^* \,:\, \mbox{$w$ is a factor of $\sigma^n(i)$ for some $i\in\cA$, $n\in \NN$}\big\},
\]
i.e.,  $L_\sigma$~is the set of words that occur as subwords in iterations of~$\sigma$ on a letter of~$\cA$. 
Using the language~$L_\sigma$, the \emph{substitutive dynamical system} $(X_\sigma,\Sigma)$ is defined by
\[
X_\sigma=\{\omega \in \cA^\NN \,:\, \mbox{each factor of $\omega$ is contained in $L_\sigma$}\},
\]
with $\Sigma$ being the \emph{shift map} $(\omega_n)_{n\in\NN} \mapsto (\omega_{n+1})_{n\in\NN}$;\footnote{We denote the shift map on any 
space of sequences by~$\Sigma$; this should not cause any confusion.} $X_\sigma$ is obviously $\Sigma$-invariant. 

 The abelianized counterpart of a substitution $\sigma$ is its so-called \emph{incidence matrix} 
\[
M_\sigma=(|\sigma(j)|_i)_{1\le i,j\le d},
\]
where $|w|_i$ denotes the number of occurrences of a letter $i\in \cA$ in the word $w\in\cA^*$. The \emph{abelianization} of a word $w\in\cA^*$ is $\bl(w) = \tr{}(|w|_1,\ldots,|w|_d)$, so that $\bl(\sigma(w)) = M_\sigma \bl(w)$. 

Many properties of a substitution depend on its incidence matrix.
Indeed, while $M_\sigma$ ``forgets'' the combinatorics of $\sigma$, it encodes letter frequencies and convergence properties of the sequences of~$X_\sigma$. 
So-called \emph{unimodular Pisot substitutions}, which are characterized in terms of incidence matrices, have received particular interest: A~unimodular Pisot substitution is a substitution~$\sigma$ whose incidence matrix~$M_\sigma$ has a characteristic polynomial which is the minimal polynomial of a Pisot unit. Recall that a Pisot unit is an algebraic integer greater than~$1$ whose norm equals $\pm 1$ and whose Galois conjugates are all contained in the open unit disk. 
For example, if $\sigma$ is unimodular Pisot, then we can infer that the elements of $X_\sigma$ are balanced in the sense defined in Section~\ref{subsec:prop}; see e.g.\ \cite[Theorem~1]{Adamdis}. Moreover, a unimodular Pisot substitution $\sigma$ is \emph{primitive} in the sense that its incidence matrix admits a positive power. This implies that the associated symbolic dynamical system $(X_\sigma,\Sigma)$ is \emph{minimal} (i.e., $X_\sigma$ has no nontrivial closed shift-invariant subset); see e.g.\ \cite{Queffelec:10}. 
%In this case a certain geometric realization of $\sigma$ by a broken line (visualized for instance in Figure~\ref{fig:constrRauzy}) stays at bounded distance from the Pisot eigendirection of $M_\sigma$. 
Throughout this paper we will assume that the incidence matrix of a substitution $\sigma$ is \emph{unimodular}, i.e., we consider the set of substitutions
\[
\cS_d = \{\sigma \,:\, \mbox{$\sigma$ is a substitution over $\cA=\{1,\ldots,d\}$},\, M_\sigma \in \cM_d\}.
\]
When we discuss sequences $(\sigma_n)_{n\in\NN}$ of unimodular substitutions later, considering the linear cocycle $(\sigma_n)_{n\in\NN} \mapsto \tr{\!M_{\sigma_0}}$ will enable us to study the convergence behavior of $(\sigma_n)_{n\in\NN}$. Here the Pisot condition (see Definition~\ref{d:Pisot}), which is also a condition on incidence matrices in this setting, will be of particular importance for us.

Substitutive dynamical systems (and related tiling flows) have been studied extensively in the literature with special emphasis on unimodular Pisot substitutions; see for instance \cite{baake2019fourier,BufSol18,Fog02,Queffelec:10}. The main conjecture in this context, the so-called \emph{Pisot substitution conjecture}, claims that, for each unimodular Pisot substitution~$\sigma$, the substitutive dynamical system $(X_\sigma,\Sigma)$ is measurably conjugate to a minimal translation on the torus~$\TT^{d-1}$, and, hence, has purely discrete spectrum. Although there are many partial results (see e.g.\ \cite{AkiBBLS,Barge:16,Barge15,HolSol:03,mercat2018characterization}), this conjecture is still open. However, given a single unimodular Pisot substitution~$\sigma$, there are many algorithms that can be used to verify that $(X_\sigma,\Sigma)$ has purely discrete spectrum; see \cite{AL11,CANTBST,mercat2018characterization,SS:02}. 
Thus, for each single unimodular Pisot substitution~$\sigma$, this property is easy to check, which is important for us.

To be more precise, in the present paper, unimodular Pisot substitutions are of importance because of their relation to  multidimensional continued fraction algorithms that satisfy the Pisot condition. Indeed, we show that wide classes of  symbolic dynamical systems of Pisot type  are measurably conjugate to minimal translations on the torus, provided that the same is true for a particular Pisot unimodular substitutive element of the class; see Theorem~\ref{theo:main}.

The concept of $S$-adic dynamical system constitutes a generalization of substitutive dynamical systems; see for instance \cite{Arnoux-Mizutani-Sellami,sadic3,Berthe-Delecroix,BST:19,thuswaldner2019boldsymbolsadic}, where $S$-adic dynamical systems are studied in a similar context as in the present paper. 
An $S$-adic dynamical system is defined in terms of a sequence $\bsigma=(\sigma_n)_{n\in \NN}$ of substitutions over a given alphabet~$\cA$ in a way that is analogous to the definition of a substitutive dynamical system. In particular, let 
\[
L_{\bsigma}= \big\{w \in \cA^* \,:\, \mbox{$w$ is a factor of $\sigma_{[0,n)}(i)$ for some $i \in \cA$,  $n \in \NN$}\big\},
\]
be the language associated with~$\bsigma$, with
\[
\sigma_{[k,n)} = \sigma_k \circ \sigma_{k+1} \circ \cdots \circ \sigma_{n-1} \qquad (0\le k\le n).
\]
Then the \emph{$S$-adic dynamical system} $(X_{\bsigma},\Sigma)$ is defined by setting
\[
X_{\bsigma}=\{\omega \in \cA^\NN \,:\, \mbox{each factor of $\omega$ is contained in $L_{\bsigma}$}\}.
\]
The sequence~$\bsigma$ is called a \emph{directive sequence} of $(X_{\bsigma},\Sigma)$.
Note that the $S$-adic dynamical system of a periodic directive sequence $(\sigma_0,\ldots,\sigma_{n-1})^\infty$ is equal to the substitutive dynamical system $(X_{\sigma_{[0,n)}},\Sigma)$. 

We say that a directive sequence~$\bsigma$ has \emph{purely discrete spectrum} if the system $(X_{\bsigma},\Sigma)$ is \emph{uniquely ergodic} (i.e., it has a unique shift-invariant measure~$\mu$), minimal, and has purely discrete measure-theoretic spectrum (i.e., the measurable eigenfunctions of the \emph{Koopman operator} $U_T: L^2(X_{\bsigma},\Sigma,\mu) \to L^2(X_{\bsigma},\Sigma,\mu)$, $f \mapsto f \circ \Sigma$, span $L^2(X_{\bsigma},\Sigma,\mu))$.

There is a tight link between $S$-adic dynamical systems and continued fraction algorithms. 
For the classical continued fraction algorithm, this is worked out in detail in \cite{Arnoux-Fisher:01,Arnoux-Fisher:05}; for multidimensional continued fractions algorithms, see for instance~\cite{BST:19,thuswaldner2019boldsymbolsadic}. 
Indeed, for each given vector, a~continued fraction algorithm creates a sequence of partial quotient matrices.
If these matrices are nonnegative and integral (i.e., if the algorithm is positive), they can be regarded as incidence matrices of a  directive sequence of substitutions of an $S$-adic dynamical system.
In fact, a continued fraction algorithm produces a whole shift of sequences of matrices, depending on the vector that has to be approximated.
The matrices are taken from a (finite or infinite) set~$\cM$ depending on the algorithm. 
While for some algorithms, all sequences in~$\cM^\NN$ occur as sequences of partial quotient matrices (as is the case for instance for the Brun and Selmer algorithms), other algorithms (like the Jacobi--Perron algorithm) impose some restrictions on these \emph{admissible sequences}, which are usually given by a finite type condition. 
As a further illustration, in the formalization of multidimensional continued fraction algorithms as Rauzy induction type algorithms developed in \cite{CN:13,Fougeron:20},
inspired by interval exchanges, finite graphs allow  to formalize admissibility conditions. 
Here, we do not need to restrict ourselves to such finite type admissibility conditions and we work with shift-invariant sets of directive sequences such as formalized below.
We will come back to the notion of admissibility in Section~\ref{sec:further-definitions}.
 
Assume throughout the paper that the space~$\cS_d^{\NN}$ of sequences over the substitutions~$\cS_d$ carries the product topology of the discrete topology on~$\cS$.  Let $D \subset \cS_d^{\NN}$ be a \emph{shift-invariant set of directive sequences} (which is not to be confused with the $S$-adic shift $(X_{\bsigma},\Sigma)$ of a single directive sequence $\bsigma \in D$); note that we do not require $D$ to be closed.
We define the \emph{linear cocycle}~$Z$ over $(D,\Sigma)$~by
\[
Z:\, D \to \cM_d, \qquad (\sigma_n)_{n\in\NN} \mapsto \tr{\!M}_{\sigma_0};
\] 
recall that $M_\sigma$ is the incidence matrix of~$\sigma$.
Analogously to the linear cocycle~$A$, we  define
\begin{equation}\label{eq:codef}
Z^{(n)}(\bsigma) = Z(\Sigma^{n-1} \bsigma) \cdots Z(\Sigma \bsigma) Z(\bsigma),
\end{equation}
so that $Z^{(n)}(\bsigma) = \tr{\!}M_{\sigma_{n-1}} \cdots \tr{\!}M_{\sigma_1} \tr{\!}M_{\sigma_0} = \tr{\!}M_{\sigma_{[0,n)}}$.
As mentioned before, this cocycle will be important in order to study convergence properties of the $S$-adic dynamical system $(X_{\bsigma},\Sigma)$. Indeed, we have under mild conditions (see Section~\ref{subsec:prop}) that
\begin{equation}\label{eq:u}
\bigcap_{n\in\NN} M_{\sigma_0} M_{\sigma_1} \cdots M_{\sigma_{n-1}} \RR^d_+ = \RR_+\bu 
\end{equation}
for some vector $\bu \in \RR_+^d$, which  is called a \emph{generalized right eigenvector} of~$\bsigma$ (or of $(M_{\sigma_n})_{n\in\NN}$) and can be seen as the generalization of the Perron--Frobenius eigenvector of a primitive matrix. 
%When normalized by $\|\bu\|_1=1$,  the vector $\bu$  is called the \emph{normalized  generalized right eigenvector} of~$\bsigma$ (or of $(M_{\sigma_n})_{n\in\NN}$).
Moreover, we wish to carry over the property of a substitution being Pisot in the substitutive case to this more general setting. 
This will be done by imposing the Pisot condition in Definition~\ref{d:Pisot} on the Lyapunov exponents of the cocycle $(D,\Sigma,Z,\nu)$ for a 
convenient $\Sigma$-invariant Borel measure~$\nu$.
Thus we do not consider a single sequence~$\bsigma$ but the behavior of $\nu$-almost all sequences in~$D$.

Finally, recall that in general a~\emph{shift} (or equivalently, a~\emph{symbolic dynamical system}) is a closed and shift-invariant set~$Y$ of sequences $\omega \in \cA^\NN$ over some alphabet~$\cA$. 
The \emph{language} $L$ of~$Y$ is the set of all factors of the sequences in~$Y$.
The \emph{factor complexity} of  $L$ (or of $Y$) is given by
\begin{equation}\label{eq:factorcx}
p_L:\,\NN \to \NN, \quad n \mapsto \#\{v \in L \,:\, v\ \mbox{has length}\ n\}.
\end{equation}

\subsection{$S$-adic shifts given by continued fraction algorithms}\label{sec:s-adic-shifts}
Our goal is to set up symbolic realizations of positive continued fraction algorithms, which in turn will provide symbolic models of toral translations, in a way that is described in Section~\ref{sec:natur-codings-bound} below. 
To this end, for a given multidimensional continued fraction algorithm $(\Delta,T,A)$, we associate with each $\bx \in \Delta$ a sequence of substitutions $\bsigma = (\sigma_n)_{n\in\NN} \in \cS_d^{\NN}$ with generalized right eigenvector~$\bx$.
In particular, given $\bx \in \Delta$, we regard the partial quotient matrices $\tr{\!}A(T^n \bx)$ as incidence matrices of substitutions, i.e., for each $n \in \NN$ we choose $\sigma_n$ with incidence matrix $M_{\sigma_n}=\tr{\!}A(T^n \bx)$. 
This obviously implies that $M_{\sigma_{[0,n)}} = \tr{\!}A^{(n)}(\bx)$.

\begin{definition}[$S$-adic realization] \label{d:realization}
We call a map $\varphi: \Delta \to \cS_d$ a \emph{substitution selection} for a positive $(d{-}1)$-dimensional continued fraction algorithm $(\Delta,T,A)$ if the incidence matrix of $\varphi(\bx)$ is equal to $\tr{\!A}(\bx)$ for all $\bx \in \Delta$.  The corresponding \emph{substitutive realization} of $(\Delta,T,A)$ is the map 
\[
\bphi:\, \Delta \to \cS_d^{\NN}, \quad \bx \mapsto (\varphi(T^n\bx))_{n\in\NN},
\]
together with the shift $(\bphi(\Delta),\Sigma)$.
For any $\bx \in \Delta$, the sequence $\bphi(\bx)$ is called an \emph{$S$-adic expansion} of~$\bx$, and $(X_{\bphi(\bx)}, \Sigma)$ is called the \emph{$S$-adic dynamical system of~$\bx$ w.r.t.\ $(\Delta,T,A,\varphi)$}.

If $\varphi(\bx) = \varphi(\by)$ for all $\bx, \by \in \Delta$ with $A(\bx) = A(\by)$, then $\varphi$ is called a \emph{faithful substitution selection} and $\bphi$ is a \emph{faithful substitutive realization}. 
\end{definition}

Note that the diagram 
\begin{equation}\label{eq:diagphis}
\begin{tikzcd}
\Delta \arrow[r, "T"]\arrow[d,"\bphi"] & \Delta \arrow[d, "\bphi"] \\
\bphi(\Delta) \arrow[r, "\Sigma"]& \bphi(\Delta) 
\end{tikzcd}
\end{equation}
commutes.
If $T$ converges weakly at~$\bx$ for $\nu$-almost all $\bx \in \Delta$ (w.r.t.\ a measure $\nu$ having the properties determined in Section~\ref{sec:cf}), then the dynamical system $(\Delta,T,\nu)$ is measure-theoretically isomorphic to its substitutive realization, which we write as 
\begin{equation} \label{eq:phi}
(\Delta,T,\nu) \overset{\bphi}{\cong} (\bphi(\Delta),\Sigma,\nu\circ\bphi^{-1}).
\end{equation}

\bigskip
The following definition will play a crucial role in the sequel.  
A~\emph{Pisot matrix} is an integer matrix with characteristic polynomial equal to the minimal polynomial of a Pisot number, and a \emph{Pisot substitution} is a substitution whose incidence matrix is a Pisot matrix.\footnote{We stress the fact that in this paper we mainly work with \emph{unimodular} Pisot substitutions and matrices.}

\begin{definition}[Pisot sequence and point] \label{d:Pisotpoint}
A~sequence $(M_n) \in \cM_d^{\NN}$ [$(\sigma_n) \in \cS_d^{\NN}$] is called a \emph{periodic Pisot sequence} if there is an $k \ge 1$ such that the sequence has period~$k$ and $M_0 M_1 \cdots M_{k-1}$ is a Pisot matrix [$\sigma_0 \circ \sigma_1 \circ \cdots \circ \sigma_{k-1}$ is a Pisot substitution].

For a multidimensional continued fraction algorithm $(\Delta,T,A,\nu)$, 
we say that $\bx_0 \in \Delta$ is a \emph{periodic Pisot point} if there is an $k \ge 1$ such that $T^k(\bx_0)=\bx_0$ and $A^{(k)}(\bx_0)$ is a Pisot matrix. 
\end{definition}

We also  need  to recall the notion of properness. 
A~substitution $\sigma$ over~$\cA$ is \emph{left [right] proper} if there exists $j \in \cA$ such that $\sigma(i)$ starts [ends] with $j$ for all $i \in \cA$.
A~sequence of substitutions $\bsigma = (\sigma_n)$ is \emph{left [right] proper} if for each $k\in \NN$ there exists $n > k$ such that $\sigma_{[k,n)}$ is left [right] proper.
It is \emph{proper} if it is both left and right proper.\footnote{We mention that, in previous papers, a sequence of substitutions $(\sigma_n)$ is called proper if each substitution $\sigma_n$ is proper, see for instance~\cite{Durand:00b,BCDLPP:19}. For our purposes, the weaker definition stated before is sufficient, i.e., for each $k\in \NN$ there exists $n > k$ such that $\sigma_{[k,n)}$ is  proper.  Via telescoping,  the definition used  in the present paper amounts  to the definition which requires   each substitution $\sigma_n$ to be  proper.}
Properness is a natural assumption introduced in \cite{DHS:99} in order to relate Bratteli--Vershik systems associated with stationary, properly ordered Bratteli  diagrams with substitutive dynamical systems. In the present paper, we will use \cite[Corollary~5.5]{BCDLPP:19} which states that if a primitive unimodular proper $S$-adic shift $(X_{\bsigma},\Sigma)$ is balanced for letters, then it is also balanced for words (see Sections~\ref{subsec:prop} and~\ref{subsec:balancewords} for definitions).  Telescoping   a directive sequence $(\sigma_n)$  means  the following  (this is also called blocking):  we consider a directive sequence  of the form $(\sigma_{[k_n, k_{n+1})} )$ for some strictly increasing sequence $(k_n)$.  Directive sequences are not assumed to take finitely many values in 
\cite{BCDLPP:19} hence, up to telescoping, we can use \cite[Corollary~5.5]{BCDLPP:19} with the present definition of properness.

\subsection{Natural codings, bounded remainder sets, and Rauzy fractals} \label{sec:natur-codings-bound}
In this section, we introduce some terminology related to symbolic codings of toral translations with respect to finite partitions; see~\cite{Chevallier} for more details
and also \cite[Section III]{AndrieuPhD}.
For $\bt \in \RR^d$, we consider the translation 
\[
R_\bt:\, \TT^d \to \TT^d, \quad \bx \mapsto \bx + \bt \pmod{\ZZ^d}
\]
on $\TT^d = \RR^d/\ZZ^d$. 
We assume that $\bt=(t_1, \dots, t_d)$ is \emph{totally irrational} in the sense that $1, t_1, \dots, t_d$ are rationally independent. 
This implies that $R_\bt$ is minimal and uniquely ergodic. 

We want to provide symbolic codings of~$R_\bt$ with respect to a given finite partition. 
There are many possible codings, and the simplest partitions, using polytopes for example, do not give the best results in terms of  multiscale bounded remainder sets.   
We rather consider partitions of a fundamental domain of~$\TT^d$ which are chosen in a way that on each atom the map~$R_\bt$ is a translation by a vector. 
This induces an exchange of domains on this fundamental domain and leads to the notion of natural partition and natural coding, which we describe now. 

\begin{definition}[Natural partition]\label{def:naturalpartition}
A~\emph{measurable fundamental domain} of $\TT^d$ is a set $\cF \subset \RR^d$ with Lebesgue measure~$1$ that satisfies $\cF + \ZZ^d = \RR^d$.
A~collection $\{\cF_1,\ldots, \cF_h\}$ is said to be a \emph{natural partition}\footnote{This is a partition up to  zero measure sets.} of~$\cF$ with respect to~$R_\bt$ if  
\begin{itemize}
\item
$\bigcup_{i=1}^h \cF_i = \cF$;
\item
the (Lebesgue) measure of $\cF_i \cap \cF_j$ is zero for all $i \ne j$, $1\le i,j \le h$;
\item
each set $\cF_i$, $1\le i \le h$, is the closure of its interior and has boundary of measure zero;
\item
there exist vectors $\bt_1, \dots, \bt_h$ in~$\RR^d$ such that $\bt_i+ \cF_i \subset \cF$ with  $\bt_i \equiv \bt \pmod{\ZZ^d}$, $1 \le i \le h$. 
\end{itemize}
A~natural partition is called \emph{bounded} if the set $\cF$ is bounded. 
\end{definition}

A~natural partition $\{\cF_1,\ldots, \cF_h\}$ of a measurable fundamental domain~$\cF$ of~$\TT^d$ allows to define a.e.\ on~$\cF$ a map $\tilde{R}_\bt:\, \cF \to \cF$ as an \emph{exchange of domains} (which depends on the partition) by $\tilde{R}_{\bt}(\bx)= \bx + \bt_i$ whenever $\bx \in \mathring{\cF_i}$.  
The map~$\tilde{R}_\bt$ is defined on $\cF \setminus \bigcup _{i=1}^h \partial \cF_i $, hence, it is defined almost everywhere. 
The  dynamical system $(\cF,\tilde{R}_{\bt},\lambda|_\cF)$, where $\lambda$ denotes the Lebesgue measure, is 
%minimal, uniquely  ergodic, and 
measurably isomorphic to  $(\TT^d, R_\bt)$ (endowed with the Haar measure). 
One has for a.e.\ $\bx \in \cF$, $ \tilde{R}_\bt(\bx) \equiv  R_\bt(\bx)  \pmod{\ZZ^d}$.  
The collection $\{\cF_1+\bt_1, \dots, \cF_h+\bt_h\}$ also forms a measurable natural partition of~$\cF$, hence the terminology exchange of domains; see Figure~\ref{fig:exchRauzy} below for an illustration. 
The \emph{language} associated with the partition $\{\cF_1, \dots,\cF_h\}$ is the set of words $i_0 \cdots i_n \in \{1, \dots,h\}^*$ such that $\bigcap _{k=0}^n  \tilde{R}_\bt^{-k} \mathring{\cF}_{i_k} \neq \emptyset$. 

\begin{definition}[Natural coding]\label{def:NC}
A~shift $(X,\Sigma)$ is a \emph{natural coding} of $(\TT^d,R_\bt)$ if  its language is the language of a natural partition $\{\cF_1,\dots,\cF_h\}$ and $\bigcap_{n\in\NN} \overline{\bigcap_{k=0}^n \tilde{R}_\bt^{-k} \mathring{\cF}_{i_k}}$ is reduced to one point for any $(i_n)_{n\in\NN} \in X$, where $\tilde{R}_\bt$ stands for the  associated  exchange of domains.\footnote{This intersection on $\cF$ is meaningful because $\{\cF_1,\dots,\cF_h\}$ is a natural partition of $\cF$; see the fourth bullet point of Definition~\ref{def:naturalpartition}.}

A~sequence $(i_n)_{n\in\NN} \in \{1,\dots,h\}^\NN$ is said to be a \emph{natural coding} of $(\TT^d,R_\bt)$ w.r.t.\ the natural partition $\{\cF_1,\dots,\cF_h\}$ if there exists $\bx \in \cF$ such that $(i_n)_{n\in\NN}$ codes the orbit of~$\bx$ under the action of~$\tilde{R}_\bt$, i.e., $\tilde{R}_\bt ^n(\bx) = \bx + \sum_{k=0}^{n-1} \bt_{i_k} \in \cF_{i_n}$ for all $n \in \NN$; note that $R_\bt ^n(\bx)  \equiv \tilde{R}_\bt ^n(\bx) \pmod{\ZZ^d}$. 
\end{definition}

If $(X,\Sigma)$ is a natural coding of $(\TT^d,R_\bt)$ w.r.t.\ a natural partition $\{\cF_1,\dots,\cF_h\}$, whose elements $\cF_1, \dots, \cF_h$ are bounded, we call $(X,\Sigma)$ a {\em bounded natural coding}. The shift $(X,\Sigma)$ is minimal, uniquely ergodic, and has purely discrete spectrum according to Lemma~\ref{lem:ncChev}.

%In this case, by \cite[Theorem A]{Chevallier} and Remark~\ref{rem:che}, one can define a continuous \deleted[id=J]{surjective} map $\chi: X \to \cF$ \added[id=J]{that is surjective up to a set of measure zero}.
%Moreover, by the same result there exists a one-to-one coding map $\Phi$ defined a.e.\ on~$\cF$ that satisfies $\chi \circ \Phi(\bx) = \bx$ for a.e.~$\bx$ and that associates with $\bx$ the natural coding of its orbit under the action~$\tilde{R}_\bt$ w.r.t.\ the partition $\{\cF_1,\dots,\cF_h\}$. 
%Furthermore, by \cite[Theorems A and~B]{Chevallier} and Remark~\ref{rem:che} the subshift  $(X, \Sigma)$ is minimal and uniquely ergodic \deleted[id=J]{$(\TT^d,R_\bt)$ is a topological factor of $(X, \Sigma)$,}  and  $(\TT^d,R_\bt)$ is measure-theoretically isomorphic to $(X, \Sigma)$. This implies that $(X, \Sigma)$ has purely discrete measure-theoretic spectrum.

We give an example for the concepts defined above. Consider the translation~$R_{\alpha}$ on~$\TT^1$ with $\alpha \in \RR \setminus \QQ$. The partition $\{\cF_1,\cF_2\}$ of $\cF=[0,1)$ given by $\cF_1=[0,1{-}\alpha)$ and $\cF_2=[1{-}\alpha,1)$ is a bounded natural partition (which corresponds to a Sturmian dynamical system \cite{MorseHedlund:40}) because $R_\alpha(x) = x+\alpha$ for $x \in  \cF_1$ and $R_\alpha(x) = x+\alpha-1$ for  $x \in \cF_2$. 
The bounded natural coding of a point $x \in \TT^1$ is the (Sturmian) sequence $(i_n)_{n\in\NN}$ given by $R_\alpha^n(x) \in \cF_{i_n}$, $n \in \NN$. 
On the contrary, the partition of $[0,1)$ by the intervals $[0,\frac{1}{2})$ and $[\frac{1}{2},1)$ is not a natural partition for~$R_\alpha$. Indeed, since we have no integers $k_1, k_2$ such that both $[\alpha + k_1, \alpha + k_1+\frac{1}{2}) \subset [0,1)$ and $[\alpha + k_2+\frac{1}{2}, \alpha + k_2+1) \subset [0,1)$, the fourth bullet point in Definition~\ref{def:naturalpartition} is not fulfilled.

\begin{definition}[Bounded remainder set]\label{def:BRS}
A~\emph{bounded remainder set} of a dynamical system $(X,T,\mu)$ with invariant probability measure~$\mu$ is a measurable set $Y \subseteq X$ such that there exists $C > 0$ with the property
\[
\big|\# \{0 \le n<N \,:\, T^n(x) \in Y\} - N \mu(Y)\big| \le C \quad \mbox{for all $N\in \NN$ and a.e. $x \in X$}.
\] 
\end{definition}

Bounded natural codings and bounded remainder sets are closely related; see for instance \cite{Rauzy:84,Ferenczi92} and Theorem~\ref{t:nc} below. 
We will define bounded natural partitions using \emph{Rauzy fractals}. 
To define Rauzy fractals, we denote by 
\begin{equation}\label{def:piu}
\pi_{\bu}:\, \RR^d \to \bone^\perp \quad \mbox{the projection along $\bu$ on $ \bone^\perp$},
\end{equation}
where $\bone^\perp$ is the hyperplane orthogonal to $\bone = (1,1,\dots,1)$.

\begin{figure}[ht]
\includegraphics[trim=0 60 0 60,width = 0.5\textwidth]{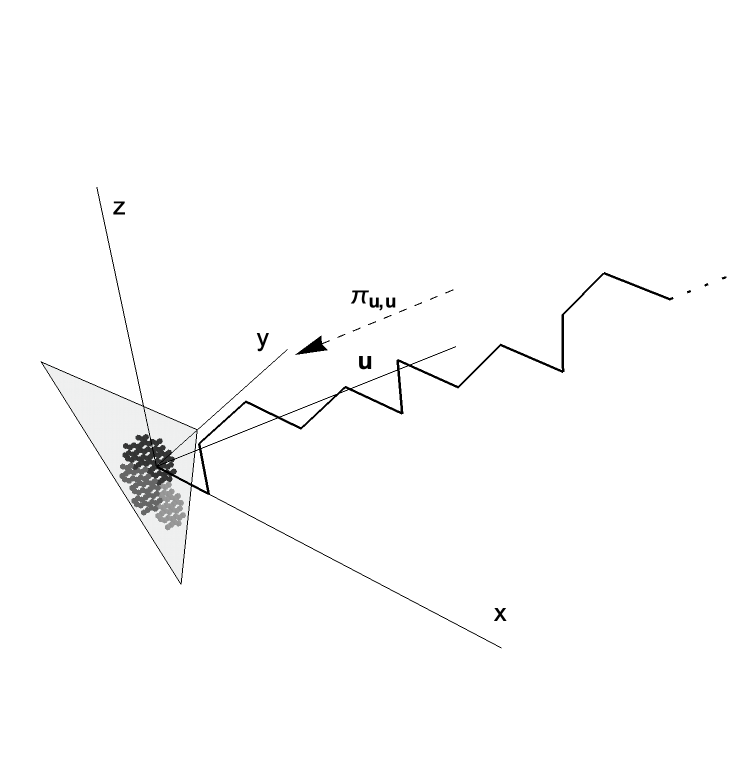} \qquad
\includegraphics[trim=0 0 0 0, width = 0.25\textwidth]{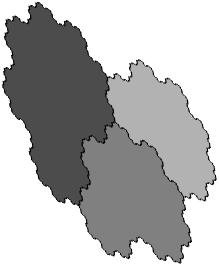}
\caption{Illustration of the definition of the Rauzy fractal $\cR_{\bsigma}$ corresponding to the periodic directive sequence $\bsigma = (\gamma_1,\gamma_2)^\infty$, where $\gamma_1,\gamma_2$ are the Cassaigne--Selmer substitutions defined in \eqref{eq:cassaigneSubs}. The abelianizations $\bl(p)$ of the prefixes of $(\gamma_1\circ \gamma_2)^n(1)$ define a broken line and are projected along~$\bu$ to~$\bone^\perp$ in order to define the Rauzy fractal~$\cR_{\bsigma}$, where $\bu$ is a generalized right eigenvector of~$\bsigma$. The subtiles $\cR_{\bsigma}(1)$, $\cR_{\bsigma}(2)$, and $\cR_{\bsigma}(3)$ are indicated by different shades of grey.} \label{fig:constrRauzy}
\end{figure}

\begin{definition}[Rauzy fractal and subtile] \label{def:RF}
Let $(X_{\bsigma},\Sigma)$ be an $S$-adic dynamical system with $\bsigma \in \cS_d^{\NN}$  having the generalized right eigenvector~$\bu$.
The \emph{Rauzy fractal} associated with $\bsigma =  (\sigma_n)_{n\in\NN}$ is defined as 
\[
\cR_{\bsigma} = \overline{\{\pi_{\bu}\, \bl(p) \,:\, \mbox{$p\preceq \sigma_{[0,n)}(j)$ for infinitely many  $n\in\NN$, $j \in \cA$}\}},
\]
and, for each word $w \in \cA^*$, a~\emph{subtile} of~$\cR_{\bsigma}$ is defined by 
\begin{equation}\label{eq:RFactor}
\cR_{\bsigma}(w) = \overline{\{\pi_{\bu}\, \bl(p) \,:\, \mbox{$p\,w \preceq \sigma_{[0,n)}(j)$ for infinitely many $n \in \NN$, $j \in \cA$}\}}.
\end{equation}
\end{definition} 

We clearly have 
\[
\cR_{\bsigma} = \bigcup_{w\in\cA^n} \cR_{\bsigma}(w) \qquad (n \in \NN),  
\]
and in particular $\cR_{\bsigma} = \bigcup_{i\in\cA} \cR_{\bsigma}(i)$. 
In Figure~\ref{fig:constrRauzy}, we illustrate the definition of the Rauzy fractal for the periodic directive sequence $\bsigma=(\gamma_1,\gamma_2)^\infty$, with $\gamma_1,\gamma_2$ being the Cassaigne--Selmer substitutions defined in \eqref{eq:cassaigneSubs} below. 
Rauzy fractals associated with periodic sequences~$\bsigma$ (and therefore related to substitutive dynamical systems) go back to \cite{Rauzy:82} and have been studied extensively; see for instance \cite{Arnoux-Ito:01,Berthe-Siegel:05,CANTBST,CS:01,Fog02,Ito-Rao:06,ST09,thuswaldner2019boldsymbolsadic}. 
Our definition of~$\cR_{\bsigma}$ is equivalent to the one in \cite[Section~2.9]{BST:19}, which uses limit sequences of~$\bsigma$, i.e., infinite sequences that are images of $\sigma_{[0,n)}$ for all $n \in \NN$.

\begin{figure}[ht]
\includegraphics[width=.8\textwidth]{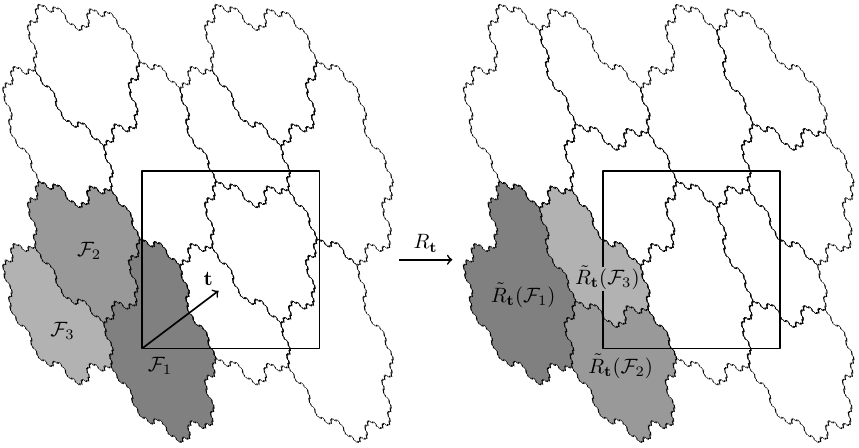}
% \includegraphics[width=0.2\textwidth]{RauzyBefore.pdf}
% \put(10,50){$\xrightarrow{\hskip 2cm}$}
% \hskip3cm
% \includegraphics[width=0.2\textwidth]{RauzyAfter.pdf}
\caption{Let $\cR_{\bsigma} = \bigcup_{i\in\cA} \cR_{\bsigma}(i)$ be the Rauzy fractal associated with the directive sequence $\bsigma = (\gamma_1,\gamma_2)^\infty$; see \eqref{eq:cassaigneSubs} for the definition of the Cassaigne--Selmer substitutions $\gamma_1$ and~$\gamma_2$. The negative projection $-\cR'_{\bsigma}$ of this Rauzy fractal is a measurable fundamental domain of~$\TT^2$ (i.e., its translates by vectors in~$\ZZ^2$ tile~$\RR^2$)  admitting the natural partition $\{\cF_1,\cF_2,\cF_3\}=\{-\cR'_{\bsigma}(1),-\cR'_{\bsigma}(2),-\cR'_{\bsigma}(3)\}$ w.r.t.\ $R_{\bt}$, where $\bt = (1/\beta^3,1/\beta^4)$ with $\beta^3=\beta+1$.
The exchange of domains~$\tilde{R}_{\bt}$ is defined by $\tilde{R}_{\bt}(\bx) = \bx + \bt_i$ on~$\cF_i$ with $\bt_1 = \bt - (1,0)$, $\bt_2= \bt - (0,1)$, $\bt_3 = \bt$.} \label{fig:exchRauzy}
\end{figure}

For convenience, we define a further ``projection''  that will provide translations on~$\TT^{d-1}$ in the main results given in Section~\ref{sec:main-results}. 
We set
\begin{equation} \label{e:defpi}
\pi':\, \RR^d \to \RR^{d-1}, \quad (x_1,\ldots,x_d) \mapsto  (x_1,\dots,x_{d-1}),
\end{equation}
i.e., we omit the last coordinate of a vector. 
(In doing so, we make an arbitrary choice; it would also be possible to omit any other coordinate.) 
Sometimes, we will just write $\bx'$ instead of~$\pi'(\bx)$.  
Similarly, for the subtiles embedded in $\RR^{d-1}$ via~$\pi'$, we will write
\begin{equation}\label{eq:Rprime}
\cR'_{\bsigma}(w) = \pi' (\cR_{\bsigma}(w)) \qquad (w\in \cA^*).
\end{equation}
Figure~\ref{fig:exchRauzy}  illustrates how subtiles of the projection of a Rauzy fractal $\cR_{\bsigma}$ give rise to a natural partition and visualizes the domain exchange~$\tilde{R}_{\bt}$. In this figure, we use again the Rauzy fractal for the periodic directive sequence $\bsigma=(\gamma_1,\gamma_2)^\infty$, with $\gamma_1,\gamma_2$ as in \eqref{eq:cassaigneSubs} below.

\subsection{Cylinders  and  positive range}\label{sec:further-definitions}
To state our theorems, we need a few more definitions on partitions associated with continued fraction algorithms.

\begin{definition}[Cylinder and follower set, positive range] \label{def:pr}
Let $(D,\Sigma,\nu)$ be a dynamical system with $D \subset \cS_d^{\NN}$ and a shift invariant Borel measure~$\nu$. 
The \emph{cylinder set} of $(\omega_0,\dots,\omega_{n-1}) \in \cS_d^n$ is defined as 
 \[
[\omega_0,\dots,\omega_{n-1}] = \big\{(\upsilon_k)_{k\in\NN} \in D\,:\, (\upsilon_0,\dots,\upsilon_{n-1}) = (\omega_0,\dots,\omega_{n-1})\big\},
\] 
and $\Sigma^n[\omega_0,\dots,\omega_{n-1}]$ is the \emph{follower set} of $(\omega_0,\dots,\omega_{n-1})$.
Moreover, we say that $(\omega_n)_{n\in\NN}$ has \emph{positive  range} in $(D,\Sigma,\nu)$ if
\[
\inf_{n\in\NN} \nu(\Sigma^n[\omega_0,\dots,\omega_{n-1}]) > 0.
\]
 
Similarly, the \emph{cylinder sets} of a multidimensional continued fraction algorithm $(\Delta,T,A,\nu)$ are given by
\begin{equation} \label{eq:Deltan}
\Delta^{(n)}(\bx) = \{\by \in \Delta \,:\, A(\by) = A(\bx), A(T\by) = A(T \bx), \dots, A(T^{n-1}\by) = A(T^{n-1} \bx)\},
\end{equation}
with $\Delta^{(0)}(\bx) = \Delta$; for convenience, we set $\Delta(\bx) = \Delta^{(1)}(\bx)$. 
In this context, the \emph{follower sets} are the sets of the form $T^n \Delta^{(n)}(\bx)$.
Then $\bx \in \Delta$ is said  to have \emph{positive range} in $(\Delta,T,A,\nu)$ if 
\[
\inf_{n\in\NN} \nu(T^n \Delta^{(n)}(\bx)) > 0.
\]
\end{definition}

Cylinder sets of $(D,\Sigma,\nu)$ are measurable because all cylinders are open sets in the subspace topology on~$D$. 
This is the reason why we assumed~$\nu$ to be a Borel measure. 
We also recall that $D$ is not necessarily closed. 
Note that measurability of the cylinder sets of $(\Delta,T,A,\nu)$ holds because $A$ is measurable by assumption.

In all the classical algorithms we are aware of, almost every $\bx \in \Delta$ has positive range, and we even have the (global) \emph{finite range property} (cf.~\cite{ItoYuri:87}) stating that the set of follower sets
\[
\cD = \{T^n \Delta^{(n)}(\bx) \,:\, \bx \in \Delta,\, n \in \NN\}
\]
is finite, where sets differing only on a set of $\nu$-measure zero are identified.
For instance, although the Jacobi--Perron algorithm is multiplicative, $\cD$ consists of only two elements; see also Section~\ref{subsec:JP}.
By the $T$-invariance of $\nu$, the finite range property obviously implies positive range for a.e.\ $\bx \in \Delta$ if we suppose that all cylinders satisfy $\nu(\Delta^{(n)}(\bx))>0$; this will be the case for the algorithms considered in Section~\ref{sec:ex}.

If $(\Delta,T,A,\nu)$ has the finite range property and $\bigcap_{n\in\NN}  \Delta^{(n)}(\bx) = \{\bx\}$ for almost all $\bx \in \Delta$, i.e., the set of cylinders $\{\Delta(\bx) : \bx \in \Delta\}$  is a generating partition, then $\{U \cap \Delta(\bx) : U \in \cD,\, \bx \in \Delta\}$ forms a (measurable countable)  {generating Markov partition} of $(\Delta,T)$; see e.g.\  \cite[Theorem~10.1]{Yuri:95}.  
Most of the classical continued fraction algorithms (like Brun, Selmer, and Jacobi--Perron) are designed in a way that this Markov partition property holds. 

We need that any set $B \subset \Delta$ with $\nu(B) > 0$ included in the follower set $T^n \Delta^{(n)}(\bx)$ leads to an intersection $T^{-n} B \cap \Delta^{(n)}(\bx)$ with positive measure. 
To this end, we always assume the stronger property that  
\begin{equation} \label{e:abscontinuity}
\nu(E) = 0 \quad \Longrightarrow \quad \nu \circ T(E)=0 \qquad \mbox{for all measurable sets $E$}.
\end{equation}
Although $\nu \circ T$ is usually  not additive and therefore not a measure, we use the notation $\nu\circ T \ll \nu$ because \eqref{e:abscontinuity} is reminiscent of absolute continuity. 

The notation $\nu\circ \Sigma \ll \nu$ has the analogous meaning in the context of a shift $(D,\Sigma,\nu)$.

\section{Main results}\label{sec:main-results}
We  present two types of results: the first type is stated in the framework 
of multidimensional continued fraction algorithms in Section~\ref{sec:mainMCF}, the second one is stated in terms of $S$-adic dynamical systems and directive sequences in Section~\ref{sec:mainSadic}. 
For both frameworks, two theorems are given.
The first one requires the existence of a single substitutive dynamical system with purely discrete spectrum which corresponds to a periodic sequence in the set of  $S$-adic sequences under consideration. 
The existence of this single system already implies purely discrete spectrum for a whole shift of $S$-adic dynamical systems. 
It is stated in Theorem~\ref{theo:MCF} for multidimensional continued fraction algorithms and in Theorem~\ref{theo:main} for shifts of directive sequences. 
The second one yields unconditional purely discrete spectrum results for accelerations and is contained in Theorem~\ref{theo:MCF2} for multidimensional continued fraction algorithms and in Theorem~\ref{theo:main2} for shifts of directive sequences. 
All these results are then made more explicit in terms of bounded remainder sets with Theorem~\ref{t:nc}. 

\subsection{Main results on multidimensional continued fraction algorithms} \label{sec:mainMCF}
In this section, we provide our main results for multidimensional continued fraction algorithms. 
We recall that we use the abbreviation $\bx' = \pi'(\bx)$ for the map~$\pi'$ defined in (\ref{e:defpi}).
In particular, following \eqref{eq:Rprime}, we wite $\cR'_{\bsigma}(i) = \pi'(\cR_{\bsigma}(i))$. 
The notation~$\ll$ is defined at the end of Section~\ref{sec:further-definitions}. 

\begin{theorem} \label{theo:MCF}
Let $(\Delta,T,A,\nu)$ be a positive $(d{-}1)$-dimensional continued fraction algorithm satisfying the Pisot condition and $\nu \circ T \ll \nu$. 
Let $\bphi$ be a faithful substitutive realization of $(\Delta,T,A,\nu)$. 
Assume that there is a periodic Pisot point $\bx_0 \in \Delta$ with positive range in  $(\Delta,T,A,\nu)$ such that $\bphi(\bx_0)$ has purely discrete spectrum.
Then, for $\nu$-almost all $\bx \in \Delta$, the $S$-adic dynamical system $(X_{\bphi(\bx)},\Sigma)$ is a bounded natural coding of the minimal translation by $\pi'(\bx)$ on~$\TT^{d-1}$ w.r.t.\ the partition $\{-\cR'_{\bphi(\bx)}(i) \,:\, i\in \cA\}$; in particular, its measure-theoretic spectrum is purely discrete. 
\end{theorem}

It will follow from Theorem~\ref{t:nc} that the sets $-\cR'_{\bphi(\bx)}(i)$, $i\in \cA$, are bounded remainder sets.
If the directive sequence  $\bphi(\bx)$ is assumed to be (left) proper (as defined in Section \ref{sec:s-adic-shifts}), Theorem~\ref{t:nc} shows that we can even refine these bounded remainder sets from letters to words. 
In particular, in this case the Rauzy fractals $-\cR'_{\bphi(\bx)}(w)$, $w\in \cA^n$, associated with words of length~$n$ are bounded remainder sets for each $n \in \NN$.

\begin{remark}\mbox{}  \label{r:faithful}
\begin{itemize}
\item[(i)]
We note that $(X_{\bphi(\bx_0)},\Sigma)$ is a substitutive dynamical system since $\bphi(\bx_0)$ is a periodic sequence of substitutions. 
For such systems, some combinatorial \emph{coincidence conditions} (as for instance the ones used in \cite{AkiBBLS,Barge-Kwapisz:06,CANTBST,Ito-Rao:06}) can be used to establish purely discrete measure-theoretic spectrum; see Section~\ref{sec:Rauzy} for precise statements. 
We could therefore replace the purely discrete spectrum condition in Theorem~\ref{theo:MCF} by ``$\varphi(\bx_0) \circ \varphi(T\bx_0) \circ \dots \circ \varphi(T^{n-1}\bx_0)$ satisfies the \emph{super coincidence condition} from \cite[Definition~4.2]{Ito-Rao:06}''. 
However, since coincidence conditions require quite some notation, we decided to formulate them later in this paper in order to make our main results easier to read.
The Pisot substitution conjecture implies that all Pisot substitutions satisfy the super coincidence condition. 
To get an impression of the techniques used in the substitutive case for proving purely discrete spectrum, see also Section~\ref{sec:ex}, where we use the balanced pair algorithm to prove purely discrete spectrum of a substitutive dynamical system.
\item[(ii)]
In Theorem~\ref{theo:MCF}, we can omit the requirement that $\bphi$ is faithful if we replace $A$ by $\varphi$ in the definition of the cylinder sets $\Delta^{(n)}(\bx)$ in \eqref{eq:Deltan}, if we assume that $\varphi$ is measurable,  and if we assume positive range with respect to  this new definition of cylinder.
\end{itemize}
\end{remark}

Since the Pisot substitution conjecture is not proved, we cannot omit the requirement of a periodic Pisot point with purely discrete spectrum in Theorem~\ref{theo:MCF}, and we do not even know whether there always exists a substitutive realization~$\bphi$ that admits such a point.
However, we are able to establish the following unconditional theorem that guarantees the existence of accelerations $(\Delta,T^k)$ for which there exists a faithful substitutive realization~$\bphi$ with a periodic Pisot point $\bx_0$ such that $\bphi(\bx_0)$ has purely discrete spectrum. 

\begin{theorem} \label{theo:MCF2}
Let $(\Delta,T,A,\nu)$ be a positive $(d{-}1)$-dimensional continued fraction algorithm satisfying the Pisot condition and $\nu \circ T \ll \nu$, and assume that there exists a periodic Pisot point with positive range. 
Then there exist a positive integer~$k$ and a (faithful) substitutive realization~$\bphi$ of $(\Delta,T^k,A,\nu)$ such that for $\nu$-almost all $\bx \in \Delta$ the $S$-adic dynamical system $(X_{\bphi(\bx)},\Sigma)$ is a bounded natural coding of the minimal translation by $\pi'(\bx)$ on~$\TT^{d-1}$ w.r.t.\ the partition $\{-\cR'_{\bphi(\bx)}(i) \,:\, i \in \cA\}$; in particular, its measure-theoretic spectrum is purely discrete.
Moreover, we have $(\Delta,T^k,\nu) \overset{\bphi}{\cong} (\bphi(\Delta),\Sigma,\nu\circ\bphi^{-1})$.
\end{theorem}

\begin{remark}\label{rem:allRotations}
The set of translations in Theorems~\ref{theo:MCF} and~\ref{theo:MCF2} does not cover~$\TT^{d-1}$ since the translations are of the form $R_{\bt}$ with $\bt \in [0,1]^{d-1}$ and $\|\bt\|_1 \le 1$. 
However, $R_{\bt}$ is conjugate to all translations $R_{\bs}$ with $\bs \in \mathrm{GL}(d{-}1,\ZZ)\, \bt$, and $\{\bt \in [0,1]^{d-1} \,:\, \|\bt\|_1 \le 1\}$ is mapped by
\[
(t_1,\ldots,t_{d-1}) \mapsto (t_1,t_1+t_2,\dots,t_1+t_2+\cdots+t_{d-1})
\]
to $\{\bt \in [0,1]^{d-1} \,:\, 0 \le t_1 \le t_2 \le \cdots \le t_{d-1} \le 1\}$.
Then, taking permutations of the coordinates of the latter set gives the whole torus~$\TT^{d-1}$.
\end{remark}

Verifying purely discrete spectrum for some concrete substitutive dynamical systems will allow us to use Theorem~\ref{theo:MCF}  in Section~\ref{sec:ex} in order to prove a.e.\ purely discrete spectrum for many continued fraction algorithms like for instance the Jacobi--Perron, Brun, Cassaigne--Selmer and Arnoux--Rauzy--Poincar\'e algorithms. 
Indeed, it is well known that these algorithms have the finite range property, and the Pisot condition holds for all these algorithms when $d=3$. 
In the case of Brun, the Pisot condition also holds for $d=4$.
Applying Theorem~\ref{theo:MCF} to these algorithms, according to Remark~\ref{rem:allRotations} we are able to realize almost all translations in~$\TT^2$ and~$\TT^3$ via systems of the form $(X_{\bphi(\bx)},\Sigma)$, $\bx \in \Delta$.
Since  the Cassaigne--Selmer algorithm (for $d=3$) gives rise to languages~$L_{\bphi(\bx)}$ of factor complexity $2n+1$, we also show that there exist natural codings for almost all translations of~$\TT^2$ with factor complexity $2n+1$, see Corollary~\ref{cor:CS}. 
%Recall that the \emph{factor complexity} $p(n)$ of a language counts the number of elements of a given length~$n$ contained in this language. 
Looking at \cite{BCDLPP:19,BST:19}, we also see other consequences for these algorithms and their associated shifts of directive sequences like bounded remainder sets for letters and words, tiling properties of Rauzy fractals, and a description of their dimension group. We will come back to these consequences in Theorem~\ref{t:nc}, in Section~\ref{sec:Rauzy}, and in Section~\ref{sec:ex}.

\subsection{Main results on shifts  of directive sequences} \label{sec:mainSadic}
We now give variants of the results of the previous section in terms of directive sequences. 

\begin{theorem} \label{theo:main}
Let $D \subset \cS_d^{\NN}$ be a shift-invariant set of directive sequences equipped with an ergodic $\Sigma$-invariant Borel probability measure~$\nu$ satisfying $\nu \circ \Sigma \ll \nu$. 
Assume that the linear cocycle $(D,\Sigma,Z,\nu)$ defined by $Z((\sigma_n)_{n\in\NN}) = \tr{\!}M_{\sigma_0}$ satisfies the Pisot condition, and that there is a periodic Pisot sequence in $D$ having positive range in $(D,\Sigma,\nu)$ and purely discrete spectrum. 
Then for $\nu$-almost all $\bsigma\in D$ the $S$-adic dynamical system $(X_{\bsigma},\Sigma)$ is a bounded natural coding of the minimal translation by $\pi'(\bu)$ on~$\TT^{d-1}$  w.r.t.\ the partition $\{-\cR'_{\bsigma}(i) \,:\, i \in \cA\}$. Here, $\bu$ is the   generalized right eigenvector of~$\bsigma$ normalized by $\|\bu\|_1=1$. 
In particular, the measure-theoretic spectrum of $(X_{\bsigma},\Sigma)$ is purely discrete.  
\end{theorem}
  
To get an analogue of Theorem~\ref{theo:MCF2} for directive sequences, we do not start with a shift of directive sequences but rather with its abelianization, i.e., a shift of sequences of matrices $(\fD,\Sigma)$, for which we would like to find a map $s: \cM_d \to \cS_d$ such that almost all $\bsigma \in s(\fD)$ have purely discrete spectrum, where $s((M_n)_{n\in\NN}) = (s(M_n))_{n\in\NN}$.
Again, we have to consider the accelerated shift $(\fD,\Sigma^k)$ for a suitable power $\Sigma^k$ to gain such a result. 
The main issue is the construction of a substitution with purely discrete spectrum associated with a given unimodular Pisot matrix, which is done in Proposition~\ref{p:sigmatilde}. 

\begin{theorem} \label{theo:main2}
Let $\fD \subset \cM_d^{\NN}$ be a shift-invariant set of sequences of unimodular matrices equipped with an ergodic $\Sigma$-invariant Borel probability measure~$\nu$ satisfying $\nu \circ \Sigma \ll \nu$.
Assume that the linear cocycle $(\fD,\Sigma,Z,\nu)$ defined by $Z((M_n)_{n\in\NN}) = \tr{\!}M_0$ satisfies the Pisot condition, and that there is a periodic Pisot sequence in $\fD$ having positive range in $(\fD,\Sigma,\nu)$. 
Then there exists a positive integer~$k$ and a map $\bpsi: \fD \to \cS_d^{\NN}$ satisfying $\bpsi \circ \Sigma^k = \Sigma \circ \bpsi$ such that for $\nu$-almost all $\bM \in \fD$ the $S$-adic dynamical system $(X_{\bpsi(\bM)},\Sigma)$ is a bounded natural coding of the minimal translation by $\pi'(\bu)$ on~$\TT^{d-1}$  w.r.t.\ the partition $\{-\cR'_{\bpsi(\bM)}(i) \,:\, i \in \cA\}$. 
Here, $\bu$ is the generalized right eigenvector of~$\bM$ normalized by $\|\bu\|_1=1$. 
In particular, the measure-theoretic spectrum of $(X_{\bpsi(\bM)},\Sigma)$ is purely discrete.  
\end{theorem}

\begin{remark}\label{rem:theo2}
Let $\bM = (M_n)$ and $\bpsi(\bM) = (\sigma_n)$. 
According to \eqref{eq:bpsidef}, the map $\bpsi$ in Theorem~\ref{theo:main2} can be chosen in a way that $M_{nk}\cdots M_{(n+1)k-1}$ is the incidence matrix of~$\sigma_n$. 
This choice is needed to derive Theorem~\ref{theo:MCF2} from Theorem~\ref{theo:main2}.
\end{remark} 
 
The main difference between the results in Section~\ref{sec:mainMCF} and the ones in Section~\ref{sec:mainSadic} is that in the latter case there can be several directive sequences in~$D$ with the same generalized right eigenvector (normalized w.r.t.\ $\|\cdot\|_1$).  

\subsection{Main results on natural codings and  bounded remainder sets} 
We now prove that natural codings with respect to bounded  fundamental domains (see Definition~\ref{def:NC}) provide bounded remainder sets and that, moreover, Rauzy fractals can be considered as canonical bounded remainder sets, up to some affine map.  In the following theorem, we need the fundamental domain~$\cF$ to be bounded and the partition of~$\cF$ to have $d$ atoms for a translation on~$\TT^{d-1}$. 
Recall that we set $\bx' = \pi'(\bx)$ for the projection $\pi'$ defined in (\ref{e:defpi}) and that $\lambda$ denotes the Lebesgue measure.

\begin{theorem} \label{t:nc}
Assume that $(X,\Sigma)$ is the natural coding of a minimal translation~$R_{\bt}$ on~$\TT^{d-1}$ w.r.t.\ a natural partition $\{\cF_1,\dots,\cF_d\}$ of a bounded fundamental domain~$\cF$.
Then the atoms  $\cF_1,\dots,\cF_d$ are bounded remainder sets of~$R_{\bt}$. 
Their Lebesgue measures are rationally independent.

If, moreover, $(X,\Sigma)$ is an $S$-adic dynamical system with $X  = X_{\bsigma}$ for some $\bsigma \in \cS_d^{\NN}$, then
\begin{itemize}
%\item  
%$\bsigma$ admits a generalized eigenvector~$\bu$ (with $\|\bu\|_1=1$),
%\item 
%the $i$-th coordinate of $\bu$ is given by the measure of~$\cF_i$ for $1\le i\le d$,  
\item
$\bu=(\lambda(\cF_1),\ldots,\lambda(\cF_d))$ is a generalized right eigenvector of $\bsigma$,
\item  
there is an affine map $H: \RR^d \to \RR^{d-1}$ such that  $\cF_i = H(\cR_{\bsigma}(i))$ for $1\le i\le d$,
\item
$(X_{\bsigma},\Sigma)$ is a natural coding of~$R_{\bu'}$ w.r.t.\ the natural partition $\{-\cR'_{\bsigma}(i) \,:\, 1 \le i \le d\}$. 
\end{itemize}
Furthermore, if the directive sequence $\bsigma$ is left proper, then for each word $i_0 i_1 \cdots i_n  \in L_{\bsigma}$, the  ``cylinder set''  $\cF_{i_0} \cap R_{\bt}^{-1} \cF_{i_1} \cap \cdots \cap R_{\bt}^{-n} \cF_{i_n}$  is also  a bounded remainder set of~$R_{\bt}$; in particular, $-\cR'_{\bsigma}(i_0 i_1 \cdots i_n)$ is a bounded remainder set of~$R_{\bu'}$.
\end{theorem}

The result also holds if one replaces left properness by right properness. 

As mentioned in the introduction, the study of bounded remainder sets started with the work of W.~M.~Schmidt~\cite{Schmidt:74}. A~vast literature is devoted to the subject, see e.g.\ \cite{Ferenczi92,Grepstad-Lev,Liardet:87,Rauzy:84}. In the case of $S$-adic dynamical systems that are natural codings of a minimal translation on a torus, Theorem~\ref{t:nc} characterizes the bounded remainder sets for letters as affine images of $S$-adic Rauzy fractals and can  be considered as a partial converse to Theorems~\ref{theo:MCF}, \ref{theo:MCF2}, \ref{theo:main} and~\ref{theo:main2}. It shows that these bounded remainder sets ``extend to words'' in the sense that they can be subdivided in a natural way to provide bounded remainder sets for words as well. This yields a great variety of sets of bounded local discrepancy for Kronecker rotations on the torus. In \cite{Liardet:87}, it is shown that only ``trivial'' axis-parallel boxes can be bounded remainder sets for Kronecker sequences and toral translations. The bounded remainder sets constructed in \cite{Grepstad-Lev} are based on polytopes. In all these cases, the bounded remainder sets do not ``extend to words'' like ours.

We also note that such natural codings by $S$-adic dynamical systems provide (nonstationary) Markov partitions in the sense of \cite{Arnoux-Fisher:05} for automorphisms of the torus. We will pursue this in the forthcoming paper \cite{sadic3}.

Theorem~\ref{t:nc} leads us to state the following conjecture  stating, roughly speaking, that a bounded remainder set that ``extends to words'' must have fractal boundary.

\begin{conjecture}\label{conj:BRS}
Let $\{\cF_1,\dots,\cF_h\}$ be a natural partition of a minimal translation~$R_{\bt}$ on $\TT^{d-1}$, $d \ge 3$, such that  all sets $\cF_{i_0} \cap R_{\bt}^{-1} \cF_{i_1} \cap \cdots \cap R_{\bt}^{-n} \cF_{i_n}$,  $i_0 i_1 \cdots i_n \in \{1,\dots,h\}^*$, are bounded remainder sets for~$R_{\bt}$.
Then $\cF_i$ cannot have piecewise smooth boundaries $(1\le i\le h)$.
\end{conjecture}

One argument supporting this conjecture is the above-mentioned relation between natural codings and Markov partitions for automorphisms of the torus, and the fact that Markov partitions cannot have smooth boundaries for hyperbolic automorphisms of the torus in dimension $d\ge 3$, see \cite{Bow78}.

\medskip

After some preparations in Section~\ref{sec:defprep}, the proofs of all main results will be contained in Section~\ref{sec:proofs}. 
The proof of the $S$-adic results in Theorem~\ref{theo:main} and Theorem~\ref{theo:main2} will be  given in Section~\ref{subsec:ptheomain} and Section~\ref{sec:proof-theor-refth}, respectively. 
The results on multidimensional continued fractions, namely Theorems~\ref{theo:MCF} and~\ref{theo:MCF2}, will then be deduced from the corresponding $S$-adic results in Section~\ref{subsec:proofsMCF}. 
Finally, Theorem~\ref{t:nc} is proved in Section~\ref{subsec:proofBRS}. 

\section{Preparations for the proofs of the main theorems} \label{sec:defprep}

Throughout the proofs of our main results, we will need notation, definitions, and results that are recalled in this section.

\subsection{Properties of sequences of substitutions}\label{subsec:prop}
In our main theorems, we put certain assumptions, most notably, the Pisot condition from Definition~\ref{d:Pisot}. 
We will now discuss combinatorial properties that will be satisfied by almost all directive sequences $\bsigma$ under these assumptions. 
We need these combinatorial properties because they occur in some results from \cite{BST:19} that will be important for us. 
Accordingly, most of the definitions stated in the present subsection are taken from \cite[Section~2]{BST:19}.

Let $\bsigma = (\sigma_n) \in \cS_d^{\NN}$ be a sequence of substitutions over a given alphabet $\cA = \{1,\dots,d\}$. 
We say that $\bsigma$ is \emph{primitive}, if for each $k \in \NN$ there exists $n > k$ such that $M_{\sigma_{[k,n)}}$ is a positive matrix. 
If each factor $(\sigma_0, \ldots, \sigma_m)$, $m\in\NN$, occurs infinitely often in~$\bsigma$, then $\bsigma$ is \emph{recurrent}. 
As observed in~\cite[p.~91--95]{Furstenberg:60}, primitivity and recurrence of $\bsigma$ allow for an analog of the Perron--Frobenius theorem for the associated sequence $(M_{\sigma_n})$ of incidence matrices. 
In particular, if $\bsigma$ is primitive and recurrent, then the generalized right eigenvector $\bu$ defined in \eqref{eq:u} exists. 

A~sequence of substitutions $\bsigma$ is said to be \emph{unimodular} if the incidence matrices of the substitutions are unimodular.

Another important property is \emph{algebraic irreducibility}. 
A~sequence of substitutions $\bsigma = (\sigma_n)$ over the alphabet~$\cA$ is called algebraically irreducible if for each $k \in \NN$ the matrix $M_{\sigma_{[k,n)}}$ has irreducible characteristic polynomial provided that $n \in \NN$ is large enough. 
For $S$-adic dynamical systems that arise from multidimensional continued fraction algorithms which satisfy primitivity and the Pisot condition, we can (almost everywhere) prove a result that is even stronger than algebraic irreducibility; see Lemma~\ref{lem:sadic87}.

%For almost every sequence of substitutions, for each $k\in \NN$ the characteristic polynomial of $M_{\sigma_{[k,n)}}$ is the minimal polynomial of a Pisot unit for $n$ large enough; see {\cite[Lemma 8.7]{BST:19}}. This is true in particular if we assume the Pisot condition.

Finally, we require the language given by a sequence of substitutions to be balanced. 
More precisely, a~language~$L$ over a finite alphabet $\cA = \{1,\ldots,d\}$ is said to be \emph{$C$-balanced} if for each two words $w, w' \in L$ with $|w| = |w'|$ we have $\big||w|_i - |w'|_i\big| \le C$ for each $i \in \cA$.
It is called \emph{balanced} if it is $C$-balanced for some $C$. 
We define
\begin{equation}\label{eq_BC}
B_C = \{\bsigma \in \cS_d^\NN \,:\, \mbox{$L_{\bsigma}$ is $C$-balanced}\}.
\end{equation} The following lemma relates balancedness to boundedness of Rauzy fractals.

%In the following lemma we use that balance of $\bsigma$ implies the existence of a generalized right eigenvector and, hence, the existence of $\mathcal{R}_{\bsigma}$.

\begin{lemma}[cf.~{\cite[Lemma~4.1]{BST:19}}]\label{lem:balbound}
Let $\bsigma$ be a primitive sequence of substitutions with a generalized right eigenvector and $C\in\NN$. Then $\bsigma\in B_C$ implies that $\mathcal{R}_{\bsigma} \subset[-C, C]^{d} \cap \bone^\bot$.
\end{lemma}

We mention that unbounded Rauzy fractals were recently studied in \cite{Andrieu:20} for the case of the Arnoux-Rauzy $S$-adic dynamical systems discussed in Section~\ref{subsec:AR}.

We will need results from \cite{BST:19} which require a set of technical conditions that goes under the name \emph{Property PRICE},  which is an abbreviation for \underline{P}rimitivity, \underline{R}ecurrence, algebraic \underline{I}rreducibility, $\underline{C}$-balancedness, and recurrent left \underline{E}igenvector. 

\begin{definition}[Property PRICE]\label{def:PRICE}
A~directive sequence $\bsigma = (\sigma_n)\in \cS_d^{\NN}$ has Property PRICE if the following conditions hold for some strictly increasing sequences $(n_k)_{k\in\NN}$ and $(\ell_k)_{k\in\NN}$ and a vector $\mathbf{v} \in \RR_{\ge0}^d \setminus \{\mathbf{0}\}$.
\begin{itemize}
\labitem{(P)}{defP}
There exists $h \in \NN$ and a positive matrix~$M'$ such that $M_{\sigma_{[\ell_k-h,\ell_k)}} = M'$ for all $k \in \NN$.
\labitem{(R)}{defR}
We have $(\sigma_{n_k}, \dots,\sigma_{n_k+\ell_k-1}) = (\sigma_0, \dots,\sigma_{\ell_k-1})$, i.e., $\Sigma^{n_k} \bsigma \in [\sigma_0, \sigma_1, \dots,\sigma_{\ell_k-1}]$ for all $k\in\NN$.
\labitem{(I)}{defI}
The directive sequence~$\bsigma$ is algebraically irreducible.
\labitem{(C)}{defC}
There exists $C > 0$ such that the language of $\Sigma^{n_k+\ell_k}\bsigma$ is $C$-balanced, i.e., $\Sigma^{n_k+\ell_k}\bsigma \in B_C$ for all $k\in\NN$.
\labitem{(E)}{defE}
We have $\lim_{k\to\infty} \tr{\!}M_{\sigma_{[0,n_k)}} \bv / \|\tr{\!}M_{\sigma_{[0,n_k)}} \bv\|_1 = \bv$.
\end{itemize}
\end{definition}
We note that if $\bsigma$ satisfies Property PRICE, then $\Sigma\bsigma$ also satisfies Property PRICE by \cite[Lemma~5.10]{BST:19}.

\begin{remark}\label{rem:PeriodicPrice}
Since a unimodular Pisot substitution $\sigma$ is primitive by {\cite[Proposition~1.3]{CS:01}} and balanced by {\cite[Theorem~13~(1)]{Adamczewski:03}}, the constant sequence $(\sigma)$ satisfies Property PRICE with $\bv$ being the dominant left eigenvector of $M_\sigma$.
\end{remark}

\subsection{Tilings by Rauzy fractals and coincidence conditions} \label{sec:Rauzy}
As mentioned before, the Rauzy fractals defined in Section~\ref{sec:natur-codings-bound} play a crucial role in proving that the $S$-adic dynamical system $(X_{\bsigma},\Sigma)$ has purely discrete spectrum. 
The importance of Rauzy fractals is due to the fact that one can ``see'' on them the toral translation to which we want to conjugate (in the measure-theoretic sense) an $S$-adic dynamical system $(X_{\bsigma},\Sigma)$; this is worked out in \cite[Section~8]{BST:19}. 
In the substitutive case, the proof of this conjugacy strongly relies on a certain self-affinity property of the subtiles $\cR_{\bsigma}(i)$, $i\in \cA$; see e.g.\ \cite{SirventWang02}. 
In the $S$-adic case, these subtiles are no longer self-affine. 
However, they still satisfy a certain \emph{set equation} that allows to express them as unions of shrunk copies of subtiles $\cR_{\Sigma^n\bsigma}(i)$ corresponding to a shift of the original directive sequence~$\bsigma$.
More precisely, we have the following slight variant of \cite[Proposition~5.6]{BST:19}.

\begin{lemma}\label{lem:seteq}
If $\bsigma$ admits a generalized right eigenvector $\bu$ then
\begin{equation} \label{e:setequationkl}
\cR_{\bsigma}(i) = \bigcup_{p \in \cA^*,\, j \in \cA \,:\, p\,i \preceq \sigma_{[0,n)}(j)} \pi_{\bu} \big(\bl(p) + M_{\sigma_{[0,n)}} \cR_{\Sigma^n\bsigma}(j) \big) \qquad(i \in \cA,\, n \in \NN).
\end{equation}
\end{lemma}

Because the notation (and also the statement) of this lemma differs from \cite[Proposition~5.6]{BST:19}, we provide a full proof for the convenience of the reader.
Figures~\ref{fig:selmertribu} and \ref{fig:1} illustrate Rauzy fractals that are subdivided into subtiles according to Lemma~\ref{lem:seteq}.

\begin{proof}
Let $i \in \cA$, $n \in \NN$. 
According to \eqref{eq:RFactor}, $\cR_{\bsigma}(i)$ is the closure of the set of points of the form $\pi_{\bu}\, \bl(p')$, where $p'i$ is a prefix of  $\sigma_{[0,k)}(j')$ for infinitely many $k >n$, $j' \in \cA$. Since $\sigma_{[0,k)}(j')=\sigma_{[0,n)}\circ\sigma_{[n,k)}(j')$, we conclude that $p'i \preceq\sigma_{[0,k)}(j')$ if and only if $p'$ can be written as $p' = \sigma_{[0,n)}(\tilde{p})\,p$ with $\tilde{p}\,j \preceq \sigma_{[n,k)}(j')$, $p\,i \preceq \sigma_{[0,n)}(j)$ for some $\tilde{p},p \in \cA^*$, $j \in \cA$. Thus
$
\bl(p') = \bl(p) +  \bl(\sigma_{[0,n)}(\tilde{p})) =
\bl(p) +M_{\sigma_{[0,n)}} \bl(\tilde{p}) 
$
and, hence,
\[
\cR_{\bsigma}(i) = \hspace{-3em} \bigcup_{p \in \cA^*,\, j \in \cA \,:\, p\,i \preceq \sigma_{[0,n)}(j)} \hspace{-3em}
\big(\pi_{\bu}\,\bl(p) +\overline{\{ \pi_{\bu} M_{\sigma_{[0,n)}}\, \bl(\tilde p) \,:\, \mbox{$\tilde p\,j \preceq \sigma_{[n,k)}(j')$ for infinitely many $k>n$, $j' \in \cA$}\}}\big).
\]
It remains to show that the latter set is equal to $\pi_{\bu} M_{\sigma_{[0,n)}} \cR_{\Sigma^n\bsigma}(j)$.
It follows from \eqref{eq:u} that $\bu^{(n)} = M_{\sigma_{[0,n)}}^{-1}\bu$ is a generalized right eigenvector of $\Sigma^n\bsigma$. Since $\pi_{\bu^{(n)}}(\bx) = \bx$ for all $\bx \in
\bone^\perp$ and $\pi_{\bu^{(n)}} \bu^{(n)} = \mathbf{0} = \pi_\bu
M_{\sigma_{[0,n)}} \bu^{(n)}$ we have $\pi_\bu M_{\sigma_{[0,n)}} =\pi_\bu M_{\sigma_{[0,n)}}
\pi_{\bu^{(n)}}$, thus
\[
\begin{split}
&\overline{\{ \pi_{\bu} M_{\sigma_{[0,n)}}\, \bl(\tilde p) \,:\, \mbox{$\tilde p\,j \preceq \sigma_{[n,k)}(j')$ for infinitely many $k>n$, $j' \in \cA$}\} }\\
&\qquad=
\pi_{\bu} M_{\sigma_{[0,n)}}\overline{\{\pi_{\bu^{(n)}} \, \bl(\tilde p) \,:\, \mbox{$\tilde p\,j \preceq \sigma_{[n,k)}(j')$ for infinitely many $k>n$, $j' \in \cA$}\} } \\
&\qquad= \pi_{\bu} M_{\sigma_{[0,n)}} \cR_{\Sigma^n\bsigma}(j). \qedhere
\end{split}
\]
\end{proof}

An $S$-adic Rauzy fractal $\cR_{\bsigma}$ has thus two different kinds of natural subsets: the \emph{subtiles} $\cR_{\bsigma}(w)$ defined in \eqref{eq:RFactor} and the (level~$n$) \emph{subdivision tiles} $\pi_{\bu} \big(\bl(p) + M_{\sigma_{[0,n)}} \cR_{\Sigma^n\bsigma}(j)\big)$ occurring on the right hand side of~\eqref{e:setequationkl} for some $i \in \cA$. 
In this section, we will mostly use the subdivision tiles.

We will need the collection\footnote{Note that we cannot exclude a priori that different pairs $(\bx,i)$ give rise to the same set $\bx + \cR_{\bsigma}(i)$, i.e., that $\cC_{\bsigma}$ is a multiset and not a set. If $\cC_{\bsigma}$ forms a tiling, then this possibility is excluded.}
\[
\cC_{\bsigma} = \{\bx + \cR_{\bsigma}(i) \,:\, \bx \in \ZZ^d \cap \bone^\perp,\, i \in \cA\}.
%\cC_{\bsigma} = (\bx + \cR_{\bsigma}(i))_{\bx \in \ZZ^d \cap \bone^\perp,\, i \in \cA}
\]
consisting of the translations of (the subtiles of) the Rauzy fractal $\cR_{\bsigma}$ by vectors in the lattice $\ZZ^d \cap \bone^\perp$.
As shown e.g.\ in \cite{BST:19}, the fact that $\cC_{\bsigma}$ forms a tiling of~$\bone^\perp$ implies that $(X_{\bsigma},\Sigma)$ has purely discrete spectrum. 
Here, a \emph{tiling}  of~$\bone^\perp$\ is a set of tiles that covers $\bone^\perp$ in a way that the intersection of any two distinct tiles has $(d{-}1)$-dimensional Lebesgue measure~$0$. 
Related results for the substitutive case are contained in \cite[Theorem~2]{Arnoux-Ito:01} and \cite[Theorem~3.8]{CS:01}; for the classical example that  initiated the whole theory we refer to~\cite{Rauzy:82}.

It is proved in \cite[Proposition~7.5]{BST:19} that, if Property PRICE holds, $\cC_{\bsigma}$ is a locally finite \emph{multiple tiling} of~$\bone^\perp$ by compact tiles (in the sense that a.e.\ point of $\bone^\perp$ is contained in exactly $m$ elements of~$\cC_{\bsigma}$ for some given $m\ge 1$). 
It is a priori not clear how to decide for a given directive sequence $\bsigma$ if this multiple tiling is actually a tiling. 
However, as shown in \cite[Section~7]{BST:19}, the following \emph{coincidence conditions} (whose meaning will be explained in Remark~\ref{rem:coinc}) can be used to get checkable criteria for this tiling property. 

\begin{definition}[{Geometric coincidence condition}]\label{def:gcc3}
A directive sequence $\bsigma= (\sigma_n)_{n\in\NN}$ satisfies the \emph{geometric coincidence condition} if for each $R > 0$, there is $k \in \NN$ such that, for all $n \ge k$, there exist $\bz_n \in \bone^\perp$, $i_n \in \cA$, such that
\begin{equation} \label{e:gcc3}
\begin{split}
\{\by \in \ZZ^d \,: & \ \|M_{\sigma_{[0,n)}}^{-1} (\by - \bz_n)\| \le R,\, 0 \le \langle \bone, \by\rangle < |\sigma_{[0,n)}(j)|\} \\
\subset \{\bl(p) \,: & \ p \in \cA^*,\, p\, i_n \preceq \sigma_{[0,n)}(j)\} 
\end{split}  \quad \mbox{for all}\ j \in \cA.
\end{equation}
(Recall that $w \preceq v$ means that $w$ is a prefix of~$v$.) 
\end{definition}

This geometric coincidence condition is a rephrasing of the more geometric variant defined in \cite[Section~2.11]{BST:19}. In this geometric setting, the condition ensures suitable growth properties of certain patches of parallelotopes that are defined by the dual $E_1^*(\sigma_{[0,n)})$ of the so-called one-dimensional geometric realization $E_1(\sigma_{[0,n)})$ of $\sigma_{[0,n)}$ for growing $n$.\footnote{The linear maps $E_1(\sigma_{[0,n)})$ and $E_1^*(\sigma_{[0,n)})$  are introduced in \cite{Arnoux-Ito:01}.}
% geometric coincidence is defined by using $E_1^*(\sigma_{[0,n)})$ in \cite[Section~2.11]{BST:19}.} 
Since we do not want to define discrete hyperplanes and dual substitutions here, we use equivalent statements with usual substitutions and abelianizations of words.

It turns out that the following version of the geometric coincidence condition taken from \cite[Proposition~7.9~(iv)]{BST:19} is more useful for our purposes.

\begin{definition}[{Effective version of the geometric coincidence condition}]\label{def:gcc4}
A directive sequence $\bsigma= (\sigma_n)_{n\in\NN}$ satisfies the \emph{effective version of the geometric coincidence condition} if 
there are $n \in \NN$, $\bz \in \bone^\perp$, $i \in \cA$, $C>0$, such that 
\begin{equation} \label{e:gcc4}
\Sigma^n\bsigma \in B_C, \
\begin{gathered}
\big\{\by \in \ZZ^d \,:\, \|\pi_{\mathbf{u}^{(n)}} M_{\sigma_{[0,n)}}^{-1} \by - \bz\|_\infty \le C,\, 0 \le \langle \bone, \by\rangle < |\sigma_{[0,n)}(j)|\big\} \\
\hspace{-11.5em}
\subset \big\{\bl(p) \,:\, p \in \cA^*,\, p\, i \preceq \sigma_{[0,n)}(j)\big\},
\end{gathered} \ \mbox{for all}\ j \in \cA,
\end{equation}
with $\bu^{(n)} = M_{\sigma_{[0,n)}}^{-1}\bu$.
\end{definition}

If $\sigma$ is a substitution for which the constant sequence $(\sigma)_{n\in\NN}$ satisfies the geometric coincidence condition, we say that $\sigma$ satisfies the geometric coincidence condition (and similarly for the effective version).

\begin{remark}\label{rem:coinc}
We want to motivate the geometric coincidence conditions of Definitions~\ref{def:gcc3} and~\ref{def:gcc4} and discuss how they imply that the multiple tiling~$\cC_{\bsigma}$ is a tiling (subject to Property PRICE; proofs will follow in Proposition~\ref{p:gccBST}). First note that these coincidence conditions are about control points of tiles and, in order to understand their meaning, it is useful to replace these control points by the associated tiles. For $n\in\NN$, let $\mathcal{T}_n$ be the collection of all $n$-th subdivision tiles (in the sense of \eqref{e:setequationkl}) of the tiles in~$\cC_{\bsigma}$. The geometric coincidence condition~\eqref{e:gcc3} states that, given $R>0$, for $n$ large enough, there is a subcollection~$\mathcal{P}_n$ consisting of all tiles of~$\mathcal{T}_n$ contained in a large ball (in terms of~$R$ and $M_{\sigma_{[0,n)}}$), such that $\mathcal{P}_n \subset \mathcal{Q}_n$, where $\mathcal{Q}_n$ is the collection of $n$-th subdivision tiles of $\cR_{\bsigma}(i_n)$ for some $i_n\in\cA$ (compare the range of the union in \eqref{e:setequationkl} to the right hand of side \eqref{e:gcc3}). Since it is known from \cite[Proposition~7.3]{BST:19} that the elements of $\mathcal{Q}_n$ are pairwise disjoint in measure (in particular, $\mathcal{Q}_n$ and thus $\mathcal{P}_n$ are sets), $\mathcal{T}_n$ is a multiple tiling that far enough inside $\mathcal{P}_n$ covers without overlaps. (Here, we need that $R$ is large enough to avoid that $M_{\sigma_{[0,n)}}^{-1}\mathcal{P}_n$ is covered again by tiles from $M_{\sigma_{[0,n)}}^{-1}(\mathcal{T}_n\setminus\mathcal{P}_n)$.)
%In other words, $\mathcal{T}_n$ is a multiple tiling that contains (open sets of) {\it exclusive points} that are covered by only one tile of $\mathcal{T}_n$, 
Thus $\mathcal{T}_n$ is a tiling. As the tiles of $\cC_{\bsigma}$ are unions of tiles of $\mathcal{T}_n$, also $\cC_{\bsigma}$ is a tiling.

The size of the patch~$\mathcal{P}_n$ that we require in order to infer that $\cC_{\bsigma}$ is a tiling is determined by the largest diameter of the subtiles in the $n$-th subdivision of $\cR_{\bsigma}(i_n)$. This diameter is in turn determined by the balance constant $C$ of the language $L_{\Sigma^n\bsigma}$. This observation leads to the quantified version of geometric coincidence in~\eqref{e:gcc4}, which is also illustrated in Figure~\ref{fig:1}.
\end{remark}

The geometric coincidence condition can be seen as an $S$-adic analog of the geometric coincidence condition (or super-coincidence condition) in \cite{Barge-Kwapisz:06,Ito-Rao:06,CANTBST}, which provides a tiling criterion in the substitutive case. 
This criterion is a coincidence type condition in the same vein as the various coincidence conditions introduced in the usual Pisot framework; see e.g.\ \cite{Solomyak:97,AL11}.
The term ``coincidence condition'' goes back to Dekking~\cite{Dekking:1977} where it meant that the letters of the images of all letters under a substitution (of constant length) ``coincide'' at a certain position. The  letter $i_n$  in Definition \ref{def:gcc3} and \ref{def:gcc4} plays the role of this common coincidence letter. This condition was further developed and, in the substitutive case, it means that certain broken lines that can be associated with the multiple tiling $\cC_{\bsigma}$ ``coincide'', in the sense that they have at least one edge in common; see e.g.\ \cite{Barge-Kwapisz:06,Ito-Rao:06}.

Results from \cite{BST:19} that are central for our proofs are contained in the following proposition.

\begin{proposition}\label{p:gccBST}
Let $\bsigma \in \cS_d^{\NN}$ be a directive sequence satisfying Property PRICE.
Then the following assertions are equivalent.
\renewcommand{\theenumi}{\roman{enumi}}
\begin{enumerate}
\item \label{i:gcc1} 
The collection $\cC_{\bsigma}$ forms a tiling.
\item \label{i:gcc2} 
The collection $\cC_{\Sigma^n\bsigma}$ forms a tiling for some $n \in \NN$. 
\item \label{i:gcc2'} 
The collection $\cC_{\Sigma^n\bsigma}$ forms a tiling for all $n \in \NN$. 
\item \label{i:gcc3}  
The sequence $\bsigma$ satisfies the geometric coincidence condition.
\item \label{i:gcc4} 
The sequence $\bsigma$ satisfies the effective version of the geometric coincidence condition.
\end{enumerate}
\end{proposition}

\begin{figure}[ht]
\includegraphics[width=.35\textwidth]{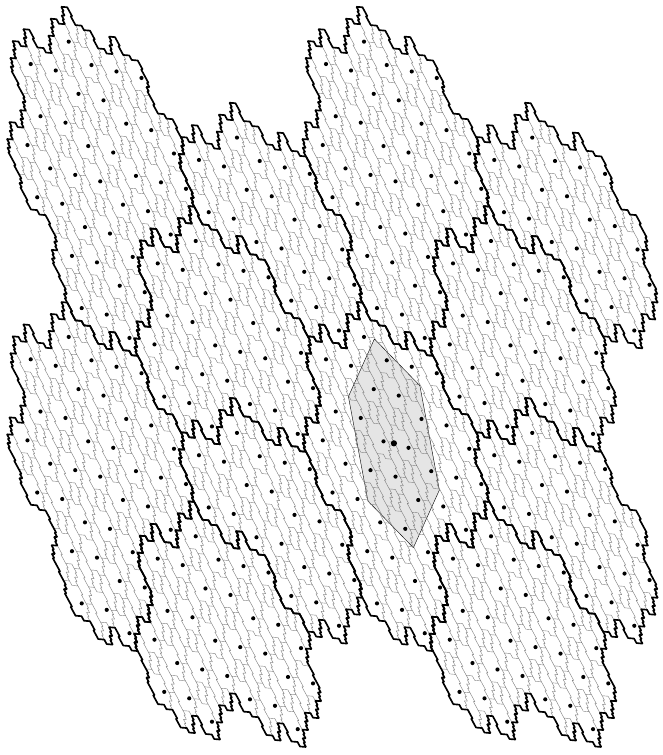}
\caption{Illustration of the proof of Proposition~\ref{p:gccBST} (\ref{i:gcc4})~$\Rightarrow$~(\ref{i:gcc1}). The large tiles are the tiles of~$\cC_{\bsigma}$, the marked points are the translation points of their level~$n$ subdivision tiles (these tiles are drawn in grey; up to three subdivision tiles can share the same translation point in this three-letter example). Because $\Sigma^n \bsigma \in B_C$, these level~$n$ subdivision tiles are bounded in terms of~$C$; here $C=2$. More precisely, the given point $\pi_{\bu} M_{\sigma_{[0,n)}} \bz$ can only be contained in level~$n$ subdivision tiles whose translation points are contained in the (shaded) parallelepiped $\pi_{\bu} M_{\sigma_{[0,n)}} (\bz+ [-C,C]^3 \cap \bone^\perp)$. 
All translation points inside the shaded parallelepiped belong to level~$n$ subdivision tiles of the same tile of~$\cC_{\bsigma}$, namely $\cR_{\bsigma}(i)$; this is the effective version of the geometric coincidence condition. Therefore, $\pi_{\bu} M_{\sigma_{[0,n)}} \bz$ belongs only to level~$n$ subdivision tiles of a single tile of~$\cC_{\bsigma}$. Thus it is an exclusive point of~$\cC_{\bsigma}$.}
\label{fig:1}
\end{figure}

\begin{proof}
This result is proved in \cite{BST:19} but, because our equivalent assertions somewhat differ from the ones in \cite[Proposition~7.9]{BST:19}, we give some details here. 
For given $\bw \in \RR^d_{\ge 0}\setminus\{\mathbf{0}\}$ and $\bsigma \in \cS_d^{\NN}$, we define in  \cite[Section~2.10]{BST:19} a collection $\cC_{\bsigma,\bw}$ similarly to~$\cC_{\bsigma}$.
However, the elements of~$\cC_{\bsigma,\bw}$ are Rauzy fractals that are projected to~$\bw^\bot$. (The detailed definition, which requires some notation, is not relevant for us and we refrain from stating it.) These collections are of particular importance when $\bw$ is equal to the generalized left eigenvector from Definition~\ref{def:PRICE}\ref{defE}. Indeed, letting $\bv$ and $\bv^{(n)}$ be generalized left eigenvectors of $\bsigma$ and $\Sigma^n\bsigma$, respectively, we can use results from \cite{BST:19} to gain that, for each $n\in\NN$,\footnote{Note the different notation in \cite{BST:19}: $\cC_{\bsigma,\bv}=\cC_{\bv}$ and $\cC_{\Sigma^n\bsigma,\bv^{(n)}}=\cC_{\bv}^{(n)}$.}
\begin{align*}
\cC_{\bsigma} & \hbox{ forms a tiling of $\bone^\bot$}  \\
& \Longleftrightarrow   \quad \mbox{$\cC_{\bsigma,\bv}$ forms a tiling of $\bv^\bot$} & & \mbox{\cite[Proposition~7.5]{BST:19}} \\
& \Longleftrightarrow   \quad \mbox{$\cC_{\Sigma^n\bsigma,\bv^{(n)}}$ forms a tiling of $(\bv^{(n)})^\bot$}  & & \mbox{\cite[Lemma~7.2]{BST:19}}
\\
& \Longleftrightarrow   \quad \mbox{$\cC_{\Sigma^n\bsigma}$ forms a tiling of $\bone^\bot$} & & \mbox{\cite[Proposition~7.5]{BST:19}}.
\end{align*}
These equivalences prove that (\ref{i:gcc1})~$\Leftrightarrow$~(\ref{i:gcc2})~$\Leftrightarrow$~(\ref{i:gcc2'}).
%More precisely, by \cite[Proposition~7.5]{BST:19}, $\cC_{\bsigma}$~tiles~$\bone^\perp$ if and only if the collection~$\cC_{\bv}$ defined in \cite{BST:19} tiles~$\bv^\perp$, where $\bv$ is the generalized left eigenvector from the condition PRICE. 
%For any $n \in \NN$, this is equivalent by \cite[Lemma~7.2]{BST:19} to $\cC_{\bv}^{(n)}$ being a tiling of~$(\tr{\!}M_{\sigma_{[0,n)}} \bv)^\perp$, where $\cC_{\bv}^{(n)}$ is defined in \cite{BST:19} by the sequence~$\Sigma^n \bsigma$.
%By applying \cite[Proposition~7.5]{BST:19} for~$\Sigma^n \bsigma$, we obtain that $\cC_{\bone}^{(n)}$~tiles $(\tr{\!}M_{\sigma_{[0,n)}} \bone)^\perp$ if and only if $\cC_{\Sigma^n\bsigma}$ tiles~$\bone^\perp$.
%This proves that (\ref{i:gcc1})~$\Leftrightarrow$~(\ref{i:gcc2})~$\Leftrightarrow$~(\ref{i:gcc2'}).
%
The equivalences (\ref{i:gcc1})~$\Leftrightarrow$~(\ref{i:gcc3})~$\Leftrightarrow$~(\ref{i:gcc4}) are treated in \cite[Proposition~7.9]{BST:19}.
However, the proof of the implication (\ref{i:gcc4})~$\Rightarrow$~(\ref{i:gcc1}) in \cite{BST:19} is somewhat sketchy. 
Since this implication will be of particular importance in the sequel, and in order to further explain the (effective version of the) geometric coincidence condition, we give a more detailed proof of it, which is illustrated in Figure~\ref{fig:1}.

\bigskip

\noindent
\textit{Proof of the implication (\ref{i:gcc4})~$\Rightarrow$~(\ref{i:gcc1}).} Let $\bu$ be a  generalized right eigenvector of~$\bsigma$, which exists because Property PRICE implies that $\bsigma$ is primitive and recurrent. 
Assume that there are $n \in \NN$, $\bz \in \bone^\perp$, $i \in \cA$, $C>0$, such that \eqref{e:gcc4} holds. 
We show that $\pi_{\bu} M_{\sigma_{[0,n)}} \bz$ is an {\it exclusive point} of the collection~$\cC_{\bsigma}$ in the sense that it is contained in only one element of $\cC_{\bsigma}$. 
Since \cite[Proposition~7.5]{BST:19} states that $\cC_{\bsigma}$ is a locally finite multiple tiling by compact tiles, this will already imply that $\cC_{\bsigma}$ is in fact a tiling,  because the compactness of the tiles together with local finiteness yield that each exclusive point has a neighborhood consisting of exclusive points.
Since $\cC_{\bsigma}$ forms a multiple tiling and, hence, a covering of~$\bone^\bot$, we have $\pi_{\bu} M_{\sigma_{[0,n)}} \bz \in \bx + \cR_{\bsigma}(i')$ for some $(\bx,i') \in (\ZZ^d \cap \bone^\perp) \times \cA$. 
To prove exclusivity, we have to show that this choice of $(\bx,i')$ is unique. 
By the set equation in Lemma~\ref{lem:seteq} for $\cR_{\bsigma}(i')$, there exist $p' \in \cA^*$, $j' \in \cA$ with 
\begin{equation}\label{eq:piprefj'}
p' i' \preceq \sigma_{[0,n)}(j')
\end{equation}
 such that
\begin{equation}\label{eq:purauzy}
\pi_{\bu} M_{\sigma_{[0,n)}} \bz \in \pi_{\bu} \big(\bx + \bl(p') + M_{\sigma_{[0,n)}} \cR_{\Sigma^n\bsigma}(j')\big).
\end{equation}
As in the proof of Lemma~\ref{lem:seteq}, note that $\pi_{\bu} M_{\sigma_{[0,n)}} = \pi_{\bu}\, M_{\sigma_{[0,n)}} \pi_{\bu^{(n)}}$, where $\bu^{(n)}=M_{\sigma_{[0,n)}}^{-1}\bu$ is a generalized right eigenvector of $\Sigma^n\bsigma$.
Therefore, \eqref{eq:purauzy} implies that
\begin{equation} \label{eq:zyDiff2}
\pi_{\bu} M_{\sigma_{[0,n)}} \bz \in \pi_{\bu} M_{\sigma_{[0,n)}} \big(\pi_{\bu^{(n)}} M_{\sigma_{[0,n)}}^{-1} \big(\bx + \bl(p')\big) + \cR_{\Sigma^n\bsigma}(j')\big).
\end{equation}
Since $\bu \in M_{\sigma_{[0,n)}} \RR_{\ge0}^d\setminus\{\mathbf{0}\}$ implies that $\bu \notin M_{\sigma_{[0,n)}} (\bone^\perp)$, the mapping $\pi_{\bu} M_{\sigma_{[0,n)}}|_{\bone^\perp}: \bone^\perp\to \bone^\perp$ is a bijection. Therefore, and because $\bz$, $\pi_{\bu^{(n)}} M_{\sigma_{[0,n)}}^{-1} \big(\bx + \bl(p')\big)$, and $\cR_{\Sigma^n\bsigma}(j')$ are contained in~$\bone^\perp$, \eqref{eq:zyDiff2} is equivalent to
\begin{equation}\label{eq:projgone}
\bz \in \pi_{\bu^{(n)}} M_{\sigma_{[0,n)}}^{-1} \big(\bx + \bl(p')\big) + \cR_{\Sigma^n\bsigma}(j').
\end{equation}
Because we assume \eqref{e:gcc4}, we have $\Sigma^n \bsigma\in B_C$ and thus Lemma~\ref{lem:balbound} implies that $\|\by\|_\infty \le C$ for all $\by \in \cR_{\Sigma^n\bsigma}$, hence, \eqref{eq:projgone} yields
%{ By the definition of the Rauzy fractal, balance of the language~$L_{\Sigma^n\bsigma}$ is related to bounds on the size of~$\cR_{\Sigma^n\bsigma}$. In particular, $\Sigma^n \bsigma \in B_C$ yields $\|\by\|_\infty \le C$ for all $\by \in \cR_{\Sigma^n\bsigma}$; see also \cite[Lemma~4.1]{BST:19}.}
\[
\|\pi_{\bu^{(n)}} M_{\sigma_{[0,n)}}^{-1} \big(\bx + \bl(p')\big) - \bz\|_\infty
\le C.
\]
Since $\langle \bone, \bx + \bl(p')\rangle = \langle \bone, \bl(p')\rangle = |p'| < |\sigma_{[0,n)}(j')|$, by \eqref{e:gcc4} we may conclude that $\bx+\bl(p') = \bl(p)$ for some $p \in \cA^*$ with $p\,i \preceq\sigma_{[0,n)}(j')$.
In particular, we have $|p'| = |p|$.
Since $p'i'$ is also a prefix of $\sigma_{[0,n)}(j')$ by \eqref{eq:piprefj'}, we obtain that $p' = p$ and $i' = i$, thus $\bx = \mathbf{0}$.
%\deleted[id=W]{(The set $\{\pi_{\bu}M_{\sigma_{[0,n)}}\by : \| \bz -\by\|_\infty \le C\}$ is the shaded cube in Figure~\ref{fig:1}; in view of the set equation \eqref{e:setequationkl}, the effective version of geometric coincidence in \eqref{e:gcc4} guarantees that this cube contains only translation points of tiles that occur as subdivision tiles of $\cR_{\bsigma}(i)$.)}
Therefore, $(\bx,i')=(\mathbf{0},i)$ is the only possible choice for $(\bx,i')$ and, hence, $\mathbf{0} + \cR_{\bsigma}(i)$ is the only tile of the collection~$\cC_{\bsigma}$ containing $\pi_{\bu} M_{\sigma_{[0,n)}} \bz$. This proves that $\pi_{\bu} M_{\sigma_{[0,n)}} \bz$ is an exclusive point of $\cC_{\bsigma}$ and, hence, yields that the collection $\cC_{\bsigma}$ is a tiling (and, a fortiori, that all elements of~$\cC_{\bsigma}$ are different). This concludes the proof of the implication (\ref{i:gcc4})~$\Rightarrow$~(\ref{i:gcc1}).
\end{proof}

\subsection{Purely discrete spectrum implies geometric coincidence}
In our main theorems, substitutive dynamical systems with purely discrete spectrum play a  key role. 
The following lemma shows that in the substitutive case purely discrete spectrum is equivalent to the geometric coincidence  condition, and thus, by Proposition~\ref{p:gccBST}, also to its effective version. 
This will be crucial in the proofs of Theorems~\ref{theo:main} and~\ref{theo:main2}; see also the discussion before Lemma~\ref{lem:tau}. 
Indeed, let $\tau$ be a unimodular Pisot substitution that satisfies the geometric coincidence condition. 
We will show that the existence of occurrences of long blocks of~$\tau$ in a given directive sequence~$\bsigma$ allows to  ``transfer'' the effective version of the coincidence condition from $\tau$ to~$\bsigma$. 
Using the following lemma, this ``transfer'' works for purely discrete spectrum property as well.

\begin{lemma}\label{rem:spectrumEquiv}
Let $\sigma$ be a unimodular Pisot substitution. 
Then $(X_\sigma,\Sigma)$ has purely discrete spectrum if and only if $\sigma$ satisfies the geometric coincidence condition.
\end{lemma}

\begin{proof} 
%Assume first that $\sigma$ satisfies the geometric coincidence condition. By Remark~\ref{rem:PeriodicPrice}, the periodic sequence $(\sigma)_{n\in\NN}$  satisfies property PRICE. Thus Proposition~\ref{p:gccBST} implies that $\cC_{(\sigma)}$ forms a tiling. Therefore, $(\sigma)$ satisfies the conditions of {\cite[Theorem~3.1 (vi)]{BST:19}} which implies that $(X_{(\sigma)},\Sigma)=(X_{\sigma},\Sigma)$ has purely discrete spectrum.
Assume that $(X_\sigma,\Sigma)$ has purely discrete spectrum.\footnote{Since there does not seem to exist a direct proof of the fact that purely discrete spectrum of $(X_\sigma,\Sigma)$ implies the geometric coincidence condition, we have to take the deviation via tiling flows in the proof of this lemma. Because tiling flows will play no role in this paper, we refrain from giving detailed definitions and refer e.g.\ to~\cite{Barge-Kwapisz:06}.} As $\sigma$ is primitive, the elements of $(X_\sigma,\Sigma)$ have sublinear complexity by \cite[Proposition~5.12]{Queffelec:10}, by \cite[Proposition~5.1.12]{Fog02} we can uniquely extend a.e.\ sequence in $(X_\sigma,\Sigma)$ to a two-sided infinite sequence having the same language. Hence, it is immaterial if we define $(X_\sigma,\Sigma)$ by using one- or two-sided infinite sequences. 

Next we claim that $(X_\sigma,\Sigma)$ has purely discrete spectrum if and only if the tiling flow $(\mathcal{T}_\sigma,T)$ associated with $(X_\sigma,\Sigma)$ has purely discrete spectrum. To prove sufficiency, assume that $(X_\sigma,\Sigma)$ has  purely discrete spectrum. Then $(X_\sigma,\Sigma)$ is measurably conjugate to a translation $x\mapsto x + \alpha$ on a compact abelian group $G$ via a measurable conjugacy $\Phi$.
Let $(\tilde X_\sigma, \tilde T)$  be the suspension flow with constant roof function $f(\omega)\equiv c$. Then $(\tilde X_\sigma, \tilde T)$ is measurably conjugate to the translation $\tilde x \mapsto \tilde x + (0,ct)$ on the compact abelian group $\tilde G = (G \times \RR)/\sim$, with
$(g,c) \sim (g+\alpha,0)$, via the measurable conjugacy $\Phi \times {\rm id}$.
Thus, by a slight variation of \cite[Theorem~3.5]{Walters:82}, $(\tilde X_\sigma, \tilde T)$ has purely discrete spectrum. If the constant~$c$ is chosen properly, \cite[Corollary~5.7]{Barge-Kwapisz:06} shows that $(\tilde X_\sigma, \tilde T)$ and the tiling flow $(\mathcal{T}_\sigma,T)$ associated with $(X_\sigma,\Sigma)$ are conjugate; see also \cite[Corollary~3.2]{CS:03}. Thus, $(\mathcal{T}_\sigma,T)$ has purely discrete spectrum. Necessity is due to \cite[Corollary~5.2]{SS:02}.

Next we establish that the tiling flow $(\mathcal{T}_\sigma,T)$ has purely discrete spectrum  if and only  the substitution~$\sigma$ satisfies the geometric coincidence condition.  Indeed, according to \cite[Corollary~9.4]{Barge-Kwapisz:06}, $(\mathcal{T}_\sigma,T)$ has purely discrete spectrum if and only if  the so-called \emph{coincidence rank} of~$T$ is equal to~$1$.\footnote{The fact that the coincidence rank is equal to one is the analog of our geometric coincidence condition in the setting of flows, see~\cite[Section~7]{Barge-Kwapisz:06}.} This, by {\cite[Remark~18.5]{Barge-Kwapisz:06}}, is in turn equivalent to the fact that the collection~$\cC_{(\sigma)}$ of (substitutive) Rauzy fractals associated with the constant sequence~$(\sigma)$ forms a tiling. Finally, because the constant sequence $(\sigma)$ satisfies Property PRICE by Remark~\ref{rem:PeriodicPrice}, Proposition~\ref{p:gccBST}  shows that this tiling property holds if and only if the substitution~$\sigma$ satisfies the geometric coincidence condition.

This chain of equivalences proves the lemma.

%
%\deleted[id=J]{In Section~\ref{sec:Rauzy}, we recalled the well-known fact that if $\sigma$ satisfies the geometric coincidence condition, then $(X_\sigma,\Sigma)$ has purely discrete spectrum; see \cite{Ito-Rao:06}. 
%As mentioned in the introduction of \cite{Barge-Kwapisz:06}, it follows from {\cite[Theorem~3.1]{CS:03}} that $(X_\sigma,\Sigma)$ has purely discrete spectrum if and only if the associated tiling flow $T$ has purely discrete spectrum (as e.g.\ in \cite{Sadun:15}, just note that if all the tiles in the self-similar tiling space $\mathcal{T}$ have length~$1$, the spectrum of~$T$ on~$\mathcal{T}$ is (up to a multiplicative constant) the logarithm of the spectrum of the shift operator $\Sigma$ on~$X_\sigma$). 
%According to {\cite[Corollary~9.4 and Remark~18.5]{Barge-Kwapisz:06}}, the flow~$T$ has purely discrete spectrum if and only if the collection~$\cC_\sigma$ of Rauzy fractals associated with~$\sigma$ forms a tiling. 
%Thus, Proposition~\ref{p:gccBST} implies that the substitution~$\sigma$ satisfies the geometric coincidence condition.
%}
\end{proof}

\subsection{Balance and bounded remainder sets} \label{subsec:balancewords}
In the sequel, we will strongly rely on the relation between balance and bounded remainder sets. We are interested in bounded remainder sets given by arbitrary words and not only by letters. Therefore, we also consider balance for words:
A~language~$L$ is {\em balanced for the word $v \in L$} if there exists some $C_v \ge 1$ such that, for any two words $w, w' \in L$ with $|w| = |w'|$, we have $\big||w|_v - |w'|_v\big| \le C_v$, and $L$ is \emph{balanced for words} if it is balanced on each $v \in L$.
Here, $|w|_v$ denotes the number of occurrences of the factor~$v$ in~$w$.  
Without further precision, balance will always refer to letters. We note that, in case a directive sequence $\bsigma$ is primitive and proper, balance for letters of the language  $L_{\bsigma}$ implies its balance for all words; see \cite[Corollary~5.5]{BCDLPP:19}.

The quantity~$C$ occurring in the definition of bounded remainder sets (i.e., in Definition~\ref{def:BRS}) can be considered as a notion of local discrepancy; see e.g.\ \cite{Adamdis}. 
To illustrate this, we characterize balance by the following geometric version of \cite[Proposition~7]{Adamczewski:03}, using the projection~$\pi_{\bu}$ defined in (\ref{def:piu}).
For $\bu \in \RR_+^d$ with $\|\bu\|_1 = 1$, we have $\pi_{\bu}\,\bl(w) = \bl(w)-|w|\,\bu$, which is a geometric version of local discrepancy when $\bu$ is a letter frequency vector.

% This is of interest in its own right because $\pi_{\bu}\,\bl(w) = \bl(w)-|w|\bu$ with $w$ in the language of $X$; indeed, by noticing that $\bu$ is a vector containing frequencies of letters of the sequences in  $X$
% we see that this is a geometric version of local discrepancy.  

\begin{proposition} \label{prop:balance}
Let $(X,\Sigma)$ be a uniquely ergodic minimal shift over the alphabet $\cA = \{1,\ldots,d\}$. 
Let $\bu = (u_1,\ldots, u_d)$ be the vector whose entry $u_i$ equals the measure of the cylinder $[i]$ for each $i \in \cA$. Then the language of~$X$ is balanced for letters if and only if $\sup\{\|\pi_{\bu}\,\bl(w)\|\,:\, \mbox{$w$ in the language of $X$}\}$ is bounded. 
Moreover, $(X,\Sigma)$ is balanced for the word~$v$ if and only if the cylinder $[v]$ is a bounded remainder set.
\end{proposition}

 \begin{proof}
Let $L$ be the language of $X$ and denote the unique $\Sigma$-invariant measure of $(X,\Sigma)$ by~$\mu$. %Recall that $v \in \cA^*$ has \emph{uniform frequency} $f_v(x)$ in $x=x_0x_1\cdots \in X$ if the ratio $|x_k \ldots x_{k+j-1}|_v / j$ tends to the limit $f_v(x)$ (which does not depend on~$k$) for $j \rightarrow \infty$ uniformly in~$k$. 
 
We start with the proof of the second assertion.
Assume first that $v\in L$ is chosen in a way that $[v]$ is a bounded remainder set, and let $w,w'\in L$ with $|w|=|w'|=m$ be given. Choose $x = x_0x_1\cdots\in X$. Then, by minimality, there exist $n,n' \in \NN$ such that $\Sigma^n x \in [w]$ and $\Sigma^{n'} x \in [w']$. Thus, because $[v]$ is a bounded remainder set, 
\begin{align*}
\big| |w|_v - |w'|_v   \big| & \le  \big|  |x_0\cdots x_{n+m-1}|_v - |x_0\cdots x_{n-1}|_v - | x_0\cdots x_{n'+m-1}|_v +  |x_0\cdots x_{n'-1}|_v    \big| + 2(|v|-1) \\
& \le \mu([v])\, \big|  (n+m) - n - (n'+m) +  n'\big|_v +2(|v|-1)+4C = 2(|v|-1)+4C
\end{align*}
holds for some $C>0$. (The summand $2(|v|-1)$ comes from occurrences of~$v$ in $x$ that partially overlap with $x_0\cdots x_{n-1}$ or with $x_0\cdots x_{n'-1}$.) Thus $L$ is $(2(|v|-1)+4C)$-balanced for the word $v$. 

Assume now that $L$ is $C$-balanced for~$v$, and let $x=x_0x_1\cdots\in X$ be generic for the measure $\mu$.
Then we have
\[
\big| |x_0\cdots x_{n-1}|_v - n\, \mu([v])\big| = \lim_{m\to\infty} \Big||x_0\cdots x_{n-1}|_v - \frac{1}{m} |x_0 \cdots x_{mn-1}|_v\Big| 
\]
for all $n \in \NN$ because $x$ is generic; moreover,
\[
0 \le |x_0 \cdots x_{mn-1}|_v - \sum_{k=0}^{m-1} |x_{kn} \cdots x_{(k+1)n-1}|_v \le (m-1)\,(|v|-1)
\]
for all $m,n \in \NN$ because we only have to count the number of  occurrences of~$v$ at positions $kn-h$, $1 \le k < m$, $1 \le h < |v|$, and
\[
\Big|m\, |x_0\cdots x_{n-1}|_v - \sum_{k=0}^{m-1} |x_{kn} \cdots x_{(k+1)n-1}|_v\Big| \le m C
\]
by the $C$-balancedness for~$v$.
Putting everything together, we obtain that 
\begin{equation} \label{eq:vwbal}
\big||x_0\cdots x_{n-1}|_v - n\, \mu([v])\big| \le C + |v|-1
\end{equation}
for all $n \in \NN$, thus $[v]$ is a bounded remainder set. 

% Now assume that $L$ is $C$-balanced on~$v$, and let $w \in L$. 
% Then $w$ is a prefix of some $x \in X$. Write $x$ as concatenation of words $w^{(k)}$, $k \in \NN$, with $|w^{(k)}|=|w|$. 
% Then $C$-balance on $v$ yields $\big| |w|_v-|w^{(k)}|_v\big| \leq C$ for all $k \in \NN$, thus $\big||w|_v - \frac{1}{n} \sum_{k=0}^{n-1} |w^{(k)}|_v\big| \leq C$ for all $n \in \NN$.  By minimality and unique ergodicity, $v$ has positive uniform frequency $\mu([v])$ in~$x$; see \cite[Corollary~4.2, and (4.1) in its proof]{Queffelec:10}, thus
% \[
% |w| \mu([v]) - (|v|-1) \le \limsup_{n\to\infty} \frac{1}{n} \sum_{k=0}^{n-1} |w^{(k)}|_v \le |w|\mu([v]).
% \]
% (Mind that the sum counts occurrences of~$v$ in the prefix $w^{(0)} \cdots w^{(n-1)}$ of~$x$ except those overlapping with $w^{(k)}$ and $w^{(k+1)}$ for some~$k$.) Consequently,
% \begin{equation}\label{eq:vwbal}
% \big| |w|_v -|w|\, \mu([v]) \big| \le \big||w|_v - \limsup_{n\to\infty} \frac{1}{n} \sum_{k=0}^{n-1} |w^{(k)}|_v\big| \le (|v|-1)+C,
% \end{equation}
% and, since $w$ can be any prefix of any element of $X$, $[v]$ is a bounded remainder set. 

To prove the first assertion, assume that $\sup\{\|\pi_{\bu}\,\bl(w)\|_\infty\,:\, \mbox{$w\in L$}\} = C$ (w.l.o.g.\ we may use the $\infty$-norm). Let $w,w'\in L$ with $|w|=|w'|$ be given. Then $\bl(w)-\bl(w') \in \bone^\bot$ and, hence,
 \[
 \max_{i\in \cA}(\big| |w|_i - |w'|_i \big|) = \|\bl(w) - \bl(w')\|_\infty= \|\pi_{\bu}(\bl(w) - \bl(w'))\|_\infty \le 2C.
 \]
Thus $L$ is $(2C)$-balanced for letters. Now assume that $L$ is $C$-balanced for letters. %We follow the proof of \cite[Lemma~4.1]{BST:19}. 
%Let $w \in L$. Then $w$ is a prefix of some $x \in X$. Write $x$ as concatenation of words $w_{k}$, $k \in \NN$, with $|w_{k}|=|w|$. Then $C$-balance yields 
%$\left\|\pi_\bu \bl(w)-\pi_\bu \bl(w_{k})\right\|_\infty \leq C$ for all $k \in \NN$, thus $\left\|\frac{1}{n} \sum_{k=0}^{n-1} \pi_\bu \bl(w)-\pi_\bu \bl(w_{k})\right\|_\infty \leq C$ for all $n \in \NN$. We have $\bu=(f_a(x))_{a\in\cA}$ by minimality and unique ergodicity (see \cite[Corollary~4.2, and in particular (4.1) in its proof]{Queffelec:10}). Therefore, we have $\lim _{n \rightarrow \infty} \frac{1}{n} \sum_{k=0}^{n-1} \bl(w_{k})\in \RR\bu$ and, hence, $\lim _{n \rightarrow \infty} \frac{1}{n} \sum_{k=0}^{n-1} \pi_\bu \bl(w_{k})=\mathbf{0}$. Consequently
%\[
%\|\pi_\bu \bl(w)\|=\left\|\lim _{n \rightarrow \infty} \frac{1}{n} \sum_{k=0}^{n-1} \pi_\bu \bl(w)-\pi_\bu \bl(w_{k})\right\|_\infty \leq C.
%\]
Then, in the same way as we derived \eqref{eq:vwbal}, we gain
%\begin{equation}\label{eq:Awbal}
$\max_{i\in \cA} \big(\big| |w|_i -|w| u_i \big|\big) = \max_{i\in \cA} \big(\big| |w|_i -|w| \mu([i]) \big|\big) \le  C$
%\end{equation}
for all $w\in L$. Since $\pi_{\bu} \be_i = \be_i - \bu$ holds for each $i\in\cA$, we have $\pi_{\bu}\,\bl(w) = (|w|_i-|w|u_i)_{i\in\cA}$ for $w \in \cA^*$. Thus $\sup\{\|\pi_{\bu}\,\bl(w)\|_\infty\,:\, \mbox{$w\in L$}\} \le C$.
% Since $\pi_{\bu} \be_i = \be_i - \bu$ holds for each $i\in\cA$, we have $\pi_{\bu}\,\bl(w) = (|w|_i-|w|u_i)_{i\in\cA}$ for $w \in \cA^*$.
%Assume that $\sup\{\|\pi_{\bu}\,\bl(w)\|: \mbox{$w$ in the language of $X$}\} = C$ and let $w,w'$ be two elements of the language
% $\|\bl(w) - \bl(w')\| = \|\pi_{\bu}(\bl(w) - \bl(w'))\| \le 2C$ for all $w,w' \in \cA^*$ with $|w|=|w'|$, thus the language of~$X$ is $(2C)$-balanced. 
 \end{proof}

%\deleted[id=W]{This result will be essential for the proof of Lemma~\ref{lem:balanced}, which relates a given unimodular Pisot matrix to substitutions~$\sigma$ having balanced language~$L_{\sigma}$; see also \cite[Theorem~1]{Adamdis}.}

\section{Proofs of the main results}\label{sec:proofs}
This section contains the proofs of all our main results. 
In Sections~\ref{subsec:ptheomain} and~\ref{sec:proof-theor-refth}, we prove the results stated in Section~\ref{sec:mainSadic} on shifts of directive sequences. 
In Section~\ref{subsec:proofsMCF}, we will use these results to derive the theorems on multidimensional continued fraction algorithms formulated in Section~\ref{sec:mainMCF}. Section~\ref{subsec:proofBRS} is devoted to the proof of Theorem~\ref{t:nc} on natural codings and bounded remainder sets.

\subsection{Proof of Theorem \ref{theo:main}} \label{subsec:ptheomain}
For convenience, we recall the assumptions of Theorem~\ref{theo:main}. 
Let $D \subset \cS_d^{\NN}$ be a shift-invariant set of directive sequences equipped with an ergodic $\Sigma$-invariant Borel probability measure~$\nu$ satisfying $\nu \circ \Sigma \ll \nu$. 
Assume that 
\begin{itemize}
\item
the linear cocycle $(D,\Sigma,Z,\nu)$ defined by $Z((\sigma_n)_{n\in\NN}) = \tr{\!}M_{\sigma_0}$ satisfies the Pisot condition;
\item
there is a periodic Pisot sequence with purely discrete spectrum and positive range in $(D,\Sigma,\nu)$. 
\end{itemize}
We first show that under these assumptions  $\nu$-almost all $\bsigma \in D$ satisfy Property PRICE. 
To this end, we need the following auxiliary results.

\begin{lemma}[{cf.~\cite[Lemma~8.7]{BST:19}}]\label{lem:sadic87}
Let the assumptions of Theorem~\ref{theo:main} be in force. If $\nu$-almost all $(\sigma_n)\in D$ are primitive, then for $\nu$-almost every sequence 
$(\sigma_n) \in D$, for each $k \in \NN$, the characteristic polynomial of $M_{\sigma_{[k,n)}}$ is the minimal polynomial of a Pisot unit for all sufficiently large $n\in\NN$.
\end{lemma}

Contrary to the assumptions in {\cite[Lemma~8.7]{BST:19}}, the shift invariant set $D$ is not required to be closed in Lemma~\ref{lem:sadic87}. Nevertheless, the lemma holds by the same proof as {\cite[Lemma~8.7]{BST:19}}.

In the statement of the next result, recall that $B_C$ is defined in \eqref{eq_BC} and denotes the set of sequences in $\cS_d^{\NN}$ with $C$-balanced language. 

\begin{lemma} \label{l:nuBC}
Under the assumptions of Theorem~\ref{theo:main}, we have $\lim_{C\to\infty} \nu(D \cap B_C) = 1$, in particular $D \cap B_C$ is $\nu$-measurable for all $C > 0$. 
\end{lemma}

\begin{proof}
We first show that a.e.\ $\bsigma\in D$ is primitive. By assumption, $D$~contains a periodic Pisot sequence with positive range, i.e., there is a sequence $\btau = (\tau_n) \in D$ with the following properties:
\begin{itemize} 
\item[(a)] 
There is $j\ge 1$ such that $\Sigma^j \btau = \btau$ and $\tau_{[0,j)}$ is a unimodular Pisot substitution;
\item[(b)] 
$\inf_{n\in\NN} \nu(\Sigma^n[\tau_0,\dots,\tau_{n-1}]) > 0$. 
\end{itemize}
Since $\tau_{[0,j)}$ is a unimodular Pisot substitution by~(a), Remark~\ref{rem:PeriodicPrice} implies that it is primitive and, hence, there is $k \in \NN$ such that $\tau_{[0,kj)}$ has positive incidence matrix. 
Set $h = kj$. 
Because $\nu \circ \Sigma^h \ll \nu$, (b)~implies that 
\begin{equation}\label{eq:taupos}
\nu([\tau_0,\dots,\tau_{h-1}]) > 0. 
\end{equation}
Ergodicity of $\nu$ and the Poincar\'e Recurrence Theorem therefore yield that a.e.\ $\bsigma\in D$ contains $[\tau_0,\dots,\tau_{h-1}]$ infinitely often and, hence, a.e.\ $\bsigma\in D$ is primitive.

Since the Pisot condition holds, since $\tau_{[0,h)}$ has positive incidence matrix, and since \eqref{eq:taupos} holds, we gain from \cite[Theorem~6.4]{Berthe-Delecroix} that $\nu\big(\bigcup_{C\in\NN} (D \cap B_C)\big) = 1$.
Since $B_C \subseteq B_{C'}$ for all $C < C'$, it only remains to show that $D \cap B_C$ is $\nu$-measurable for all $C > 0$.
Let $C > 0$ be arbitrary but fixed and set
\[
B'_C = \bigcap_{n\in\NN} \bigcup_{(\sigma_0,\dots,\sigma_{n-1})\in \cS_d^n:\, \nu([\sigma_0,\dots,\sigma_{n-1}] \cap B_C) > 0} [\sigma_0,\dots,\sigma_{n-1}].
\]
(Recall that the cylinders $[\sigma_0,\dots,\sigma_{n-1}]$ are subsets of $D$ according to Definition~\ref{def:pr}.)
Then we clearly have $D \cap B_C \subseteq B'_C$.
On the other hand, if  $\bsigma \in B'_C$ is primitive, then $\bsigma \in D$ and the finite languages 
\[
L_{\bsigma}^{(n)} = \big\{  w\in \cA^* \,:\, \mbox{$w$ is a factor of $\sigma_{[0,n)}(i)$ for some $i \in \cA$}\big\}
\]
are $C$-balanced for all $n \in \NN$.
Since $L_{\bsigma}^{(0)} \subseteq L_{\bsigma}^{(1)} \subseteq \cdots$, also $L_{\bsigma} = \bigcup_{n\in\NN} L_{\bsigma}^{(n)}$ is $C$-balanced, i.e., $\bsigma \in B_C$. 
Hence, because a.e.\ directive sequence in~$D$ is primitive, we have $\nu((D \cap B_C) \triangle B'_C)=0$.
Since cylinders are measurable (they are open sets and $\nu$ is a Borel measure on~$D$) and countable unions and intersections of measurable sets are measurable, we obtain that $B'_C$ and, hence, also $D \cap B_C$ is $\nu$-measurable. 
\end{proof}

\begin{proposition}\label{prop:combcond}
Under the assumptions of Theorem~\ref{theo:main}, $\nu$-almost every
$\bsigma \in D$ satisfies  Property PRICE. 
\end{proposition}

\begin{proof}
Let $\btau=(\tau_n)\in D$ be a periodic Pisot point. We saw in the proof of Lemma~\ref{l:nuBC}  that there is $h\in \NN$ such that $\tau_{[0,h)}$ has positive incidence matrix and \eqref{eq:taupos} holds. Thus by Lemma~\ref{l:nuBC} there is $C \in \NN$ such that $\nu(\Sigma^{-h} (D \cap B_C)) = \nu(D \cap B_C) > 1 - \nu([\tau_0,\dots,\tau_{h-1}])$, and, hence, $\nu\big([\tau_0,\dots,\tau_{h-1}] \cap \Sigma^{-h} B_C\big) > 0$.

By ergodicity of~$\nu$ together with the Poincar\'e Recurrence Theorem, we have for almost all $\bsigma = (\sigma_n)_{n\in\NN} \in D$ some $\ell_0(\bsigma) \ge h$ such that $\Sigma^{\ell_0(\bsigma)-h} \bsigma \in [\tau_0,\dots,\tau_{h-1}] \cap \Sigma^{-h} B_C$, i.e., $(\sigma_0,\dots,\sigma_{\ell_0(\bsigma)-1})$ ends with $(\tau_0,\dots,\tau_{h-1})$ and $\Sigma^{\ell_0(\bsigma)} \bsigma \in B_C$. 
We will now extend $\ell_0(\bsigma)$ for almost all $\bsigma \in D$ to a sequence $(\ell_k(\bsigma))_{k\in\NN}$ such that
\begin{itemize}
\item
$(\sigma_0,\dots,\sigma_{\ell_{k+1}(\bsigma)-1})$ ends with $(\sigma_0,\dots,\sigma_{\ell_k(\bsigma)-1})$ (and, a fortiori, with $(\tau_0,\ldots,\tau_{h-1})$),
\item
$\Sigma^{\ell_{k+1}(\bsigma)} \bsigma \in B_C$,
\item
$\ell_{k+1}(\bsigma) \ge 2\ell_k(\bsigma)$,
\end{itemize}
for all $k \in \NN$.
To this end, assume that $\ell_0(\bsigma),\ldots,\ell_k(\bsigma)$ are already defined for almost all $\bsigma \in D$.
Consider the set of all $\bsigma$ having a given value $\ell_k=\ell_k(\bsigma)$ and a given prefix $(\sigma_0,\dots,\sigma_{\ell_k-1})$.
Assume that this set has positive measure, which implies that $\nu\big([\sigma_0,\dots,\sigma_{\ell_k-1}] \cap \Sigma^{-\ell_k} B_C\big) > 0$.
Then, for almost all $\bsigma$ in this set, we obtain (by the Poincar\'e Recurrence Theorem and ergodicity of~$\nu$) some $\ell_{k+1}(\bsigma)$ with the required properties. 
Applying this for all choices of $\ell_k$ and $(\sigma_0,\dots,\sigma_{\ell_k-1})$, we get some $\ell_{k+1}(\bsigma)$ for almost all $\bsigma \in D$.
Therefore, such a sequence $(\ell_k(\bsigma))_{k\in\NN}$ exists for almost all $\bsigma \in D$.

Setting $n_k (\bsigma) = \ell_{k+1}(\bsigma) - \ell_k(\bsigma)$, we obtain that conditions (P), (R) and~(C) of Property PRICE hold for almost all $\bsigma \in D$.
By \cite[Lemma~5.7]{BST:19}, we can replace $(n_k)$ and $(\ell_k)$ by subsequences such that condition~(E) holds. 
These subsequences also satisfy (P), (R), and~(C). 
From the Pisot condition and Lemma~\ref{lem:sadic87}, we obtain that almost all $\bsigma \in D$ are algebraically irreducible, i.e., also (I) holds a.e.\ and we are done. 
\end{proof}

With Proposition~\ref{prop:combcond} at our disposal, we can use a slight variation of \cite[Theorem~3.1]{BST:19} to show without much effort that under the conditions of Theorem~\ref{theo:main} the following is true: For almost all $\bsigma=(\sigma_n) \in D$, the dynamical system $(X_{\bsigma},\Sigma)$ has an $m$-to-$1$ factor which is a minimal translation on~$\TT^{d-1}$ for some $m \in\NN$. 
However, in order to prove Theorem~\ref{theo:main}, we have to show that $m=1$, i.e., that $(X_{\bsigma},\Sigma)$ is measurably conjugate to a minimal translation on~$\TT^{d-1}$, which is way more difficult. Indeed, to prove this, according to Proposition~\ref{p:gccBST} and \cite[Theorem~3.1]{BST:19}, one has to verify the (effective version of the) geometric coincidence condition for a.e.\ element of~$D$.\footnote{Recall that even in the substitutive case, it is not known if the geometric coincidence is always fulfilled.}
%A priori, we need to verify geometric coincidence for a.a.\  $\bsigma \in D$ to gain purely discrete spectrum for a.a.\ elements of $D$. 
This would require tedious combinatorial verifications: 
By interpreting geometric coincidence geometrically (as indicated by its name),
this was done for some instances  in the case of three-letter alphabets in \cite{BBJS14} by using the dual $E_1^*(\sigma_{[0,n)})$ of the one-dimensional geometric realization of $\sigma_{[0,n)}$ for growing $n$. As recalled in the introduction,  this  requires  both combinatorial and geometric arguments relying on planar  topology, which restricts the scope of application of   such   methods to the case of three-letter alphabets.
 In the present paper, we use an ergodic argument to simplify this decisively, allowing  us to consider general alphabets,  and we show that it suffices to check the condition on the Pisot point in the statement of Theorem~\ref{theo:main}.

The idea behind this ergodic argument is as follows. 
The geometric coincidence condition \eqref{e:gcc3} is satisfied for a given directive sequence $\bsigma \in D$ if certain sets defined in terms of balls of arbitrarily large radius~$R$ are contained in sets that are defined by the combinatorics of~$\bsigma$. 
According to the effective version of the geometric coincidence condition \eqref{e:gcc4}, it is even sufficient to consider balls with a radius chosen in terms of certain properties of  balance of languages related to~$\bsigma$.
By the assumptions of Theorem~\ref{theo:main}, there exists a substitutive system $(X_\tau,\Sigma)$, with $\tau=\tau_{[0,j)}$ and $(\tau_0,\ldots, \tau_{j-1})^\infty \in D$,  which has purely discrete spectrum and, hence, by Lemma~\ref{rem:spectrumEquiv} satisfies the geometric coincidence condition \eqref{e:gcc3} for balls of arbitrarily large radii~$R$. 
After a technical preparation contained in Lemma~\ref{lem:tau}, in  Lemma~\ref{lem:sadic7.9crit} we show that this has the following consequence: Each $S$-adic dynamical system whose directive sequence $\bsigma = (\sigma_n)$ has Property  PRICE and contains a sufficiently long block $(\sigma_n,\ldots,\sigma_{n+\ell-1})$ with $\sigma_{[n,n+\ell)}=\tau^m$ (i.e., $m$ is sufficiently large) followed by some tail $\Sigma^{n+\ell}\bsigma \in B_C$, satisfies the effective version \eqref{e:gcc4} of geometric coincidence condition. 
Informally speaking, in $\bsigma$ we need a sufficiently long block consisting of the repetition of a given substitution that satisfies the coincidence condition \eqref{e:gcc3}, which is followed by a tail that  is ``balanced enough'', to guarantee the coincidence condition for the whole sequence~$\bsigma$.
Using the Poincar\'e Recurrence Theorem, we are able to show that almost all directive sequences in~$D$ contain such a block. 
This will finally imply Theorem~\ref{theo:main}.

\begin{lemma}\label{lem:tau}
Let $\tau$ be a unimodular Pisot substitution with geometric coincidence. 
Then for each $C>0$ there are $m = m_\tau(C) \in \NN$, $\bz \in \bone^\perp$, and $i \in \cA$ such that for each $\bt \in \RR^d_{\ge0} \setminus \{\mathbf{0}\}$ we have
\begin{equation} \label{e:gcct}
\begin{split}
\big\{\by \in \ZZ^d \,: & \ \Vert \pi_{\bt} M_\tau^{-m} \by - \bz\Vert_\infty \le C, \, 0 \le \langle\bone,\by \rangle < |\tau^m(j)|\big\} \\
\subset \{\bl(p) \,: & \ p\in \cA^*,\, p\, i \preceq \tau^m(j)\}
\end{split} \quad \mbox{for all}\ j \in \cA.
\end{equation}
\end{lemma}

\begin{remark}
If we look at the definition of the effective version of geometric coincidence in \eqref{e:gcc4}, the lemma states that the inclusion in this definition still holds if we replace $\pi_{\bu^{(n)}}$ by an arbitrary projection~$\pi_{\bt}$ with some nonnegative vector~$\bt$. 
Indeed, because the elements $M_\tau^{-m} \by$ that are projected are close to a hyperplane that is ``sufficiently orthogonal'' to~$\bt$ and~$\bone$, this projection does not change these vectors too much.
\end{remark}

\begin{proof}
Since $\tau$ satisfies the geometric coincidence condition, there exist, for each $R > 0$ and sufficiently large $m \in\NN$, some $i \in \cA$ and  $\bz' \in M_\tau^{-m} \bone^\perp = (\tr{\!M}_\tau^m\bone)^\perp$, such that 
\begin{equation} \label{e:gcc}
\begin{split}
 \{\by \in \ZZ^d \,: & \ \|M_\tau^{-m} \by -  \bz'\|_\infty \le R,  \,0\le \langle\bone, \by\rangle < |\tau^m(j)|\} \\
\subset \{\bl(p) \,: & \ p\in \cA^*,\, p\, i \preceq \tau^m(j)\}
\end{split} \quad \mbox{for all}\ j \in \cA.
\end{equation}

Since $\tr{\!M}_\tau^m\bone/\|\tr{\!M}_\tau^m\bone\|$ converges to a dominant eigenvector of~$\tr{\!M}_\tau$ which is positive, there exists a constant $c_1 > 0$ such that $\|\bx\|_\infty \le c_1 \|\pi_{\bt} \bx\|_\infty$ for all $\bt \in \RR^d_{\ge 0} \setminus \{\mathbf{0}\}$, $\bx \in (\tr{\!M}_\tau^m\bone)^\perp$, $m \in \NN$.
Let $\tilde{\pi}_{\bt,m}$ denote the projection along~$\bt$ onto $(\tr{\!M}_\tau^m\bone)^\perp$.
There is another constant $c_2 > 0$ such that $\|\bx - \tilde{\pi}_{\bt,m} \bx\|_\infty \le c_2$ for all $\bt \in \RR^d_{\ge 0} \setminus \{\mathbf{0}\}$, $\bx \in \RR^d$ with $0 \le \langle \tr{\!M}_\tau^m\bone, \bx\rangle < \max_{j\in\cA} \langle \tr{\!M}_\tau^m\bone, \be_j\rangle = \max_{j\in\cA} |\tau^m(j)|$, $m \in \NN$.
To see this, note that $\langle \tr{\!M}_\tau^m\bone, \bx\rangle < \max_{j\in\cA} \langle \tr{\!M}_\tau^m\bone, \be_j\rangle$ says that the orthogonal distance between $\bx$ and $(\tr{\!M}_\tau^m\bone)^\perp$ is smaller than the maximum of the orthogonal distances between $\be_j$ and $(\tr{\!M}_\tau^m\bone)^\perp$. This implies that the same is true for the corresponding distances ``along $\bt$'', i.e., $\|\bx - \tilde{\pi}_{\bt,m} \bx\|_\infty \le \max_{j\in\cA} \|\be_j-\tilde{\pi}_{\bt,m} \be_j\|_\infty$ and we can take $c_2=\max_{m'\in\NN}\max_{j\in\cA} \|\be_j-\tilde{\pi}_{\bt,m'} \be_j\|_\infty$, which is finite because $\tr{\!M}_\tau^m\bone/\|\tr{\!M}_\tau^m\bone\|$ converges to a positive dominant eigenvector of~$\tr{\!M}_\tau$.
%\deleted[id=J]{Here, we use that $\|\bx - \tilde{\pi}_{\bt,m} \bx\|_\infty \le \max_{j\in\cA} \|\tilde{\pi}_{\bt,m} \be_j\|_\infty$ for these~$\bx$.}
Therefore, we have 
\[
\|M_\tau^{-m} \by - \bz'\|_\infty \le \|\tilde{\pi}_{\bt,m} M_\tau^{-m} \by - \bz'\|_\infty + c_2 \le c_1 \|\pi_{\bt} (M_\tau^{-m} \by -  \bz') \|_\infty + c_2
\]
for all $\by \in \ZZ^d$, $\bz' \in (\tr{\!M}_\tau^m\bone)^\perp$ with $0\le \langle\bone,\by \rangle < \max_{j\in\cA} |\tau^m(j)|$.
Choosing $m=m_\tau(C)$ such that \eqref{e:gcc} holds for $R = c_1 C+c_2$ and some $\bz' \in \bone^\perp$, $i \in \cA$, we obtain that \eqref{e:gcct} holds with $\bz = \pi_{\bt}\bz'$. 
\end{proof}

Let $\tau$ be a unimodular Pisot substitution that satisfies geometric coincidence. We now prove geometric coincidence for directive sequences $\bsigma = (\sigma_n)$ containing a long block $(\sigma_n,\ldots,\sigma_{n+\ell-1})$ satisfying $\sigma_{[n,n+\ell)} = \tau^m$ followed by a tail $\Sigma^{n+\ell}\bsigma \in B_C$. Indeed, this constellation will allow us to apply Lemma~\ref{lem:tau} in order to fulfill the effective version of the geometric coincidence condition for $\Sigma^{n+\ell}\bsigma$. Thus $\Sigma^{n+\ell}\bsigma$ gives rise to tilings which will lead to the desired conclusion.

\begin{lemma} \label{lem:sadic7.9crit}
Let $\tau$ be a unimodular Pisot substitution that satisfies geometric coincidence.
Let $\bsigma = (\sigma_n)$ be a sequence satisfying  Property PRICE with $C > 0$ chosen in a way that there are $\ell,n \in \NN$ such that, for $m = m_\tau(C)$ as in Lemma~\ref{lem:tau}, we have $\sigma_{[n,n+\ell)} = \tau^m$ and $\Sigma^{n+\ell} \bsigma \in B_C$. 
Then $\cC_{\bsigma}$ forms a tiling of~$\bone^\perp$.
\end{lemma}

\begin{proof}
Let $\bu$ be a generalized right eigenvector of~$\bsigma$.
Then $\bu^{(n)} = M_{\sigma_{[0,n)}}^{-1} \bu$ is a generalized right eigenvector of~$\Sigma^n \bsigma$. 
Since $\bsigma$ satisfies Property PRICE, $\Sigma^n \bsigma$ also satisfies Property PRICE by \cite[Lemma~5.10]{BST:19}.  
We want to prove that $\Sigma^n \bsigma$ satisfies \eqref{e:gcc4}. 
To this end, we apply Lemma~\ref{lem:tau} to~$\tau$ and $\bt = \bu^{(n+\ell)}$. 
Since $\sigma_{[n,n+\ell)} = \tau^m$, this yields that
\[
\begin{split}
& \{\by \in \ZZ^d \,:\, \| \pi_{\bu^{(n+\ell)}} M_{\sigma_{[n,n+\ell)}}^{-1}\by - \bz \|_\infty  \le  C,  \,0\le \langle\bone,\by \rangle < |\sigma_{[n,n+\ell)}(j)|\} \\
& = \{\by \in \ZZ^d \,:\,  \| \pi_{\bt} M_\tau^{-m}\by - \bz  \|_\infty  \le  C,  \,0\le \langle\bone,\by \rangle < |\tau^m(j)|\} \\
& \subset \{\bl(p) \,:\, p\in \cA^*,\, p\, i \preceq \sigma_{[n,n+\ell)}(j)\}
\end{split} \quad \mbox{for all}\ j \in \cA.
\]
Thus all conditions of Proposition~\ref{p:gccBST}~(\ref{i:gcc4}), in particular \eqref{e:gcc4}, are satisfied by $\Sigma^n \bsigma$, hence, by Proposition~\ref{p:gccBST} each of the two collections $\cC_{\Sigma^n\bsigma}$ and~$\cC_{\bsigma}$ forms a tiling of~$\bone^\bot$.
\end{proof}

We are now in a position to prove Theorem~\ref{theo:main}. 
Indeed, we use the Poincar\'e Recurrence Theorem together with the ergodicity of $\nu$  in order to show that under the conditions of Theorem~\ref{theo:main}, Lemma~\ref{lem:sadic7.9crit} can be applied to almost all directive sequences $\bsigma \in D$.

\begin{proof}[Conclusion of the proof of Theorem~\ref{theo:main}]
According to the assumptions of Theorem~\ref{theo:main}, there is a sequence $(\tau_n) \in D$ with period~$k$ and positive range such that $\tau = \tau_{[0,k)}$ is a Pisot substitution and the substitutive dynamical system $(X_\tau,\Sigma)$ has purely discrete spectrum. Lemma~\ref{rem:spectrumEquiv} implies that $\tau$ satisfies the geometric coincidence condition. 
By Lemma~\ref{l:nuBC} and the positive range of $(\tau_n)$, there is $C \in \NN$ such that 
\[
\nu(D \cap B_C) > 1 - \inf_{n\in\NN} \nu(\Sigma^n [\tau_0,\dots,\tau_{n-1}]).
\]
This yields that $\nu(\Sigma^{n}[\tau_0,\dots,\tau_{n-1}] \cap B_C) > 0$ and, since $\nu \circ \Sigma^n \ll \nu$, $\nu([\tau_0,\dots,\tau_{n-1}]  \cap \Sigma^{-n}B_C) > 0$  for all $n \in \NN$.
Choose $m = m_\tau(C)$ as in Lemma~\ref{lem:tau}.  
By the Poincar\'e Recurrence Theorem and the ergodicity of~$\nu$,  for almost all sequences $\bsigma \in D$, there exists~$n$ such that $\Sigma^n \bsigma \in [\tau_0,\dots,\tau_{km-1}] \cap \Sigma^{-km} B_C$, which is equivalent to the conditions $\sigma_{[n,n+\ell)} = \tau^m$ and $\Sigma^{n+\ell} \bsigma \in B_C$ in the formulation of Lemma~\ref{lem:sadic7.9crit}. 
Thus, since Property PRICE holds for a.e.\ $\bsigma \in D$ by Proposition~\ref{prop:combcond}, Lemma~\ref{lem:sadic7.9crit} yields geometric coincidence for almost all $\bsigma \in D$. 
This implies that $\cC_{\bsigma}$ forms a tiling of~$\bone^\perp$. 
We may thus apply \cite[Proposition~8.5]{BST:19} to conclude that $(X_{\bsigma}, \Sigma, \mu)$ is conjugate to the translation by $\pi_{\bu}\be_i = \be_i-\bu$ on $\bone^\perp/\ZZ^d$ for all $i \in \{1,\ldots,d\}$, where $\bu$ is the generalized right eigenvector of~$\bsigma$ normalized by $\|\bu\|_1=1$.
Taking $i = d$ and omitting the $d$-th coordinate, we obtain that $(X_{\bsigma}, \Sigma, \mu)$ is conjugate to the translation by $-\pi'(\bu)=-\bu'$ on~$\TT^{d-1}$, thus also to the translation by~$\bu'$. 
In particular, $(X_{\bsigma}, \Sigma, \mu)$ has purely discrete measure-theoretic spectrum. It remains to prove that the shift $(X_{\bsigma},\Sigma)$ is a natural coding of~$R_{\bu'}$  w.r.t.\ the natural partition $\{-\cR'_{\bsigma}(1), \ldots, -\cR'_{\bsigma}(d)\}$. The required topological properties of the atoms of the natural partition are established in \cite[Theorem~3.1]{BST:19}. 
We then consider the action of the domain exchange from \cite[Proposition~8.4]{BST:19} on the pieces of the Rauzy fractal which gives $\cR_{\bsigma}(i)  + \be_i-\bu  \subset  \cR_{\bsigma}$, for $1 \le i \le d$.
This yields, after applying~$\pi'$, that $-\cR'_{\bsigma}(i)  -\be_i +\bu'  \subset -\cR'_{\bsigma}$ for $1 \le i < d$ and $-\cR'_{\bsigma}(d) +\bu'  \subset -\cR'_{\bsigma}$. Lastly, the fact that the intersection of cylinders from Definition~\ref{def:NC} consists always of a single point holds by \cite[Lemma~8.3]{BST:19}.
\end{proof}

\subsection{Proof of Theorem~\ref{theo:main2}} \label{sec:proof-theor-refth}
To prove Theorem~\ref{theo:main2}, we need to get rid of the condition on the existence of a periodic Pisot sequence with purely discrete spectrum present in Theorem~\ref{theo:main}. 
In other words, under the conditions of Theorem~\ref{theo:main2}, we have to provide an ``accelerated'' substitution with purely discrete spectrum (i.e., satisfying the geometric coincidence condition by Lemma~\ref{rem:spectrumEquiv}). 
This is the objective of Proposition~\ref{p:sigmatilde}, which, for any given unimodular Pisot matrix~$M$, provides a substitution with incidence matrix~$M^n$ (for some $n\ge 1$) having purely discrete spectrum. 

We start with two technical lemmas. 
Lemma~\ref{lem:balanced} recalls the classical connection between Pisot substitutions and balance; see Remark~\ref{rem:PeriodicPrice}. 
%(see e.g.\ 8{\cite[Theorem 1]{Adamdis}}, which states that  a substitution whose incidence matrix is a unimodular Pisot matrix is balanced). 
%\deleted[id=W]{Its proof relies on the geometric characterization of balance given in Proposition~\ref{prop:balance}.}
Moreover, Lemma~\ref{lem:word} recalls that  for any given integer vector~$\bx$ with nonnegative entries, there exists a word~$w$ with uniformly bounded balance (w.r.t.\ the direction of~$\bx$) whose abelianization satisfies~$\bl(w) = \bx$.

\begin{lemma}\label{lem:balanced}
Let $M$ be a unimodular Pisot matrix with dominant right eigenvector~$\bu$. 
There exists a constant $C > 0$ such that each substitution~$\sigma$ satisfying $M_\sigma = M^k$ for some $k \in \NN$ and
\begin{equation}\label{eq:prefixsigmabound2}
\max_{p\in\cA^* \,:\, p \preceq \sigma(i),\, i\in\cA} \|\pi_{\bu} \bl(p)\|_\infty < 2
\end{equation}
has $C$-balanced language~$L_\sigma$. 
\end{lemma}

\begin{proof}
Let $\sigma$ be a substitution satisfying the conditions indicated in the statement of the lemma. 
Let $n \in \NN$ be arbitrary but fixed and choose a prefix~$p$ of $\sigma^n(i)$ for some $i \in \cA$. 
Then we have $p = \sigma^{n-1}(p_{n-1}) \cdots \sigma(p_1) p_0$ for some prefixes $p_j$ of $\sigma(i_j)$, $i_j \in \cA$, with $\sigma(i_j) \in p_j i_{j-1} \cA^*$; thus 
\[
\bl(p) = M^{k(n-1)} \bl(p_{n-1}) + \cdots + M^k \bl(p_1) + \bl(p_0).
\]
Let $\bv$ be a dominant left eigenvector of~$M$, $\varrho < 1$ the maximal absolute value of the nondominant eigenvalues of~$M$ and $\tilde{\pi}_{\bu}$ the projection along~$\bu$ on~$\bv^\perp$. 
Then we have a constant $c_1 > 0$ such that $\|M^\ell \bx\|_\infty \le c_1 \varrho^\ell \|\bx\|_\infty$ for all $\ell \in \NN$, $\bx \in \bv^\perp$.
Thus we have $\|\tilde{\pi}_{\bu} M^\ell\bx\|_\infty = \|M^\ell \tilde{\pi}_{\bu} \bx\|_\infty \le c_1 \varrho^\ell \|\tilde{\pi}_{\bu} \bx\|_\infty$ for all $\bx \in \RR^d$, hence 
\begin{equation}\label{eq:prefixsigmaboundA}
\|\tilde{\pi}_{\bu} \bl(p)\|_\infty < \frac{c_1}{1-\varrho^k} \max_{q\in\cA^*:\, q \preceq \sigma(i),\, i\in\cA} \|\tilde{\pi}_{\bu} \bl(q)\|_\infty.
\end{equation}
There is a constant $c_2 > 0$ such that $\|\pi_{\bu} \bx\|_\infty \le c_2 \|\bx\|_\infty$ for all $\bx \in \bv^\perp$ and $\|\tilde{\pi}_{\bu} \bx\|_\infty \le c_2 \|\bx\|_\infty$ for all $\bx \in \bone^\perp$.
Thus \eqref{eq:prefixsigmaboundA} and \eqref{eq:prefixsigmabound2} yield
\[
\begin{split}
\|\pi_{\bu} \bl(p)\|_\infty &= \|\pi_{\bu} \tilde{\pi}_{\bu} \bl(p)\|_\infty \le c_2 \|\tilde{\pi}_{\bu} \bl(p)\|_\infty
< \frac{c_1c_2}{1-\varrho^k} \max_{q\in\cA^*:\, q \preceq \sigma(i),\, i\in\cA} \|\tilde{\pi}_{\bu} \bl(q)\|_\infty \\
&= \frac{c_1c_2}{1-\varrho^k} \max_{q\in\cA^*:\, q \preceq \sigma(i),\, i\in\cA} \|\tilde{\pi}_{\bu}\pi_{\bu} \bl(q)\|_\infty 
\le  \frac{c_1c_2^2}{1-\varrho^k} \max_{q\in\cA^*:\, q \preceq \sigma(i),\, i\in\cA} \|\pi_{\bu} \bl(q)\|_\infty 
< \frac{2c_1c_2^2}{1-\varrho^k}.
\end{split}
\]
If $v\in L_\sigma$, then $v$ is a factor of $\sigma^n(i)$ for some $n\in \NN$, $i\in \cA$. Thus there are two prefixes $p_1,p_2$ of $\sigma^n(i)$ such that $p_1v=p_2$ and, hence,
$\|\pi_{\bu} \bl(v)\|_\infty \le 
\|\pi_{\bu} \bl(p_1)\|_\infty+\|\pi_{\bu} \bl(p_2)\|_\infty < \frac{4c_1c_2^2}{1-\varrho^k}$. Moreover, for two factors $v_1,v_2$ with $|v_1| = |v_2|$, we have
\[
\| \bl(v_1)-\bl(v_2) \|_\infty = \|\pi_{\bu} \bl(v_1)-\pi_{\bu} \bl(v_2)\|_\infty \le \|\pi_{\bu} \bl(v_1)\|_\infty + \|\pi_{\bu} \bl(v_2)\|_\infty \le \frac{8c_1c_2^2}{1-\varrho^k},
\]
thus $L_\sigma$ is $C$-balanced with $C = \frac{8c_1c_2^2}{1-\varrho^k}$. %\deleted[id=W]{according to Proposition~\ref{prop:balance}}
\end{proof}

\begin{lemma} \label{lem:word}
Let $\bx \in \NN^d$. 
Then there exists a word $w \in \cA^*$ such that $\bl(w) = \bx$ and  $\|\pi_{\bx} \bl(p)\|_\infty \le 1-\frac{1}{2d-2}$ for $p\preceq w$. 
Moreover, $w$ starts with the letter corresponding to the largest coordinate of~$\bx$.
\end{lemma}

\begin{proof}
This is proved in \cite{Meijer:73,Tijdeman:80}.
\end{proof}
 
The construction of the desired substitution is contained in the following proposition.
 
\begin{proposition}\label{p:sigmatilde}
Let $M$ be a nonnegative unimodular Pisot matrix. 
Then there exists a substitution~$\sigma$ with incidence matrix~$M_\sigma$ satisfying  $M_\sigma= M^n$ for some $n \in \NN$ such that the geometric coincidence condition holds. 
Moreover, we can choose $\sigma$ in a way that $\sigma(i) \prec \sigma(j)$ or $\sigma(j) \prec \sigma(i)$ for all $i,j \in \cA$. 
\end{proposition}

\begin{proof}
Let $\bu$ be a dominant right eigenvector of~$M$.
We construct a substitution~$\sigma$ using the set
\[
P = \{\by \in \ZZ^d \,:\, \|\pi_{\bu} \by\|_\infty \le C,\, 0 \le \langle \bone, M^n \by \rangle \le \max_{i\in\cA} \langle \bone, M^n \be_i \rangle\ \mbox{for some}\ n \in\NN\},
\]
with $C$ as in Lemma~\ref{lem:balanced}.
Note that $P$ is a finite set since $\langle \bone, M^n \by \rangle = \langle \tr{\!M}^n \bone, \by \rangle$ and $\bu \in \RR_+^d$.
Write $P = \{\by_\ell:\, 0 \le \ell \le L\}$ such that $0 = \langle \bu, \by_0\rangle < \langle \bu, \by_1\rangle < \cdots < \langle \bu, \by_L\rangle$; this is possible since $\bu$ has rationally independent coordinates.
Then for $n \in \NN$ large enough we have 
\begin{equation}\label{eq:le13}
\|\pi_{\bu} M^n \by\|_\infty \le \frac13 \qquad\mbox{for all $\by \in P$} 
\end{equation}
and $M^n (\by_{\ell+1} - \by_\ell)  \in \NN^d$ for all $0 \le \ell < L$. 
Let words~$w_\ell$ be given by Lemma~\ref{lem:word} with $\bx=\bx_\ell = M^n (\by_{\ell+1} - \by_\ell)$ for $0 \le \ell < L$, and let $L_j$, $j \in \cA$, be such that $\by_{L_j} = \be_j$. (Note that $\be_j \in P$ since $C \ge 1$.) 
Define the substitution $\sigma$ by $\sigma(j) = w_0 w_1 \cdots w_{L_j-1}$ for all $j \in \cA$.
Note that $\sigma(i)$ is a prefix of $\sigma(j)$ if and only if $\langle \bu, \be_i\rangle < \langle \bu, \be_j\rangle$. 

To show that $L_\sigma$ is $C$-balanced, consider $p \preceq \sigma(j)$ for some $j \in \cA$. 
Then $p = w_{0}\cdots w_{\ell-1} p'$ for some $0 \le \ell < L$, $p' \preceq w_\ell$. (Here, $w_{0}\cdots w_{\ell-1}$ is the empty word for $\ell=0$.)
This yields
\[
\begin{split}
\|\pi_{\bu} \bl(p)\|_\infty & = \| \pi_{\bu}\by_{\ell} + \pi_{\bu} \bl(p') \|_\infty  \le \| \pi_{\bu}\by_{\ell} \|_\infty + \| \pi_{\bu} \bl(p') \|_\infty \\
& \le \| \pi_{\bu}\by_{\ell} \|_\infty + \| \pi_{\bu}M^n (\by_{\ell+1} - \by_{\ell}) \|_\infty + \| \pi_{M^n (\by_{\ell+1} - \by_{\ell})} \bl(p') \|_\infty < 2,
\end{split}
\]
where the last inequality follows from \eqref{eq:le13} and Lemma~\ref{lem:word}.
Therefore, Lemma~\ref{lem:balanced} gives that $L_\sigma$ is $C$-balanced.
It remains to show that the constant sequence $(\sigma)$ satisfies the effective version of the geometric coincidence condition \eqref{e:gcc4}. 

By the construction of~$\sigma$, we have $M_\sigma = M^n$ and
\begin{equation}\label{eq:unionGcc}
\begin{split}
& \{\by \in \ZZ^d \,:\, \|\pi_{\bu}  M_\sigma^{-1} \by\|_\infty \le C,\, 0 \le \langle \bone, \by \rangle < |\sigma(j)|\} \\
& = \bigcup_{j'\in \cA}\{\bl(w_0 \cdots w_\ell) \,:\, 0\le \ell < L_{j'}-1 \} \subset \bigcup_{i\in\cA} \{\bl(p) \,:\, p \in \cA^*,\, p\, i \preceq \sigma(j)\}
\end{split}
\end{equation}
for all $j\in\cA$. 
Let $i_0 \in \cA$ be chosen in a way that $\langle \bu, \be_{i_0}\rangle=\max_{j\in\cA}\langle \bu, \be_{j}\rangle$.
Then the $i_0$-th coordinate of $\bx_\ell$ is the largest one for each $0 \le \ell < L$ if $n$ is chosen large enough.
%In the construction of \cite{Tijdeman:80}, the word $w = w_\ell$ in Lemma~\ref{lem:word} starts with the letter~$i_0$ for each $\bx = \bx_\ell$ ($1\le \ell < L$). 
%Therefore, we can choose~$n$ large enough such that all words~$w_\ell$, $0 \le \ell < L$, start with~$i_0$. 
Since we defined the words $w_\ell$ by Lemma~\ref{lem:word}, this means that $w_\ell$ starts with~$i_0$ for each $0 \le \ell < L$ if $n$ is chosen large enough, and 
we can sharpen the inclusion in \eqref{eq:unionGcc} to 
\[
\bigcup_{j'\in \cA}\{\bl(w_0\cdots w_\ell) \,:\, 0\le \ell < L_{j'} -1\} \subset  \{\bl(p) \,:\, p \in \cA^*,\, p\, i_0 \preceq \sigma(j)\}
\]
for all $j\in\cA$. Together with \eqref{eq:unionGcc}, this yields
\begin{gather*}
\{\by \in \ZZ^d :\, \|\pi_{\bu}  M_\sigma^{-1} \by\|_\infty \le C,\, 0 \le \langle \bone, \by \rangle < |\sigma(j)|\} 
\subset  \{\bl(p) \,:\, p \in \cA^*,\, p\, i_0 \preceq \sigma(j)\}
\end{gather*}
for all $j\in\cA$.
Therefore, $\sigma$ satisfies the effective version of the geometric coincidence condition, and, hence, by Proposition~\ref{p:gccBST}, also the geometric coincidence condition.
\end{proof}

\begin{remark}
To prove the main statement of Proposition~\ref{p:sigmatilde}, we could also have used the condition from \cite[Corollary~2]{Barge:16} to check geometric coincidence.\footnote{In \cite{Barge:16} the author deals with tiling flows, however, as we saw in the proof of Lemma~\ref{rem:spectrumEquiv}, this makes no difference, see also~\cite[Theorem~3.1]{CS:03}.}
This condition requires that the last letter of $\sigma(i)$ is equal for all $i \in \cA$ and the first letter of $\sigma(i)$ is different from the first letter of $\sigma(j)$ if $i\ne j$. 
If $M$ is a unimodular Pisot matrix, then it is also primitive and thus there is $n\in\NN$ such that $M^n$ is a positive matrix. By this positivity, there is clearly a substitution $\sigma$ with incidence matrix $M^n$ having this property. Because our proof is elementary and much shorter than the proof of \cite[Corollary~2]{Barge:16}, we decided to give a direct proof.
%However, since \cite{Barge:16} deals with an $\RR$-action which is a suspension of the shift~$\Sigma$, a~some more detailed discussion (like the one contained in \cite{Barge-Kwapisz:06}) would be needed to adapt the results of \cite{Barge:16} to our setting. 
\end{remark}

We can now finish the proof of Theorem~\ref{theo:main2}. 
The idea is to  provide a suitable substitutive realization in the same flavor as the substitutive realizations associated with multidimensional continued fraction algorithms from Section~\ref{sec:s-adic-shifts}. 
Analogously to compositions of substitutions, in the sequel we will use the notation $M_{[k,n)} = M_k \cdots M_{n-1}$ for products of matrices.

\begin{proof}[Proof of Theorem~\ref{theo:main2}]
Let $(\fD,\Sigma,Z,\nu)$ be as in the statement of Theorem~\ref{theo:main2}. 
Then for some $k>0$ there is a sequence $(\tilde{M}_n) \in \fD$ with period $k$ and positive range such that $\tilde{M}_{[0,k)}$ is a Pisot matrix. 
Since $\tilde{M}_{[0,k)}$ and $\tilde{M}_{[i,i+k)}=\tilde{M}_{[0,i)}^{-1}\tilde{M}_{[0,k)}\tilde{M}_{[0,i)}$ are similar matrices, $\tilde{M}_{[i,i+k)}$ is a Pisot matrix for each $0 \le i < k$. 
By Proposition~\ref{p:sigmatilde}, there is a substitution~$\tau_i$ with incidence matrix $M_{\tau_i} = \tilde{M}_{[i,i+k)}$ satisfying the geometric coincidence condition (replace $k$ by $km$ for some $m \in \NN$ if necessary). 
We choose $\tau_i$ in a way that $\tau_i = \tau_j$ if $\tilde{M}_{[i,i+k)}=\tilde{M}_{[j,j+k)}$, $0\le i,j<k$. 

Choose a map $s: \cM_d^k \to \cS_d$ with the properties that 
\begin{itemize}
\item
the incidence matrix of $s(M_0,\dots,M_{k-1})$ is $M_{[0,k)}$ for all $(M_0,\dots,M_{k-1}) \in \cM_d^k$,
\item
$s(M_0,\dots,M_{k-1}) = s(M'_0,\dots,M'_{k-1})$ if $M_{[0,k)} = M'_{[0,k)}$,
\item
$s(M_0,\dots,M_{k-1}) = \tau_i$ if $M_{[0,k)} =\tilde{M}_{[i,i+k)}$ for some $0 \le i < k$.
\end{itemize}
Then the map
\begin{equation}\label{eq:bpsidef}
\bpsi:\, \fD \to \cS_d^\NN, \quad (M_n)_{n\in\NN} \mapsto \big(s(M_{kn},\dots,M_{kn+k-1})\big)_{n\in\NN}
\end{equation}
is well defined, and, setting $D=\bpsi(\fD)$ we have the commutative diagram
\[
\begin{tikzcd}
\fD \arrow[r, "\Sigma^k"]\arrow[d,"\bpsi"] & \fD \arrow[d, "\bpsi"] \\
D \arrow[r, "\Sigma"]& D
\end{tikzcd}
\]
The acceleration $\Sigma^k$ of $\Sigma$ may no longer be ergodic with respect to~$\nu$. 
Thus the system $ (D,\Sigma,\nu')$ may be nonergodic, with $\nu' = \nu \circ \bpsi^{-1}$. 
However, we will now show that $(D,\Sigma,\nu')$ can be partitioned into ergodic systems\footnote{These systems  correspond to sets of directive sequences
that may not be closed in~$D$. This is why we chose not to confine ourselves to closed  sets of directive sequences.} that satisfy the conditions of Theorem~\ref{theo:main}. 
Since all cylinders in~$D$ are measurable, $\nu'$~is a Borel probability measure on~$D$. 
Suppose that $(D,\Sigma,\nu')$ is not ergodic. 
Then there exists a $\Sigma$-invariant (up to measure zero) subset $\tilde{D} \subseteq D$ with $0 < \nu'(\tilde{D}) < 1$.
Then $\bpsi^{-1}(\tilde{D}) \subset \fD$ is $\Sigma^k$-invariant, hence $\bigcup_{i=0}^{k-1} \Sigma^{-i} \bpsi^{-1}(\tilde{D})$ is $\Sigma$-invariant and, by ergodicity of~$\nu$, equal to $\fD$ up to measure zero.
Therefore, we have $\nu'(\tilde{D}) = \nu(\bpsi^{-1}(\tilde{D}))  \ge 1/k$.
Since $D \setminus \tilde{D}$ is also $\Sigma$-invariant, we also have $\nu'(\tilde{D}) \le 1-1/k$. 
We repeat the argument until we have a measurable partition $\{D_1,\dots,D_\ell\}$ of~$D$, with $1 \le \ell \le k$, such that $\big(D_j,\Sigma,\frac{\nu'|_{D_j}}{\nu(D_j)}\big)$ is ergodic for all $1 \le j \le \ell$. 

Let now $j$ be fixed. 
We need to prove that, for some $0 \le i < k$, the constant (and hence periodic) Pisot sequence $(\tau_i)_{n\in\NN}$ has positive range in~$D_j$.
For all $0 \le i < k$, we have 
\begin{align}
\nu'\big(\Sigma^n[\underbrace{\tau_i,\ldots,\tau_i}_{n \hbox{ \scriptsize times}}] \cap D_j\big) & \ge
\nu\big(\Sigma^{kn} \big[\tilde{M}_i,\ldots,\tilde{M}_{i+kn-1}\big] \cap \bpsi^{-1}(D_j)\big) 
\nonumber\\
&= \nu\big(\Sigma^{-i} \Sigma^{kn} \big[\tilde{M}_i,\ldots,\tilde{M}_{i+kn-1}\big] \cap \Sigma^{-i}\bpsi^{-1}(D_j)\big) 
\label{eq:MtildeTau3}\\
& \ge \nu\big(  \Sigma^{kn} \big[\tilde{M}_0,\ldots,\tilde{M}_{i+kn-1}\big]\cap \Sigma^{-i} \bpsi^{-1}(D_j)\big).
\nonumber
\end{align}
Since
\[
\Sigma^i \bigcap_{n\in\NN} \Sigma^{kn}\big[\tilde{M}_0,\ldots,\tilde{M}_{i+kn-1}\big] = \bigcap_{n\in\NN} \Sigma^{kn+i}\big[\tilde{M}_0,\ldots,\tilde{M}_{i+kn-1}\big]
\]
and 
\begin{equation*}
 \nu\bigg(\bigcap_{n\in\NN} \Sigma^{kn+i}\big[\tilde{M}_0,\ldots,\tilde{M}_{i+kn-1}\big]\bigg)  = \inf_{n\in\NN} \nu\big(\Sigma^{kn+i}[\tilde{M}_0,\ldots,\tilde{M}_{i+kn-1}]\big) > 0
\end{equation*}
by the positive range of $(\tilde{M}_n)$, $\nu\circ\Sigma^i \ll \nu$ gives that 
\[
\nu\bigg(\bigcap_{n\in\NN} \Sigma^{kn}\big[\tilde{M}_0,\ldots,\tilde{M}_{i+kn-1}\big]\bigg) > 0.
\]
Therefore, by \eqref{eq:MtildeTau3} and because $\bigcup_{i=0}^{k-1}\Sigma^{-i} \bpsi^{-1}(D_j)=\fD$, there is $0 \le i < k$  such that 
\[
\inf_{n\in\NN} \nu'\big(\Sigma^n [\underbrace{\tau_i,\ldots,\tau_i}_{n \hbox{ \scriptsize times}}] \cap D_j\big) > 0.
\] 
Note that the constant sequence $(\tau_i)_{n\in\NN}$ may not be in~$D_j$, but the proof of Theorem~\ref{theo:main} also goes through for Pisot directive sequences with positive range that are not contained in~$D_j$ (but in the closure of~$D_j$ relative to~$D$).
Thus $(\tau_i)_{n\in\NN} \in D_j$ is a  periodic Pisot sequence having positive range in $\big(D_j,\Sigma,\frac{\nu'|_{D_j}}{\nu'(D_j)}\big)$ and purely discrete spectrum.
Since the cocycle~$Z$ satisfies the Pisot condition, the same is true for the cocycle $Z_j:\, D_j \to\cM_d$, $(\sigma_n) \mapsto \tr{\!}M_{\sigma_0}$.  
Summing up, we can apply Theorem~\ref{theo:main} to $\big(D_j,\Sigma,Z_j, \frac{\nu'|_{D_j}}{\nu'(D_j)}\big)$.
This proves the result.
\end{proof}

\subsection{Proofs of Theorems~\ref{theo:MCF} and~\ref{theo:MCF2}} \label{subsec:proofsMCF}
We now prove Theorems~\ref{theo:MCF} and~\ref{theo:MCF2} by reducing them to Theorems~\ref{theo:main} and~\ref{theo:main2} (see also Remark~\ref{rem:theo2}), respectively. 

\begin{proof}[Proof of Theorem~\ref{theo:MCF}] 
Recall that $(\Delta,T,A,\nu)$ is a positive $(d{-}1)$-dimensional continued fraction algorithm satisfying the Pisot condition and $\nu \circ T \ll \nu$, that $\bphi$ is a faithful substitutive realization of $(\Delta,T,A,\nu)$, and that there is a periodic Pisot point~$\bx_0$ such that $\bphi(\bx_0)$ has purely discrete spectrum and positive range  in $(\Delta,T,A,\nu)$.
Then we have $(\Delta,T,\nu) \overset{\bphi}{\cong} (\bphi(\Delta),\Sigma,\nu\circ\bphi^{-1})$, hence $\nu \circ \bphi^{-1}$ is an ergodic $\Sigma$-invariant Borel probability measure satisfying $\nu \circ \bphi^{-1} \circ \Sigma \ll \nu \circ \bphi^{-1}$, the linear cocycle $(\bphi(\Delta),\Sigma,Z,\nu\circ\bphi^{-1})$ defined by $Z((\sigma_n)_{n\in\NN}) = \tr{\!}M_{\sigma_0}$ satisfies the Pisot condition, and $\bphi(\bx_0)$ is a periodic Pisot sequence with purely discrete spectrum having positive range in $(\bphi(\Delta),\Sigma,\nu\circ\bphi^{-1})$.
Therefore, by Theorem~\ref{theo:main}, for $\nu$-almost all $\bx \in \Delta$, the $S$-adic dynamical system $(X_{\bphi(\bx)},\Sigma)$ is a natural coding of the minimal translation by $\pi'(\bu)$ on~$\TT^{d-1}$ with respect to the partition $\{-\cR'_{\bphi(\bx)}(i) \,:\, i \in \cA\}$ of the bounded fundamental domain $-\cR'_{\bphi(\bx)}$, where $\bu$ is the  generalized right eigenvector of~$\bphi(\bx)$ normalized by $\|\bu\|_1=1$.
Since $\bx$ is the generalized right eigenvector of~$\bphi(\bx)$ satisfying $\|\bx\|_1=1$, we have $\bx=\bu$, which proves Theorem~\ref{theo:MCF}.
\end{proof} 

Theorem~\ref{theo:MCF2} follows from Theorem~\ref{theo:main2} in the following way.

\begin{proof}[Proof of Theorem~\ref{theo:MCF2}] 
Recall that $(\Delta,T,A,\nu)$ is a positive $({d{-}1})$-dimensional continued fraction algorithm satisfying the Pisot condition and $\nu \circ T \ll \nu$, and that there is
a periodic Pisot point $\bx_0 \in \Delta$ having positive range in  $(\Delta,T,A,\nu)$. Define $\bseta: \Delta \to \cM_d^{\NN}$ by $\bx \mapsto (\tr{\!}A(T^n \bx))_{n\in\NN}$. 
Then we have $(\Delta,T,\nu) \overset{\bseta}{\cong} (\bseta(\Delta),\Sigma,\nu\circ\bseta^{-1})$, hence, $\nu \circ \bseta^{-1}$ is an ergodic $\Sigma$-invariant Borel probability measure satisfying $\nu \circ \bseta^{-1} \circ \Sigma \ll \nu \circ \bseta^{-1}$, the linear cocycle $(\bseta(\Delta),\Sigma,Z,\nu\circ\bseta^{-1})$ defined by $Z((M_n)_{n\in\NN}) = \tr{\!}M_0$ satisfies the Pisot condition, and $\bseta(\bx_0)$ has positive range  in $(\bseta(\Delta),\Sigma,\nu\circ\bseta^{-1})$.
Therefore, by Theorem~\ref{theo:main2}, there exists a positive integer~$k$ and a map $\bpsi: \bseta(\Delta) \to \cS_d^{\NN}$ (which we choose as in \eqref{eq:bpsidef}) satisfying $\bpsi \circ \Sigma^k = \Sigma \circ \bpsi$ such that, for $\nu$-almost all $\bx \in \Delta$, the $S$-adic dynamical system $(X_{\bpsi\circ\bseta(\bx)},\Sigma)$ is a natural coding of the minimal translation by $\pi'(\bx)$ on~$\TT^{d-1}$ with respect to a partition of a bounded fundamental domain.
Setting $\bphi = \bpsi \circ \bseta$, we obtain that the diagram
\[
\begin{tikzcd}
\Delta \arrow[r, "T^k"]\arrow[d,"\bseta"] \arrow[dd,swap,"\bphi",bend right=50] & \Delta \arrow[d, swap,"\bseta"] \arrow[dd,"\bphi",bend left=50] \\
\bseta(\Delta) \arrow[r, "\Sigma^k"]\arrow[d,"\bpsi"] & \bseta(\Delta) \arrow[d, swap,"\bpsi"] \\
\bphi(\Delta) \arrow[r, "\Sigma"]& \bphi(\Delta)
\end{tikzcd}
\]
commutes. 
Because we have chosen $\bpsi$ as in \eqref{eq:bpsidef}, $\bphi$ is a substitutive realization of $(\Delta,T^k,A,\nu)$ such that for $\nu$-almost all $\bx \in \Delta$ the $S$-adic dynamical system $(X_{\bphi(\bx)},\Sigma)$ is a natural coding of the minimal translation by $\pi'(\bx)$ on~$\TT^{d-1}$ with respect to the partition $\{-\cR'_{\bphi(\bx)}(i) \,:\, i \in \cA\}$ of the bounded fundamental domain $-\cR'_{\bphi(\bx)}$. 
This implies that $(X_{\bphi(\bx)},\Sigma)$ has purely discrete spectrum.
Since by construction, $\bx$~is a generalized right eigenvector of~$\bphi(\bx)$, the map $\bphi$ is injective, thus $(\Delta,T^k,\nu) \overset{\bphi}{\cong} (\bphi(\Delta),\Sigma,\nu\circ\bphi^{-1})$.
\end{proof} 

\subsection{Proof of Theorem \ref{t:nc} }\label{subsec:proofBRS}
We now establish the relation between a natural coding with $d$ atoms, bounded remainder sets, and Rauzy fractals asserted in Theorem~\ref{t:nc}. 
To this end, we need Lemma~\ref{l:strongconvergence} that states in a nutshell that   balance implies strong convergence. 
Like in Section~\ref{sec:cf}, strong convergence refers to the convergence of the column vectors $M_{[0,n)} \be_i $ towards multiples of the generalized right eigenvector~$\bu$, for a sequence   $\bsigma \in \cS_d^{\NN}$.
Lemma~\ref{l:strongconvergence} was proved in \cite[Proposition~4.3]{BST:19} with the additional assumption that $\bsigma$ is recurrent. 
We give a slightly simpler proof that does not require recurrence.
Recall that $\pi_{\bu}$ denotes the projection along~$\bu$ onto~$\bone^\perp$.

\begin{lemma} \label{l:strongconvergence}
Let $\bsigma \in \cS_d^{\NN}$. 
If $L_{\bsigma}$ is balanced and the generalized right eigenvector~$\bu$ of~$\bsigma$ has rationally independent coordinates, then $\lim_{n\to\infty} \pi_{\bu} M_{\sigma_{[0,n)}} \be_i = \mathbf{0}$ for all $i \in \cA$ and
\begin{equation} \label{e:Lsn}
\lim_{n\to\infty} \sup\big\{ \|\pi_{\mathbf{u}} M_{\sigma_{[0,n)}}\, \bl(w)\| \,:\,  w \in L_{\Sigma^n\bsigma}\big\} = 0.
\end{equation}
\end{lemma}

\begin{proof}
Assume that $\bsigma = (\sigma_n)_{n\in\NN} \in \cS_d^{\NN}$ has balanced language~$L_{\bsigma}$ and a generalized right eigenvector~$\bu$ with rationally independent coordinates.
We first show that $\bsigma$ is a primitive sequence of substitutions. 
Suppose that there exists $k \in \NN$ such that $M_{\sigma_{[k,n)}}$ is not positive for all $n > k$. 
Then there exist coordinates $i,j$ such that the $(i,j)$-element of $M_{\sigma_{[k,n)}}$ is $0$ for infinitely many~$n$, i.e., $M_{\sigma_{[k,n)}} \be_j \in \be_i^\perp$.
Since the cones $M_{\sigma_{[k,n)}} \RR_{\ge0}^d$ form a nested sequence of nonempty compact sets, their intersection is nonempty,  and we obtain that $ \be_i^\perp \cap \bigcap_{n\in\NN} M_{\sigma_{[k,n)}} \RR_{\ge0}^d  \ne \{\mathbf{0}\}$, thus $M_{\sigma_{[0,k)}} (\be_i^\perp) \cap \bigcap_{n\in\NN} M_{\sigma_{[0,n)}} \RR_{\ge0}^d  \ne \{\mathbf{0}\}$,  which implies that $\bu \in M_{\sigma_{[0,k)}} (\be_i^\perp)$, contradicting that $\bu$ has rationally independent coordinates.
Therefore, $\bsigma$ is primitive.

Choose a sequence $(i_n)_{n\in\NN} \in \cA^{\NN}$ such that $i_n \preceq \sigma_n(i_{n+1})$ for all $n \in \NN$, and let $\omega^{(n)}$ be such that $\sigma_{[n,\ell)}(i_\ell) \prec \omega^{(n)}$ for all $\ell > n$, i.e., $\omega^{(n)}$~is a so-called \emph{limit sequence} of $\Sigma^n \bsigma$. 
Set 
\[
P = \{w \in \cA^* \,:\, w \prec \omega^{(0)}\} \quad \mbox{and} \quad
P_n = \{w \in \cA^* \,:\, w \prec \sigma_{[0,n)}(i_n)\} \quad (j \in \cA,\, n \in \NN). 
\]
Since $\bsigma$ is balanced, the set $\pi_{\bu} \bl(P)$ is bounded by  Proposition~\ref{prop:balance}.
From $P_0 \subseteq P_1 \subseteq \cdots \subseteq \bigcup_{n\in\NN} P_n = P$, we obtain that there is a sequence of positive numbers $(\varepsilon_n)_{n\in\NN}$ with $\lim_{n\to\infty} \varepsilon_n = 0$ such that 
\begin{equation} \label{e:epsn}
\|\bx\| \le \varepsilon_n \quad \mbox{for all}\ \bx \in \bone^\perp \ \mbox{satisfying} \quad \bx + \pi_{\bu} \bl\big(P_n\big) \subseteq \pi_{\bu} \bl(P). 
\end{equation}

We can now show that $\pi_{\bu} M_{\sigma_{[0,n)}} \bl(Q_n)$ is small, where
\[
Q_n = \{w \in \cA^* \,:\, \mbox{$pj \prec \omega^{(n)}$ and $pwj \prec \omega^{(n)}$ for some $p \in \cA^*$, $j \in \cA$}\}
\]
is the set of return words in~$\omega^{(n)}$ to some letter.
More precisely, we have 
\begin{equation} \label{eq:epsK}
\|\pi_{\bu} M_{\sigma_{[0,n)}} \bl(w)\| \le 2\varepsilon_k \ \mbox{for all $w \in Q_n$, $k < n$, provided that $M_{\sigma_{[k,n)}}$ is a positive matrix}. 
\end{equation}
To prove this, let $w\in  Q_n$.
If $M_{\sigma_{[k,n)}}$ is a positive matrix (which holds for sufficiently large $n$ by the primitivity of~$\bsigma$) and $j \in \cA$, then there exists $v \in \cA^*$ with $v\, i_k \preceq \sigma_{[k,n)}(j)$. 
Because $w\in Q_n$, we have some $p \in \cA^*$ such that
\[
\sigma_{[0,n)}(p) \sigma_{[0,k)}(v) u \preceq \sigma_{[0,n)}(p) \sigma_{[0,k)}(v\, i_k) \preceq \sigma_{[0,n)}(p\,j) \prec \sigma_{[0,n)}(\omega^{(n)}) = \omega^{(0)} \quad \mbox{for all}\ u \in P_k,
\]
and the same holds when we replace $p$ by~$pw$, thus
\[
\pi_{\bu}\bl\big(\sigma_{[0,n)}(p) \sigma_{[0,k)}(v)\big) + \pi_{\bu} \bl\big(P_k\big) \subseteq \pi_{\bu} \bl(P) \ \mbox{and} \ \pi_{\bu}\bl\big(\sigma_{[0,n)}(pw) \sigma_{[0,k)}(v)) + \pi_{\bu} \bl\big(P_k\big) \subseteq \pi_{\bu} \bl(P).
\]
From \eqref{e:epsn}, we obtain that
\[
\|\pi_{\bu} M_{\sigma_{[0,n)}} \bl(w)\| \le \big\|\pi_{\bu}\bl\big(\sigma_{[0,n)}(p) \sigma_{[0,k)}(v)\big)\big\| + \big\|\pi_{\bu}\bl\big(\sigma_{[0,n)}(pw) \sigma_{[0,k)}(v))\big\| \le 2\varepsilon_k.
\]

Next we show that, for each $n \in \NN$, the Minkowski sum 
\begin{equation}\label{eq:MinkowskiBasis}
\bl(Q_n)-\sum_{j=1}^d\bl(Q_n) \ \mbox{contains a basis of $\RR^d$ with vectors in $\{0,1\}^d$}.
\end{equation} 
First note that $\bl(Q_n)$ contains a basis of~$\RR^d$ by the rational indepence of~$\bu$ and the balance of~$L_{\bsigma}$. 
If this was not the case then, since $\bl(Q_n) \subset \ZZ^d$, there would be $\bv^\perp \in \ZZ^d$ with $\bl(Q_n) \subset \bv^\perp$. 
However, such~$\bv$ cannot exist because $Q_n$ contains arbitrarily long factors of $\omega^{(n)}$, hence $M_{\sigma_{[0,n)}} \bl(Q_n)$ contains arbitrarily large vectors with bounded distance from~$\RR \bu$ (by the balance of~$L_{\bsigma}$), which implies that $\bu \in M_{\sigma_{[0,n)}} \bv^\perp$, contradicting that $\bu$ is rationally independent. Thus we may choose words $w_i \in Q_n$ such that $\{\bl(w_i): 1 \le i \le d\}$ forms a basis of~$\RR^d$. 
If $\bl(w_i) \in \{0,1\}^d$ for all~$i$, then we have found a basis of the required form because $\mathbf{0} \in \bl(Q_n)$. 
Otherwise note that each  nonempty factor~$w$ of $\omega^{(n)}$ can be written as
\begin{equation} \label{e:wdecomposition}
w = v_1 a_1 v_2 a_2 \cdots v_\ell a_\ell \quad \mbox{with}\ 1 \le \ell \le d,\, v_j \in Q_n,\,  a_j \in \cA \ \mbox{for all}\ 1 \le j \le \ell,\, a_j \ne a_k\ \mbox{if}\ j \ne k.
\end{equation}
Indeed, let $a_1$ be the first letter of~$w$ and $v_1$ the longest (possibly empty) word such that $v_1 a_1 \preceq w$; then $v_1 \in Q_n$ and $(v_1 a_1)^{-1} w$ has no occurrence of~$a_1$; if $w \ne v_1 a_1$, then let $a_2 \in \cA$ be the first letter of $(v_1 a_1)^{-1} w$ and $v_2$ the longest word such that $v_2 a_2 \preceq (v_1 a_1)^{-1} w$; repeat this procedure until $(v_1a_1\cdots v_\ell a_\ell)^{-1} w$ (which has no occurrences of $a_1,\dots,a_\ell$) is the empty word. 
Now, if $\bl(w_i) \notin \{0,1\}^d$ and $w_i = v_1 a_1 v_2 a_2 \cdots v_\ell a_\ell$, then we can replace $w_i$ by the shorter word~$v_j$ for some~$j$ or, when all $\bl(v_j)$ are in the span of the other basis vectors, we replace $\bl(w_i)$ by $\bl(w_i)-\sum_{j=1}^\ell\bl(v_j)$ without losing the basis property. 
Since $\bl(w_i)-\sum_{j=1}^\ell\bl(v_j) = \bl(a_1\cdots a_\ell) \in \{0,1\}^d$ and the replacement by a shorter word can happen only finitely many times, this proves \eqref{eq:MinkowskiBasis}. 

From \eqref{eq:epsK} and \eqref{eq:MinkowskiBasis} we see that, for each $n \in \NN$, there is a basis of $\RR^d$ with vectors $\bx \in \{0,1\}^d$ satisfying $\|\pi_{\bu} M_{\sigma_{[0,n)}} \bx\| \le 2(d{+}1)\varepsilon_k$ for all $k < n$ such that $M_{\sigma_{[k,n)}}$ is positive. 
In particular, we have the same basis for infinitely many~$n$, and obtain that $\lim_{n\to\infty} \pi_{\bu} M_{\sigma_{[0,n)}} \be_i = \mathbf{0}$ for all~$i \in \cA$. 

Finally, let $w \in L_{\Sigma^n\bsigma}$. 
By primitivity, $w$~is a factor of~$\omega^{(n)}$. 
Writing $w$ as in~\eqref{e:wdecomposition}, we obtain that $\|\pi_{\bu} M_{\sigma_{[0,n)}} \bl(w)\| \le 2d\varepsilon_k + \sum_{i=1}^d \|\pi_{\bu} M_{\sigma_{[0,n)}} \be_i\|$ for all $k < n$ such that $M_{\sigma_{[k,n)}}$ is positive. 
This proves the lemma.
\end{proof}

Before we start with the core part of the proof of Theorem~\ref{t:nc}, we need the following variant of a result of Chevallier~\cite{Chevallier}.

\begin{lemma}[{cf.\ \cite[Theorems~A and~B]{Chevallier}}]\label{lem:ncChev}
Let $(X,\Sigma)$ be a bounded natural coding of $(\TT^d,R_\bt)$ w.r.t.\ a natural partition $\{\cF_1,\dots,\cF_h\}$ of a fundamental domain $\cF$. Then there is a continuous surjective map $\chi: X \to \cF$ and a one-to-one coding map $\Phi$ defined a.e.\ on~$\cF$ that satisfies $\chi \circ \Phi(\bx) = \bx$ for a.e.\ $\bx \in \cF$. 
Furthermore, the shift  $(X, \Sigma)$ is minimal and uniquely ergodic, 
%\deleted{$(\TT^d,R_\bt)$ is a topological factor of $(X, \Sigma)$,}  
and $(\TT^d,R_\bt)$ is measure-theoretically isomorphic to $(X, \Sigma)$. Thus $(X, \Sigma)$ has purely discrete measure-theoretic spectrum.
\end{lemma}

\begin{proof}[Sketch of the proof]\label{rem:che}
The proof of this lemma is very similar to the proofs of {\cite[Theorems~A and~B]{Chevallier}} (which are valid for $\TT^d$ according to the remark after their statement), and we will refer to these proofs in the present sketch. Also observe that the aperiodicity condition of {\cite[Theorems~A and~B]{Chevallier}} is used at the beginning of the proof of \cite[Theorem~A]{Chevallier} just to ensure that for each $(i_0i_1\dots)\in X$ the set $\bigcap_{n\in\NN}\overline{\bigcap_{k=0}^n \tilde{R}_\bt^{-k} \mathring{\cF}_{i_k}}$ is either empty or a singleton. This is true by assumption in our setting. 

Define $\chi: X \to \cF$ by $(i_0i_1\cdots)\mapsto \bigcap_{n\in\NN}\overline{\bigcap_{k=0}^n \tilde{R}_\bt^{-k} \mathring{\cF}_{i_k}}$. This is well defined because the intersection is exactly one point by the definition of a natural coding. Continuity of $\chi$ is proved in the same way as in the proof of  {\cite[Theorem~A]{Chevallier}} (in this part of the proof, we need the natural coding to be bounded), surjectivity follows because $\chi(X)$ is dense in $\cF$ and compact. Thus $\chi(X)=\cF$ and $\cF$ is bounded. Clearly, also $\cF_1,\ldots, \cF_h$ are bounded. We can also define a map $\Phi$ for a.e. $\bx\in \cF$, more precisely, for every $\bx \in \mathcal{G}$ with
\[
\mathcal{G}=\cF \setminus \bigcup_{n\in\NN}\bigcup_{i=1}^h \tilde R_\bt^{-n}\partial \cF_i,
\]
 by associating with $\bx$ the natural coding of its orbit under the action~$\tilde{R}_\bt$ w.r.t.\ the natural partition $\{\cF_1,\dots,\cF_h\}$. Now $\chi \circ \Phi(\bx) = \bx$ for each $\bx \in \mathcal{G}$ follows from the definition of a natural coding.

Define $\chi': X \to \TT^d$ by $\chi'=\chi \pmod{\TT^d}$. Then $\chi'$ is obviously continuous and surjective. Following the proof of {\cite[Theorem~A]{Chevallier}}, we can show that $(\TT^d,R_\bt)$ and $(\cF, \tilde R_\bt)$ are topological factors of $(X, \Sigma)$, in particular, $R_\bt \circ\chi'=\chi' \circ\Sigma$ and $\tilde R_\bt \circ\chi=\chi \circ\Sigma$.

The proof of minimality of $(X, \Sigma)$ deviates a bit from Chevallier's proof, and we provide the details. Fix $\omega\in X$. We want to prove that the orbit of $\omega$ is dense in $X$. Let $\omega'=(i_0i_1 \cdots )$ be an arbitrary element of the set
$\Phi(\mathcal{G})$ (which is dense in $X$) and let $U=\left[i_0, \ldots, i_n\right]$ be a neighborhood of~$\omega'$ (for some $n\in \NN$). The open set $\mathcal{V}=\bigcap_{k=0}^{n} \tilde R_{\bt}^{-i_k} \mathring{\cF}_{i_k}$ is nonempty because $\bx=\chi(\omega')$ is in~$\mathcal{V}$. Since $R_{\bt}$ is minimal, there exists an integer $m \geq 0$ such that $R_{\bt}^{m}  \circ \chi(\omega)$ belongs to~$\mathcal{V}$.
But $\tilde R_\bt \circ\chi=\chi \circ\Sigma$ and, hence, $\chi \circ\Sigma^m(\omega)$ belongs to $\mathcal{V}$. Since $\{\cF_1,\dots,\cF_h\}$ is a natural partition, for each $i \leq k$, $\mathring{\cF}_{i}\cap \cF_{j}=\emptyset$. 
By the definition of~$\chi$, we have $\chi^{-1}(\mathcal{V}) \subset U$, and therefore $\Sigma^{m}(\omega) \in U$. So all the elements of $Z$ are limit points of the sequence $(\Sigma^{m}(\omega))_{m \geq 0}$, which shows that $(X, \Sigma)$ is minimal.

The remaining assertions are of a measure theoretic nature. Thus for the proofs of these assertions it is immaterial that the partition $\{\cF_1,\dots, \cF_h\}$ is measure theoretic. They follow by rephrasing the proof of {\cite[Theorem~B]{Chevallier}} {\it verbatim}.
\end{proof}

\begin{proof}[Proof of Theorem~\ref{t:nc}]
Let $(X,\Sigma)$ be the natural coding of the minimal translation~$R_{\bt}$ on~$\TT^{d-1}$ w.r.t.\ the natural partition $\{\cF_1,\ldots, \cF_d\}$ of the bounded fundamental domain~$\cF$. 
We consider the associated exchange of domains~$\tilde{R}_{\bt}$ defined on~$\cF$; see Section \ref{sec:natur-codings-bound}.
Let $\bt_i$ be such that $\tilde{R}_{\bt}(\bx) = \bx + \bt_i$ on~$\cF_i$ (note that $\bt_i-\bt \in \ZZ^d$), and let $\bu = (u_1,\dots,u_d)$ with $u_i = \lambda(\cF_i)$, where  $\lambda$ denotes the Lebesgue measure. 
Then we have $\sum_{i=1}^d u_i = 1$.
Since $\cF$ is bounded and $(\cF,\tilde{R}_{\bt},\lambda|_\cF)$ is ergodic, we have for almost all $\bx \in \cF$, by the Birkhoff Ergodic Theorem,
\[
\sum_{i=1}^d u_i \bt_i = \sum_{i=1}^d \bt_i \int_{\cF_i} \mathrm{d}\lambda = \lim_{n\to\infty} \frac{1}{n} \sum_{k=0}^{n-1} \big(\tilde{R}_{\bt}^{k+1}(\bx) - \tilde{R}_{\bt}^k(\bx)\big) = \lim_{n\to\infty} \frac{1}{n} \big(\tilde{R}_{\bt}^n(\bx) - \bx\big) = \mathbf{0}.
\]

Define a matrix $N \in \RR^{(d-1)\times d}$ by $N \be_i = \bt_i$, i.e., the columns of~$N$ are the vectors~$\bt_i$. 
Then we have $N \bu = \mathbf{0}$ and, by the minimality of~$\tilde{R}_{\bt}$, the vectors~$\bt_i$ span~$\RR^{d-1}$, thus the kernel of~$N$ is~$\RR \bu$.
Hence we have $\|\bx\|_\infty \le c\, \|N \bx\|_\infty$ for all $\bx \in \bone^\perp$, with $1/c = \min\{\|N \bx\|_\infty \,:\, \bx\in \bone^\perp, \|\bx\|_\infty = 1\} > 0$. If $w$ is in the language of~$X$, then $N \bl(w) = \sum_{i=1}^d |w|_i \bt_i  = \tilde{R}_{\bt}^{|w|}(\bx) - \bx$ for some $\bx \in \cF$, thus $\|N \bl(w) \|_\infty \le \mathrm{diam}(\cF)$, where $\mathrm{diam}(\cF)$ denotes the diameter of~$\cF$.
Hence, we have 
\[
|w|_i - |w|\, u_i \le \|\bl(w) - |w|\, \bu\|_\infty \le c\, \|N\,(\bl(w) - |w|\, \bu)\|_\infty =  c\, \|N\bl(w) \|_\infty \le c\, \mathrm{diam}(\cF).
\]
Therefore, $\cF_i$~is a bounded remainder set for all $1 \le i \le d$, and the language of~$X$ is $(2\, c\, \mathrm{diam}(\cF))$-balanced.

Minimality of~$R_{\bt}$ implies total irrationality of~$\bt$. 
We will show that this in turn implies that the vector $\bu = (\lambda(\cF_1),\dots,\lambda(\cF_d))$ has rationally independent coordinates. 
Indeed, suppose on the contrary that $\langle \bz, \bu\rangle = 0$ for some $\bz \in \ZZ^d \setminus \{\mathbf{0}\}$. 
Consider the $d\times d$ matrix $\tilde{N}$ that is obtained from~$N$ by subtracting~$\bt$ from each column and adding the row $(1,\dots,1)$ at the bottom. 
Because $\bt_i - \bt \in \ZZ^d$, the matrix~$\tilde{N}$ is an integer matrix. 
Moreover, since $N \bu = 0$ and $\|\bu\|_1=1$, we have $\tilde{N} \bu = \binom{-\bt}{1}$. 
If $\det \tilde{N} \neq 0$, then we have $\tr{\!}\tilde{N} \by = \bz$ for some $\by \in \QQ^d \setminus \{\mathbf{0}\}$; if $\det \tilde{N} = 0$, then we have $\tr{\!}\tilde{N} \by = \mathbf{0}$ for some $\by \in \ZZ^d \setminus \{\mathbf{0}\}$. 
In both cases, we have $0 = \langle \tr{\!}\tilde{N} \by, \bu\rangle = \langle \by,  \binom{-\bt}{1} \rangle$, contradicting the total irrationality of~$\bt$.

Assume now that $X = X_{\bsigma}$ for some sequence of substitutions $\bsigma \in \cS_d^{\NN}$. 
Because $\cF_1,\dots,\cF_d$ are bounded remainder sets with measures $u_1,\ldots,u_d$, $X_{\bsigma}$~has uniform letter frequencies. 
Thus \cite[Theorem~5.7]{Berthe-Delecroix} implies that $\bu = (u_1,\ldots,u_d)$ is the (rationally independent) generalized right eigenvector of~$\bsigma$ normalized by $\|\bu\|_1=1$ (moreover, $\bsigma$ is primitive; see the first part of the proof of Lemma~\ref{l:strongconvergence}). 
Let $\omega^{(0)} \in X_{\bsigma}$ be as in the proof of Lemma~\ref{l:strongconvergence}, and write $\omega^{(0)} = \omega_0 \omega_1 \cdots$ with $\omega_n \in \cA$. 
Then there is some $\bz \in \cF$ such that $R_{\bt}^n(\bz) \in \cF_{\omega_n}$ for all $n \in \NN$.
Define the affine map $H: \RR^d \to \RR^{d-1}$ by $H(\bx) = \bz + N \bx$. 
Then, because $\RR \bu$ is in the kernel of~$N$, we have $H(\bx) = H(\pi_{\bu} \bx)$, in particular $H(\pi_{\bu} \be_i) = \bz + \bt_i$. 
By minimality, we have
\begin{equation}\label{eq:Fi}
\cF_i = \overline{\big\{\bz + N \bl(p) \,:\, p \in \cA^*,\, p\, i \prec \omega^{(0)}\big\}} \subseteq H(\cR_{\bsigma}(i)) \quad \mbox{for all}\ i \in \cA.
\end{equation}
On the other hand, if $p\, i \preceq \sigma_{[0,n)}(j)$ for infinitely many $(n,j) \in \NN \times \cA$, then for all these~$n$ there are words $w_n \in L_{\Sigma^n\bsigma}$ such that $\sigma_{[0,n)}(w_n)\, p\, i \prec \omega^{(0)}$, which implies $H\big(M_{\sigma_{[0,n)}} \bl(w_n) + \bl(p)\big) \in \cF_i$ for infinitely many~$n$ and, by Lemma~\ref{l:strongconvergence}, $H(\bl(p)) \in \cF_i$. Hence, we have $H(\cR_{\bsigma}(i)) \subseteq \cF_i$, thus $H(\cR_{\bsigma}(i)) = \cF_i$.
This means that $(\cF, \tilde{R}_\bt)$ is the domain exchange $H(\cR_{\bsigma}(i)) \mapsto H(\cR_{\bsigma}(i)) + H(\pi_{\bu} \be_i)$. 
Therefore, $(\cF, \tilde{R}_\bt)$ is conjugate to the domain exchange $\cR'_{\bsigma}(i) \mapsto \cR'_{\bsigma}(i) + \be_i' - \bu'$, and $(X_{\bsigma},\Sigma)$ is a natural coding of~$R_{\bu'}$ w.r.t.\ the natural partition $\{-\cR'_{\bsigma}(i) \,:\, 1 \le i \le d\}$, by the same arguments as in the proof of Theorem~\ref{theo:main}.

Assume now that the directive sequence~$\bsigma$ is left proper. 
Then by \cite[Lemma~3.2]{BCDLPP:19} the shift $(X_{\bsigma},\Sigma)$ can be represented as $(X_{\bsigma'},\Sigma)$, where $\bsigma'$ is proper (and still unimodular). 
Like~$\bsigma$, also $\bsigma'$ is primitive; see again the first part of the proof of Lemma~\ref{l:strongconvergence}. 
From \cite[Corollary~5.5]{BCDLPP:19}, we gain that if a primitive unimodular proper $S$-adic shift $(X_{\bsigma},\Sigma)$ is balanced for letters, then it is also balanced for words. 
Hence, cylinders associated with factors are also bounded remainder sets, by Proposition~\ref{prop:balance}.
\end{proof}

\section{Examples}\label{sec:ex}
In this section, we show that our theory can easily be applied to well-known multidimensional continued fraction algorithms, in particular to the Jacobi--Perron, Brun,  (Cassaigne--)Selmer and Arnoux--Rauzy algorithms. 
For dimension $d=3$, corresponding results for the Brun and the Arnoux--Rauzy algorithms were already given in \cite{BST:19},  and for the Cassaigne--Selmer
algorithm in~\cite{Fogg:20}. 
Using our new theory, the conditions we need to check are easier to verify than the ones in \cite{BST:19,Fogg:20}. 
This even allows us to treat the Arnoux--Rauzy algorithm in arbitrary dimension (see Section~\ref{subsec:AR}), the (multiplicative) Jacobi--Perron algorithm ($d=3$) in Section~\ref{subsec:JP} and the Brun algorithm for $d=4$ in Section~\ref{subsec:Brun}.
We start with the Cassaigne--Selmer algorithm ($d=3$) in Section~\ref{sec:exSelmer}, for which we can also prove more general results than the ones in~\cite{Fogg:20}.

Save for the Arnoux--Rauzy algorithm, we focus on algorithms with dimension $d \in \{3,4\}$. 
For higher dimensions, the main problem is to prove the Pisot condition (see Definition~\ref{d:Pisot}). 
Usually, heavy computer calculations are needed to prove that the second Lyapunov exponent of an algorithm is negative; see \cite{berth2021second} for the Selmer algorithm with $d=4$. 
Moreover, somewhat surprising numerical experiments from \cite{berth2021second} indicate that the second Lyapunov exponent is positive for most of the classical continued fraction algorithms if the dimension is beyond a certain threshold. 
In other words, the Pisot condition seems to be violated in these cases, and our results cannot be applied.
For instance, the Brun and Jacobi--Perron algorithms seem to have positive second Lyapunov exponent in dimension $d \ge 10$, contrarily to what was expected e.g.\ in \cite{Lagarias:93,HK00}. 
For the Selmer algorithm, the Pisot condition seems to be violated already in dimension $d \geq 5$.

\subsection{Balanced pair algorithm}\label{subsec:ba}

Before studying individual continued fraction algorithms, we recall an algorithm that can be used to check whether a substitution has purely discrete spectrum. (For each continued fraction algorithm, we have to do this for one substitution associated with a periodic point.) 
The \emph{balanced pair algorithm} was introduced by Livshits \cite{Livshits:87,Livshits:92}  and was inspired by   the notion of coincidence for non-constant length substitutions  such as   considered  for instance in  \cite{Michel:78}; see also \cite[Section~3]{SS:02}, \cite[Section~17]{Barge-Kwapisz:06} or \cite[Section~5.8]{CANTBST}.
This algorithm is usually simpler than checking geometric coincidence.

Let $\sigma$ be a unimodular Pisot substitution.
A~\emph{balanced pair} is a pair $(v_1,v_2) \in \cA^*\times \cA^*$ with $\bl(v_1) = \bl(v_2)$. 
It is called \emph{irreducible} if no proper prefixes of~$v_1$ and~$v_2$ give rise to a balanced pair. 
Each balanced pair can be decomposed into irreducible balanced pairs in an obvious way.
The balanced pair algorithm for a substitution~$\sigma$ on the alphabet $\cA = \{1,\dots,d\}$ starts with $I_0 = \{(ij,ji) \,:\, i,j \in \cA,\, i\ne j\}$. 
Given~$I_k$ for some $k\in \NN$, the set $I_{k+1}$ is defined recursively by the set of all irreducible balanced pairs occurring in a decomposition of a balanced pair $(\sigma(v_1),\sigma(v_2))$ with $(v_1,v_2) \in I_k$. 
We say that the balanced pair algorithm \emph{terminates} if for some $k \in \NN$ the set $I_k \setminus (I_{0} \cup \dots \cup I_{k-1}) = \emptyset$ and if each $(v_1,v_2)\in \bigcup_{j=0}^k I_j$ eventually contains a \emph{coincidence}, i.e., there is a pair of the form $(i,i) \in \cA \times \cA$  that occurs in $(\sigma^j(v_1),\sigma^j(v_2))$ for some $j \in \NN$. 

According to \cite[Theorem~5.8.8]{CANTBST}, the balanced pair algorithm terminates if and only if $\cC_{(\sigma)}$ forms a tiling of~$\bone^\perp$. (More precisely, the theorem states that a certain collection of tiles forms a tiling of~$\bv^\perp$, where $\bv$ is a left eigenvector of~$M_\sigma$, but this is equivalent to $\cC_{(\sigma)}$ being a tiling of~$\bone^\perp$; see the proof of Proposition~\ref{p:gccBST}.)
By Proposition~\ref{p:gccBST} and Lemma~\ref{rem:spectrumEquiv}, this is equivalent to $\sigma$ having purely discrete spectrum.
For a direct proof that a slightly different version of the balanced pair algorithm implies purely discrete spectrum, see \cite[Theorem~3.1]{SS:02}.

\begin{proposition}[{cf.\ \cite[Theorem~5.8.8]{CANTBST}}] \label{p:bpa}
A~unimodular Pisot irreducible substitution $\sigma$ has purely discrete spectrum if and only if the balanced pair algorithm starting with $I_0$ terminates. 
\end{proposition}

\subsection{The Cassaigne--Selmer algorithm} \label{sec:exSelmer}
In 2015, Cassaigne announced a $2$-dimensional continued fraction algorithm that was first studied in \cite{CLL:17}, and in more detail in \cite{CLL:21}. 
This algorithm is called \emph{Cassaigne--Selmer algorithm} because it is measurably conjugate to a semi-sorted version of the $2$-dimensional Selmer algorithm (with the conjugation given by a linear map, see~\cite[Proposition~11.1]{CLL:21}); Selmer's algorithm goes back to~\cite{Selmer:61} (see also \cite[Section~6]{Lagarias:93}) and is conjugate on the absorbing simplex to M\"onkemeyer's algorithm~\cite{Monkemeyer:54} (see \cite{Panti:08}). 
Cassaigne's representation of this algorithm is remarkable  because it admits a set of substitutions that is particularly relevant from a symbolic point of view. 
As shown in \cite{CLL:17}, the $S$-adic dynamical systems defined in terms of these substitutions have factor complexity $2n+1$ (see \eqref{eq:factorcx}) and, as underlined in \cite{BCDLPP:19},  belong to the family of so-called \emph{dendric subshifts}.   Dendric subshifts have   the striking property that  the sets of return words all have the same cardinality for every factor   (they even generate the free group),    which,  among other properties,  provides
 a simple expression for  their dimension group.

Let $\Delta =\{(x_1,x_2,x_3)\in [0,1]^3 : x_1+x_2+x_3=1\}$. 
Using the matrices
\[
C_1 = \begin{pmatrix}1&1&0\\0&0&1\\0&1&0\end{pmatrix} \qquad \mbox{and} \qquad C_2 = \begin{pmatrix}0&1&0\\1&0&0\\0&1&1\end{pmatrix},
\]
we define the matrix valued function
\[
A_\sC:\, \Delta \to \mathrm{GL}(3,\ZZ), \quad \bx \mapsto \begin{cases}\tr{C}_1 & \mbox{if}\ \bx \in \Delta_1 := \{(x_1,x_2,x_3) \in \Delta:\, x_1\ge x_3\}, \\ \tr{C}_2 & \mbox{if}\ \bx \in \Delta_2: = \{(x_1,x_2,x_3) \in \Delta:\, x_1<x_3\}.\end{cases}
\]
Then $T_\sC$ is given by 
\[
T_\sC:\, \Delta \to \Delta, \quad (x_1,x_2,x_3) \mapsto \begin{cases}(\frac{x_1-x_3}{x_1+x_2}, \frac{x_3}{x_1+x_2}, \frac{x_2}{x_1+x_2}) & \mbox{if}\ x_1 \ge x_3, \\ 
(\frac{x_2}{x_2+x_3}, \frac{x_1}{x_2+x_3}, \frac{x_3-x_1}{x_2+x_3}) & \mbox{if}\ x_1 < x_3,\end{cases}
\]
and $(\Delta,T_\sC,A_\sC)$ is called \emph{Cassaigne--Selmer algorithm}.
In \cite[Proposition~22]{AL18}, it is proved that the density of the absolutely continous invariant probability measure~$\nu_\sC$ of~$T_\sC$ equals $\frac{12}{\pi^2(1-x_1)(1-x_3)}$. 
Following~\cite{CLL:17}, we consider the \emph{Cassaigne--Selmer substitutions} 
\begin{equation}\label{eq:cassaigneSubs}
\gamma_1 : \begin{cases} 1 \mapsto 1  \\ 2 \mapsto 13 \\ 3 \mapsto 2 \end{cases} \qquad\qquad
\gamma_2 : \begin{cases} 1 \mapsto 2  \\ 2 \mapsto 13 \\ 3 \mapsto 3 \end{cases} \end{equation}
The corresponding faithful substitution selection is defined by $\varphi_\sC(\bx) = \gamma_j$ if $\bx \in \Delta_j$.
By Definition~\ref{d:realization}, the map
\begin{equation}\label{eq:subrelSel}
\bphi_\sC:\, \Delta \to \{\gamma_1,\gamma_2\}^{\NN}, \quad \bx \mapsto (\varphi_\sC(T^n\bx))_{n\in\NN},
\end{equation}
is a faithful substitutive realization of $(\Delta, T_\sC, A_\sC)$. 
We have $T_\sC(\Delta_1) = T_\sC(\Delta_2) = \Delta$, thus the algorithm satisfies the finite range property and each $\bx \in \Delta$ has  positive range (in the sense of Definition~\ref{def:pr}). 
Moreover, $\bphi_\sC(\Delta) = \{\gamma_1,\gamma_2\}^\NN$ (up to a set of measure zero).
According to \cite[Section~6]{Lagarias:93} and \cite[Chapter~7]{SCHWEIGER}, $T_\sC$~is $\nu_\sC$-almost everywhere weakly convergent, $\{\Delta_1,\Delta_2\}$ is a generating (Markov) partition  for~$T_\sC$, and, hence, one has 
\[
(\Delta,T_\sC,\nu_\sC) \overset{\bphi_\sC}{\cong} (\{\gamma_1,\gamma_2\}^{\NN},\Sigma,\nu_\sC\circ\bphi_\sC^{-1}).
\]
The linear cocycle~$A_\sC$ is $\log$-integrable since  the Cassaigne--Selmer algorithm is additive, with $A_\sC$ taking only $2$ values. 
By \cite[Theorem 1]{Schweiger:04}  and \cite[Theorem 5.1]{berth2021second} (see also \cite[Section 6]{Lagarias:93}), we know that $(\Delta,T_\sC,A_\sC,\nu_\sC)$ satisfies the Pisot condition. 
Moreover, since $\nu_\sC$ is a Borel probability measure which is equivalent to the Lebesgue measure and $T_\sC$ maps open sets to open sets, we have $\nu_\sC \circ T_\sC \ll \nu_\sC$. 

To apply Theorem~\ref{theo:MCF}, we have to find a periodic Pisot point $\bx \in \Delta$ (see Definition~\ref{d:Pisotpoint}) such that $\bphi_\sC(\bx)$ has purely discrete spectrum. To this end, consider 
\begin{equation}\label{eq:stau}
\tau = \gamma_1 \circ \gamma_2:\, \begin{cases} 1 \mapsto 13  \\ 2 \mapsto 12 \\ 3 \mapsto 2 \end{cases}
\end{equation}
and let $\bx \in \Delta$ be the dominant right eigenvector of~$M_\tau$. 
Then we have $\bphi_\sC(\bx) = (\gamma_1,\gamma_2)^\infty$. 
Since $M_\tau$ is a (unimodular) Pisot matrix, we conclude that $\bx$ is a periodic Pisot point, which has positive range by the above considerations.
It only remains to prove that the substitutive dynamical system $(X_\tau,\Sigma)$ has purely discrete spectrum.

For the balanced pair algorithm, we start with $(12,21) \xrightarrow{\tau} (1312,1213)$, which splits into the irreducible pairs $(1,1)$, a coincidence, and $(312,213)$. 
Moreover, $(13,31) \xrightarrow{\tau} (132,213)$ does not split and $(23,32) \xrightarrow{\tau} (122,212)$ splits into $(12,21)$ and the coincidence $(2,2)$. 
Thus 
\[
I_1=\{(1,1),(2,2),(12,21),(312,213),(132,213)\}. 
\]
We have to go on with the new pairs $(1,1), (2,2), (312,213), (132,213)$ occurring in~$I_1$. 
While coincidences yield only coincidences again, we get the pairs $(312,213) \xrightarrow{\tau} (21312,12132)$ and $(132,213) \xrightarrow{\tau} (13212,12132)$. Splitting these yields the new pair $(321,213)$. 
Summing up, the set~$I_2$ contains the new pairs $(3,3)$ and $(321,213)$. 
We only have to check the one which is not a coincidence, getting $(321,213) \xrightarrow{\tau} (21213,12132)$. 
This gives (up to switching the order of the pair) no new pairs in~$I_3$. 
Since all occurring pairs eventually end up in coincidences, the balanced pair algorithm terminates for~$\tau$ and, hence, $(X_\tau,\Sigma)$ has purely discrete spectrum by Proposition~\ref{p:bpa}.

Note that $\tau^2$ and thus the periodic directive sequence $(\gamma_1,\gamma_2)^\infty$ is proper. 
Hence, combining Theorem~\ref{theo:MCF} with Theorem~\ref{t:nc}, we obtain the following result.
(Recall that $\bx' = \pi'(\bx)$ for the projection $\pi'$ defined in~(\ref{e:defpi}); the
corresponding projections of the subtiles, $\cR'_{\bsigma}(w)$, $w\in\cA^*$, are defined in \eqref{eq:Rprime}.)

\begin{theorem}\label{theo:c}
Let $(\Delta,T_\sC,A_\sC,\nu_\sC)$ be the Cassaigne--Selmer algorithm, with substitutive realization~$\bphi_\sC$ defined in \eqref{eq:subrelSel}.
Then $(\Delta,T_\sC,\nu_\sC)\overset{\bphi_\sC}{\cong} (\{\gamma_1,\gamma_2\}^\NN,\Sigma,\nu_\sC\circ\bphi_\sC^{-1})$, and for $\nu_\sC$-almost all $\bx \in \Delta$ the following assertions hold. 
\begin{itemize}
\item[(i)] 
The shift $X_{\bphi_\sC(\bx)}$ is a natural coding of the toral translation $R_{\bx'}$ w.r.t.\ the natural partition $\{-\cR'_{\bsigma}(i) \,:\, i \in \cA\}$.
\item[(ii)] 
The $S$-adic dynamical system $(X_{\bphi_\sC(\bx)},\Sigma) \cong (\TT^2,R_{\bx'})$ has purely discrete spectrum.
\item[(iii)] 
The set $-\cR'_{\bsigma}(w)$ is a bounded remainder set for~$R_{\bx'}$ for each $w \in \cA^*$.
\end{itemize}
\end{theorem}

According to \cite[Theorem~B]{CLL:21}, the system $X_{\bphi_\sC(\bx)}$ has factor complexity $2n+1$ provided that $\bphi_\sC(\bx)$ is primitive. 
Thus Theorem~\ref{theo:c} has the following consequence; see Remark~\ref{rem:allRotations} and the fact that $\nu_\sC$ is equivalent to the Lebesgue measure.

\begin{corollary}\label{cor:CS}
For (Lebesgue) almost all $\bt \in \TT^2$, there exists a minimal subshift $X \subset \{1,2,3\}^{\NN}$ with factor complexity $2n+1$ and language balanced for factors such that $(X,\Sigma)$ is a natural coding of the toral translation~$R_{\bt}$.
\end{corollary}

This result is optimal in the sense that, according to \cite{BB:13}, we cannot reach a smaller factor complexity for a natural coding of a two-dimensional translation.
The asserted balance for words means that all $\cF_{i_0} \cap R_{\bt}^{-1} \cF_{i_1} \cap \cdots \cap R_{\bt}^{-n} \cF_{i_n}$ are bounded remainder sets of~$R_{\bt}$, with the notation of Theorem~\ref{t:nc}.
We mention that  the dimension group of~$X$ can be completely described: 
It is isomorphic to $(\ZZ^3, \{\bx \in \ZZ^3: \langle \bx, \bu\rangle > 0\} \cup \{ \mathbf{0}\}, \bone)$, where $\bu$ stands for the associated generalized right eigenvector which is normalized by $\|\bu\|_1=1$; see \cite{BCDLPP:19} for more on this topic. 
All this extends many properties of Sturmian sequences to sequences on $3$-letter alphabets.
 
The Selmer algorithm also exists in higher dimensions; see e.g.~\cite{BFK:15,BFK:19}. However, to be able to extend the previous results to higher dimensions, two problems  occur: firstly, one has to find a suitable substitutive realization leading to $S$-adic dynamical systems of factor complexity $(d{-}1)n+1$; secondly, as mentioned above, the second Lyapunov exponent seems to be negative only for $d \le 4$ \cite{berth2021second}.

\subsection{The Arnoux--Rauzy algorithm}\label{subsec:AR}
In this section, we apply our results to the Arnoux--Rauzy algorithm in arbitrary dimension $d \ge 3$. Like the Cassaigne--Selmer algorithm (with $d=3$), the Arnoux--Rauzy algorithm generates symbolic dynamical systems that have factor complexity $(d{-}1)n+1$ and belong to the family of dendric subshifts. 

Define the set of \emph{Arnoux--Rauzy substitutions} over the alphabet $\cA = \{1,\dots,d\}$ by
\begin{equation}\label{eq:AR}
\alpha_i:\ i \mapsto i,\ j \mapsto ij\ \mbox{for}\ j \in \cA \setminus \{i\}\qquad (i\in\cA).
\end{equation}
Let 
\[
\Delta_i = \bigg\{(x_1,\dots,x_d) \in [0,1]^d \,:\, x_i \ge \sum_{j\ne i} x_j,\, \sum_{i=1}^d x_i = 1\bigg\}.
\]
Using the transposed incidence matrices of~$\alpha_i$, we define the matrix valued function
\[
A_\sAR:\, \bigcup_{i\in\cA} \Delta_i \to \mathrm{GL}(d,\ZZ), \quad \bx \mapsto \tr{\!M}_{\alpha_i}\quad \mbox{if}\ \bx \in \Delta_i,
\]
which gives that
\[
T_\sAR(x_1,\dots,x_d) = \bigg(\frac{x_1}{x_i},\dots, \frac{x_{i-1}}{x_i}, \frac{x_i-\sum_{j\ne i}x_j}{x_i}, \frac{x_{i+1}}{x_i},\dots, \frac{x_d}{x_i}\bigg) \quad \mbox{if}\ \bx \in \Delta_i.
\]
We have $T_\sAR(\Delta_i) = \{\bx \in [0,1]^d : \|\bx\|_1 = 1\} $ for all $i \in \cA$, thus the image of~$T_\sAR$ need not be contained in $\bigcup_{i\in\cA} \Delta_i$. 
For this reason, we have to restrict the domain of~$T_\sAR$ to the $d$-dimensional \emph{Rauzy simplex}, which is defined by
\[
\Delta_\sAR = \bigg\{\bx \in [0,1]^d \,:\; \|\bx\|_1 = 1\ \mbox{and}\ T_\sAR^n(\bx) \in \bigcup_{i\in\cA} \Delta_i\ \mbox{for all}\ n \in\NN\bigg\}.
\]
The Rauzy simplex is defined in a way that $T_\sAR(\Delta_\sAR) = \Delta_\sAR$. 
The algorithm $(\Delta_\sAR,T_\sAR,A_\sAR)$ is called \emph{Arnoux--Rauzy algorithm} and goes back to \cite{Arnoux-Rauzy:91}.
The Rauzy simplex has zero Lebesgue measure by \cite[Section~7]{Arnoux-Starosta:13}.
We consider $T_\sAR$-invariant probability measures~$\nu$ of $(\Delta_\sAR, T_\sAR)$  satisfying $\nu \circ T \ll \nu$. (The latter condition is satisfied for instance for Borel probability measures~$\nu$ w.r.t.\ the subspace topology on $\Delta_\sAR$; see e.g.\ \cite{AvilaHubSkripbis}.)  
The map~$\varphi_\sAR$ defined by $\varphi_\sAR(\bx) = \alpha_i$ when $\bx \in \Delta_i$ is a faithful substitution selection. 
We have $T_\sAR(\Delta_\sAR \cap \Delta_i) =\Delta_\sAR$, thus the algorithm satisfies the finite range property and each $\bx \in\Delta$ has positive range (in the sense of Definition~\ref{def:pr}). 
The associated substitutive realization
\begin{equation}\label{eq:subrelAR}
\bphi_\sAR:\, \Delta \to \{\alpha_1,\dots,\alpha_d\}^{\NN}, \quad \bx \mapsto (\varphi_\sAR(T^n\bx))_{n\in\NN}
\end{equation}
thus satisfies $\bphi_\sAR(\Delta_\sAR) = \{\alpha_1,\dots,\alpha_d\}^\NN$ (up to a set of measure zero). 

By \cite{AD13}, we know that the second Lyapunov exponent of the fully subtractive algorithm is negative. 
Here, we show the same for the Arnoux--Rauzy algorithm, which is closely related to the fully subtractive algorithm. 

\begin{proposition}
Let $(\Delta_\sAR,T_\sAR,A_\sAR,\nu)$ be the Arnoux--Rauzy algorithm for $d \ge 2$, where $\nu$ is an ergodic invariant probability measure with support~$\Delta_\sAR$. 
Then the Lyapunov exponents satisfy $\vartheta_1(A_\sAR) > 0 > \vartheta_2(A_\sAR)$.
\end{proposition}

\begin{proof}
Since the matrices are unimodular, it suffices to show that $\vartheta_2(A_\sAR) < 0$.
We show that the restriction of $\|A_\sAR^{(n)}(\bx)\|$ to $\bx^\perp$ is exponentially shrinking for a.e.\ $\bx \in \Delta_\sAR$.
Indeed, define a sequence $(\tilde{M}_n(\bx))$ of $\RR^{d\times d}$-matrices as in \cite{Delecroix-Hejda-Steiner}, i.e., if $A_\sAR(T_\sAR^n(\bx)) = \tr{\!M}_{\alpha_i}$, then $\tr{\!\tilde{M}}_n(\bx)$ is given by subtracting $A_\sAR^{(n)}(\bx)\, \bone / \|A_\sAR^{(n)}(\bx)\, \bone\|_\infty$ from the $i$-th column of $\tr{\!M}_{\alpha_i}$.
Then
\[
\tilde{M}_0(\bx) \cdots \tilde{M}_{n-1}(\bx)\, \by = \tr{\!A}_\sAR^{(n)}(\bx)\, \by \quad \mbox{for all}\ \by \in \tr{\!A}_\sAR^{(n)}(\bx)^{-1}\, \bone^\perp,
\]
$\|\tilde{M}_n(\bx)\|_\infty \le 1$ for all $\bx \in \Delta_\sAR$, $n \in \NN$, and
\begin{equation} \label{e:tildeM}
\|\tilde{M}_k(\bx) \cdots \tilde{M}_{\ell-1}(\bx)\|_\infty < \frac{2^h-d}{2^h-1}
\end{equation}
if all matrices $\tilde{M}_k(\bx) \cdots \tilde{M}_{\ell-1}(\bx)$ and $\tilde{M}_{n-h+1}(\bx) \cdots \tilde{M}_n(\bx)$,  with $k \le n < \ell$, are primitive; see \cite[Lemma~6]{Delecroix-Hejda-Steiner}.
For $\bv \in \bx^\perp$ and $i \in \cA$, we have
\begin{align*}
\langle \be_i, A_\sAR^{(n)}(\bx)\, \bv \rangle & = \langle \tr{\!A}_\sAR^{(n)}(\bx)\, \be_i, \bv \rangle = \langle \pi_\bx \tr{\!A}_\sAR^{(n)}(\bx)\, \be_i, \bv \rangle = \langle \tr{\!A}_\sAR^{(n)}(\bx)\, \pi^{(n)}_\bx \be_i, \bv \rangle \\
& = \langle \tilde{M}_0(\bx) \cdots \tilde{M}_{n-1}(\bx)\, \pi^{(n)}_\bx \be_i, \bv \rangle, 
\end{align*}
where $\pi^{(n)}_\bx$ denotes the projection along $T_\sAR^n(\bx)$ on $\tr{\!A}_\sAR^{(n)}(\bx)^{-1}\, \bone^\perp$.

Let $\tilde{T}$ be the induced map of $T_\sAR$ on the cylinder $\tilde{\Delta} = \bphi_\sAR([\alpha_2,\dots,\alpha_d,\alpha_1,\alpha_2,\dots,\alpha_d])$, and $\tilde{A}$ the induced cocyle  (so that we can apply \eqref{e:tildeM} with $h=d=\ell-k$).
Then there exists $c > 0$ such that $\|\pi_\bx^{(n)} \be_i\|_\infty \le c$ for all $\bx \in \Delta_\sAR$ with $T_\sAR^n(\bx) \in \tilde{\Delta}$, thus 
\[
\|A_\sAR^{(n)}(\bx)\, \bv\|_\infty \le c\, d\, \|\tilde{M}_0(\bx) \cdots \tilde{M}_{n-1}(\bx)\|_\infty \|\bv\|_\infty.
\]
We have thus 
\[
\|\tilde{A}^{(n)}(\bx)\, \bv\|_\infty \le c\, d\, \bigg(\frac{2^d-d}{2^d-1}\bigg)^n \|\bv\|_\infty \quad \mbox{for all}\ \bx \in \tilde{\Delta},\, \bv \in \bx^\perp. 
\]
which implies that the second Lyapunov exponent of $\tilde{A}$ and thus of $A$ is negative; see e.g.  \cite[Section 4.4.1]{Viana}.
\end{proof}

By induction on~$d$, we can show that 
\[
\alpha_1 \circ \alpha_2 \circ \cdots \circ \alpha_d = \tilde{\alpha}^d, \quad \mbox{with} \quad \tilde{\alpha}(i) = 1(i\!+\!1)\ \mbox{for}\ 1 \le i < d,\ \tilde{\alpha}(d) = 1.
\]
The substitution $\tilde{\alpha}$ is the $d$-bonacci substitution; the characteristic polynomial of the incidence matrix $M_{\tilde{\alpha}}$ of~$\tilde{\alpha}$ is $x^d-x^{d-1}-\cdots-x-1$. 
Thus the dominant right eigenvector $\bx \in \Delta_\sAR$ of $M_{\tilde{\alpha}}$ is a periodic Pisot point. 
It has, like all points of $\Delta_\sAR$, positive range. 
It is well known that $\tilde{\alpha}$ has purely discrete spectrum; see e.g.\ \cite[Theorem~1.2 and Example~3.1]{Ito-Rao:06}, which is based on \cite{Frougny-Solomyak:92}, or \cite[Corollary~4.3]{Barge:16}. (It is also not difficult to show that the balanced pair algorithm terminates for~$\tilde{\alpha}$.) Moreover, $\tilde{\alpha}$~is clearly left proper. 
Thus, combining again Theorem~\ref{theo:MCF} with Theorem~\ref{t:nc} and using the results on factor complexity from \cite[p.~209]{Arnoux-Rauzy:91} and \cite[Theorem~III.8]{RZ:00}, we obtain the following result (parts of which were proved for $d=3$ in \cite{BST:19}). 

\begin{theorem}\label{theo:ar}
Let $(\Delta_\sAR,T_\sAR,A_\sAR,\nu)$ be the Arnoux--Rauzy algorithm for $d \ge 2$, where $\nu$ is an ergodic invariant probability measure with support~$\Delta_\sAR$, and let $\bphi_\sAR$ be as in \eqref{eq:subrelAR}.
Then we have $(\Delta_\sAR,T_\sAR,\nu)\overset{\bphi_\sAR}{\cong} (\{\alpha_1,\dots,\alpha_d\}^\NN,\Sigma,\nu\circ\bphi_\sAR^{-1})$, and for $\nu$-almost all $\bx\in \Delta_\sAR$ the following assertions hold. 
\begin{itemize}
\item[(i)] 
The shift~$X_{\bphi_\sAR(\bx)}$ is a natural coding of the toral translation~$R_{\bx'}$ w.r.t.\ the natural partition $\{-\cR'_{\bsigma}(i) \,:\, i \in \cA\}$.
\item[(ii)] 
The $S$-adic dynamical system $(X_{\bphi_\sAR(\bx)},\Sigma) \cong (\TT^{d-1},R_{\bx'})$ has purely discrete spectrum.
\item[(iii)] 
The set $- \cR'_{\bsigma}(w)$ is a bounded remainder set for~$R_{\bx'}$ for each $w \in \cA^*$.
\item[(iv)] 
The shift $X_{\bphi_\sAR(\bx)}$ has factor complexity  $(d-1)n+1$ and is balanced for words. 
\end{itemize}
\end{theorem}

Note that  Arnoux--Rauzy shifts are also dendric and their dimension group has a similar description as the one given in the previous section for the Cassaigne--Selmer shifts; see \cite{BCDLPP:19}.

\subsection{The Jacobi--Perron algorithm}\label{subsec:JP}
One of the most famous multidimensional continued fraction algorithms is the Jacobi--Perron algorithm; see e.g. \cite[Section~2]{Lagarias:93} or ~\cite[Chapter~4]{SCHWEIGER}. 
We want to apply our theory to the case $d = 3$. 
In this case, the Jacobi--Perron algorithm is defined on the set $\Delta  = \{(x_1,x_2,x_3) \in \RR_+^3 \,:\, x_1+x_2+x_3 = 1,\, x_1 \le x_3,\, x_2 \le x_3\}$.
Let $\cL = \{(a,b) \in \NN^2 \,:\, 0 \le a \le b \neq 0\}$, and for $(a,b) \in \cL$ define the matrices 
\[
J_{a,b}= \begin{pmatrix}0&1&0\\0&0&1 \\1&a&b\end{pmatrix}
\]
and the sets $\Delta_{a,b} = \{(x_1,x_2,x_3) \in \Delta \,:\, a x_1 \le x_2 < (a+1) x_1\ \mbox{and}\ b x_1 \le x_3 < (b+1)x_1\}$.
Then $\mathcal{U}_{\sJP} = \{\Delta_{a,b} \,:\, (a,b) \in \cL\}$ forms a partition of~$\Delta$. We can thus define the matrix valued function
\[
A_{\sJP} :\, \Delta \to \mathrm{GL}(3,\ZZ), \quad \bx \mapsto J_{a,b} \quad \mbox{if}\ \bx \in \Delta_{a,b}.
\]
This function is used to define the piecewise linear function~$T_\sJP$ according to \eqref{eq:Tdef}, which yields
\[
T_\sJP(x_1,x_2,x_3) = \bigg(\frac{x_2-ax_1}{1-(a+b)x_1}, \frac{x_3-bx_1}{1-(a+b)x_1}, \frac{x_1}{1-(a+b)x_1}\bigg) \quad \mbox{if}\ \bx \in \Delta(a,b).
\]
The  algorithm $(\Delta,T_{\sJP},A_{\sJP})$ is called ($2$-dimensional) \emph{Jacobi--Perron algorithm} and goes back to Jacobi's posthumously published work~\cite{Heine1868}. 
Note that, contrary to the Cassaigne--Selmer algorithm, this algorithm is multiplicative (its linear cocycle~$A_{\sJP}$  produces infinitely many different matrices). 
It is known from \cite{SchweigerJP:90} that the invariant measure $\nu_{\sJP}$ of $T_{\sJP}$ is equivalent to the Lebesgue measure on~$\Delta$ and, hence, has full support and 
satisfies $\nu_{\sJP}\circ T\ll \nu_{\sJP}$.
However, there is no known simple expression for the density of~$\nu_{\sJP}$; for more on this subject, see \cite{Broise:96}. 
A~cylinder
\begin{align*}
\Delta_{(a_0,b_0),\dots, (a_{n-1},b_{n-1})} 
& = \big\{\bx\in \Delta \,:\, (A_\sJP(T^0\bx),\dots,A_\sJP(T^{n-1}\bx)) = (J_{a_0,b_0},\dots,J_{a_{n-1},b_{n-1}})\big\} \\
& = \bigcap_{k=0}^{n-1} T_\sJP^{-k}(\Delta_{a_k,b_k})
\end{align*}
is nonempty if and only if the pairs $(a_0,b_0),\dots,(a_{n-1},b_{n-1}) \in \cL$ satisfy the \emph{admissibility condition} 
\begin{equation}\label{eq:admJP}
0 \leq a_k \leq b_k \neq 0, \quad \mbox{and if $a_k = b_k$ then $a_{k+1}=0$}
\end{equation}
for all $0 \le k < n$; see \cite[Section~4.1]{SCHWEIGER}.
This implies that the Jacobi--Perron algorithm satisfies the finite range property. 
In other words, this admissibility condition is a sofic condition that can be recognized by a finite graph.
It is proved in \cite[p.~322]{Lagarias:93} that the cocycle~$A_\sJP$ is $\log$-integrable (which is nontrivial in this case because~$A_\sJP$ has infinite range). 
Thus, because $\nu_\sJP$ has full support, each $\bx \in \Delta$ has positive range. 
The fact that $A_\sJP$ satisfies the Pisot condition is proved in \cite[Chapter~16]{SCHWEIGER}.
Following~\cite{Berthe:16}, we define the \emph{Jacobi--Perron substitutions} 
\begin{equation}\label{eq:JPSubs}
\iota_{a,b} : \begin{cases} 1 \mapsto 2  \\ 2 \mapsto 3 \\ 3 \mapsto 12^a3^b \end{cases} \qquad \big((a,b) \in \cL\big).
\end{equation}
Then $\tr{\!}J_{a,b}$ is the incidence matrix of~$\iota_{a,b}$ for each pair $(a,b) \in \cL$. 
Define the substitution selection~$\varphi_\sJP$ on~$\Delta$ by setting $\varphi_\sJP(\bx) = \iota_{a,b}$ if $\bx \in \Delta_{a,b}$.
The associated faithful substitutive realization~$\bphi_\sJP$ yields 
\[
(\Delta_\sJP,T_\sJP,\nu_\sJP) \overset{\bphi_\sJP}{\cong} (D_\sJP,\Sigma,\nu_\sJP\circ\bphi_\sJP^{-1}),
\]
where $D_\sJP$ is the set of all directive sequences $(\iota_{a_k,b_k})$ satisfying the admissibility condition~\eqref{eq:admJP} for all $k \in \NN$.
This isomorphy is due to the fact that the set $\{\Delta_{a,b} : (a,b)\in  \cL\}$ is a generating (Markov) partition for~$T_\sJP$, which yields weak convergence; see \cite[Section~5]{Lagarias:93}.

To apply Theorem~\ref{theo:MCF}, it remains to establish  that there exists a periodic Pisot point $\bx \in \Delta$ for which $\bphi_\sJP(\bx)$ has purely discrete spectrum.
This assertion is easily checked.
Indeed, $\iota_{0,1}$ is a unimodular Pisot substitution (see also \cite{DFP:04} for relations between the Jacobi--Perron algorithm and Pisot numbers) and $(\iota_{0,1})^\infty \in D_\sJP$ is admissible. 
Moreover, using for instance the balanced pair algorithm (as we did in Section~\ref{sec:exSelmer} for another substitution), one easily checks that $\iota_{0,1}$ has purely discrete spectrum.  This implies that the right eigenvector $\bx \in \Delta$ of the incidence matrix of~$\sigma$ is a periodic Pisot point with $\bphi_\sJP(\bx)$ having purely discrete spectrum.
Thus, all the conditions of Theorem~\ref{theo:MCF} are satisfied and, because of right properness of all directive sequences, we arrive together with Theorem~\ref{t:nc} at the following result. 
%Recall again that $\bx' = \pi'(\bx)$ for the projection~$\pi'$ defined in~(\ref{e:defpi}). 
%The corresponding projections of the subtiles, $\cR'_{\bsigma}(w)$, $w \in \cA^*$, are defined in \eqref{eq:Rprime}.

\begin{theorem}\label{theo:jp}
Let $(\Delta,T_\sJP,A_\sJP,\nu_\sJP)$ be the $2$-dimensional Jacobi--Perron algorithm, and let $\bphi_\sJP$ be as above.
Then we have $(\Delta,T_\sJP,\nu_\sJP) \overset{\bphi_\sJP}{\cong} (D_\sJP,\Sigma,\nu_\sJP\circ\bphi_\sJP^{-1})$, and for $\nu_\sJP$-a.e.\ $\bx \in \Delta$ the following assertions hold. 
\begin{itemize}
\item[(i)] 
The shift $X_{\bphi_\sJP(\bx)}$ is a natural coding of the toral translation~$R_{\bx'}$ w.r.t.\ the natural partition $\{-\cR'_{\bsigma}(i) \,:\, i \in \cA\}$.
\item[(ii)] 
The $S$-adic dynamical system $(X_{\bphi_\sJP(\bx)},\Sigma) \cong (\TT^2,R_{\bx'})$ has purely discrete spectrum.
\item[(iii)] 
The set $-\cR'_{\bsigma}(w)$ is a bounded remainder set for~$R_{\bx'}$ for each $w \in \cA^*$.
\item[(iv)] 
The shift~$X_{\bphi_\sJP(\bx)}$ is balanced for words.
\end{itemize}
\end{theorem}

\subsection{The Brun algorithm}\label{subsec:Brun}
The case $d=3$ of the Brun algorithm is treated in \cite{BST:19}. 
Here, we consider the unordered version of the Brun algorithm, as defined  in \cite{Delecroix-Hejda-Steiner}, with special emphasis on the case $d=4$. 
We start with the definition of the algorithm for arbitrary $d \ge 3$.  
For this algorithm, we have $\Delta = \{\bx \in [0,1]^d \,:\, \|\bx\|_1 = 1\}$, and the set of \emph{Brun substitutions} over~$\cA$ is defined by
\begin{equation}\label{eq:subsBrun}
\beta_{i,j}:\, j \mapsto ij,\ k\mapsto k\ \mbox{for}\ k \in \cA \setminus \{j\}.
\end{equation}
(We emphasize that in \cite{Berthe-Fernique:11} the authors deal with other substitutions related to this algorithm.)
Let 
\[
\Delta_{i,j} = \big\{(x_1,\dots,x_d) \in \Delta \,:\, x_i \ge x_j \ge x_k\ \mbox{for all}\ k \in \cA \setminus \{i,j\}\big\}.
\]
Using the transposed incidence matrices of~$\beta_{i,j}$, we define the matrix valued function
\[
A_\sB:\, \Delta \to \mathrm{GL}(d,\ZZ), \quad \bx \mapsto \tr{\!M}_{\beta_{i,j}}\quad \mbox{if}\ \bx \in \Delta_{i,j},
\]
which yields
\[
T_\sB(x_1,\dots,x_d) = \bigg(\frac{x_1}{1-x_j}, \dots, \frac{x_{i-1}}{1-x_j}, \frac{x_i-x_j}{1-x_j}, \frac{x_{i+1}}{1-x_j} , \dots, \frac{x_d}{1-x_j}\bigg) \quad \mbox{if}\ \bx \in \Delta_{i,j}.
\]
The algorithm $(\Delta,T_\sB,A_\sB)$ is called (unordered) \emph{Brun algorithm}. 
It goes back to~\cite{Brun19,Brun20,BRUN}. 
The faithful substitution selection corresponding to the substitutions in \eqref{eq:subsBrun} is defined by $\varphi_\sB(\bx) = \beta_{i,j}$ if $\bx \in \Delta_{i,j}$.
As indicated in \cite{Delecroix-Hejda-Steiner}, the directive sequences $\bsigma = (\sigma_n)$ that are generated by this algorithm are characterized by the \emph{admissibility condition}
\begin{equation} \label{eq:Brun:admis}
\begin{split}
(\sigma_n,\sigma_{n+1}) \in \big\{(\beta_{i,j},\beta_{i,j}) \,: & \ i \in \cA,\, j \in \cA \setminus \{i\}\big\} \\ 
\cup\, \big\{(\beta_{i,j},\beta_{j,k}) \,: & \ i \in \cA,\, j \in \cA \setminus \{i\},\, k \in \cA \setminus \{j\}\big\}
\end{split}
\quad \mbox{for all}\ n \in \NN.
\end{equation}
This is again a sofic condition that can be recognized by a finite graph. 
For the faithful substitutive realization~$\bphi_\sB$ associated with~$\varphi_\sB$, we have thus $\bphi_\sB(\Delta) = D_\sB$ for a sofic shift~$D_\sB$, and the algorithm satisfies the finite range property. 
The Brun algorithm has an ergodic  invariant  probability measure~$\nu_\sB$ that is equivalent to the Lebesgue measure; see e.g.~\cite[Proposition~28]{AL18}.
Then each $\bx\in\Delta$ has positive range. 
Moreover, as $T_\sB$ maps open sets to open sets, we have $\nu_\sB\circ T\ll\nu_\sB$. 

We now confine ourselves to the case $d=4$. 
The linear cocycle~$A_\sB$ is $\log$-integrable since $A_\sB$ takes only $12$ values. 
By Schratzberger \cite{Schratzberger:01}, we know that $(\Delta,T_\sB,A_\sB,\nu_\sB)$ satisfies the Pisot condition; see also \cite{HK00,Hardcastle:02}, where an acceleration of Brun's algorithm is considered. This implies that $\{\Delta_{i,j} \,:\, i\neq j\}$ is a generating  partition for~$T_\sB$ and that $T_\sB$ is weakly convergent, hence,
$(\Delta,T_\sB,\nu_\sB) \overset{\bphi_\sB}{\cong} (D_\sB,\Sigma,\nu_\sB\circ\bphi_\sB^{-1})$.

To apply Theorem~\ref{theo:MCF}, we have to find a periodic Pisot point $\bx \in \Delta$ such that $\bphi_\sB(\bx)$ has purely discrete spectrum. To this end, consider 
\begin{equation*}
\tau = \beta_{12}\circ\beta_{23}\circ\beta_{34}\circ\beta_{41}:\, \begin{cases}1 \mapsto 12341 \\ 2 \mapsto 12 \\ 3 \mapsto 123 \\ 4 \mapsto 1234\end{cases}
\end{equation*}
and let $\bx \in \Delta$ be the dominant right eigenvector of~$M_\tau$. 
Then $\bphi_\sB(\bx) = (\beta_{12},\beta_{23},\beta_{34},\beta_{41})^\infty \in D_\sB$ is an admissible sequence. 
Since $M_\tau$ is a Pisot matrix, we conclude that $\bx$ is a periodic Pisot point, which  has positive range by the above considerations. 
Again using the balanced pair algorithm, one can show that $\tau$ has purely discrete spectrum.
Since $\tau$ is left proper, $\bphi_\sB(\bx)$ is also left proper for $\nu_\sB$-a.e.\ $\bx \in \Delta$. 
Combining Theorem~\ref{theo:MCF} with Theorem~\ref{t:nc}, we thus obtain the following result. 

\begin{theorem}\label{theo:b}
Let $(\Delta,T_\sB,A_\sB,\nu_\sB)$ be the Brun algorithm with $d=4$, and let $\bphi_\sB$ be as above. %the substitutive realization defined in \eqref{eq:subrelBrun}.
Then $(\Delta,T_\sB,\nu_\sB)\overset{\bphi_\sB}{\cong} (D_\sB,\Sigma,\nu_\sB\circ\bphi_\sB^{-1})$, and for $\nu_\sB$-almost all $\bx\in \Delta$, the following assertions hold. 
\begin{itemize}
\item[(i)] 
The shift~$X_{\bphi_\sB(\bx)}$ is a natural coding of the toral translation~$R_{\bx'}$ w.r.t.\ the natural partition $\{-\cR'_{\bsigma}(i)\,:\, i\in\cA\}$.
\item[(ii)] 
The $S$-adic dynamical system $(X_{\bphi_\sB(\bx)},\Sigma) \cong (\TT^3,R_{\bx'})$ has purely discrete spectrum.
\item[(iii)] 
The set $-\cR'_{\bsigma}(w)$ is a bounded remainder set for~$R_{\bx'}$ for each $w \in \cA^*$. 
\item[(iv)] 
The shift~$X_{\bphi_\sB(\bx)}$  is balanced for words.
\end{itemize}
\end{theorem}

Note that  this result gives a natural coding for (Lebesgue) a.a.\ points of~$\TT^3$ in terms of ``Brun $S$-adic sequences'' by Remark~\ref{rem:allRotations}, and by recalling that the ergodic invariant measure~$\nu_\sB$ of the Brun algorithm is equivalent to Lebesgue measure.

\begin{corollary}\label{cor:Br}
For (Lebesgue) almost all $\bt \in \TT^3$, there exists a minimal subshift $X \subset \{1,2,3,4\}^{\NN}$ with language balanced for factors such that $(X,\Sigma)$ is a natural coding of the toral translation~$R_{\bt}$.
\end{corollary}

\section{Acknowledgements}
We are indebted to the two anonymous referees whose comments helped us to improve the exposition of the present paper.

\bibliographystyle{amsalpha}
\bibliography{sadic5bis}
\end{document}